\documentclass[11pt]{amsart}

\begin{document}
\numberwithin{equation}{section}

\def\1#1{\overline{#1}}
\def\2#1{\widetilde{#1}}
\def\3#1{\widehat{#1}}
\def\4#1{\mathbb{#1}}
\def\5#1{\frak{#1}}
\def\6#1{{\mathcal{#1}}}
\def\7#1{{\bf{#1}}}

\def\C{{\4C}}
\def\R{{\4R}}
\def\N{{\4N}}
\def\Z{{\4Z}}
\def\P{{\4P}}
\def\Q{{\4Q}}
\def\D{{\4D}}
\def\uk#1{u_k\left(#1 \right)}
\def\ukk#1{u_{k-1}\left( #1 \right)}
\def\oo#1{O\left( #1 \right)}
\def\[{{[\![}}
\def\]{{]\!]}}
\newcommand{\norm}[1]{\lVert #1 \rVert}
\newcommand{\phd}{\Pi_r^i}
\newcommand{\modu}[1]{\left | #1 \right |}
\def\T{{\rm{T}}}
\newcommand{\Et}{E_{\T}^i(r, C_0, C_1)}
\newcommand{\normu}[1]{\left\lVert #1 \right\rVert}
\newcommand{\mlog}[1]{\modu{\log\modu{#1}}}
\newcommand{\xp}{x^{\prime}}
\newcommand{\xs}{x^{\prime\prime}}
\newcommand{\x}[1]{x^{\prime}_{#1}}

\title[On \'Ecalle-Hakim's theorems in holomorphic dynamics]{On \'Ecalle-Hakim's theorems in holomorphic dynamics}
\author[M. Arizzi]{Marco Arizzi}
\address{M. Arizzi} \email{marco.arizzi@gmail.com}
\author[J. Raissy]{Jasmin Raissy*}
\address{J. Raissy: Dipartimento Di Matematica e Applicazioni, Universit\`{a} degli Studi di Milano Bicocca, Via Roberto Cozzi 53,
20125, Milano, Italy. } \email{jasmin.raissy@unimib.it}
\thanks{$^{*}$Supported in part by FSE, Regione Lombardia.}


\def\Label#1{\label{#1}{\bf (#1)}~}


\def\cn{{\C^n}}
\def\cnn{{\C^{n'}}}
\def\ocn{\2{\C^n}}
\def\ocnn{\2{\C^{n'}}}


\let\no=\noindent
\let\bi=\bigskip
\let\me=\medskip
\let\sm=\smallskip
\let\ce=\centerline
\let\ri=\rightline
\let\te=\textstyle
\let\gd=\goodbreak
\let\io=\infty
\def\qqquad{\quad\qquad}

\def\dist{{\rm dist}}
\def\const{{\rm const}}
\def\rk{{\rm rank\,}}
\def\id{{\sf id}}
\def\aut{{\sf aut}}
\def\Aut{{\sf Aut}}
\def\CR{{\rm CR}}
\def\GL{{\sf GL}}
\def\Re{{\sf Re}\,}
\def\Im{{\sf Im}\,}
\def\span{\text{\rm span}}
\def\res{{\rm Res}\,}

\def\codim{{\rm codim}}
\def\crd{\dim_{{\rm CR}}}
\def\crc{{\rm codim_{CR}}}

\def\phe{\varphi}
\def\eps{\varepsilon}
\def\d{\partial}
\def\a{\alpha}
\def\b{\beta}
\def\g{\gamma}
\def\G{\Gamma}
\def\Om{\Omega}
\def\k{\kappa}
\def\l{\lambda}
\def\L{\Lambda}
\def\z{{\bar z}}
\def\w{{\bar w}}
\def\t{\tau}
\def\th{\theta}
\def\ta{\tilde{\alpha}}
\def\epsilon{{\varepsilon}}

\def\sideremark#1{\ifvmode\leavevmode\fi\vadjust{
\vbox to0pt{\hbox to 0pt{\hskip\hsize\hskip1em
\vbox{\hsize1.5cm\tiny\raggedright\pretolerance10000
\noindent #1\hfill}\hss}\vbox to8pt{\vfil}\vss}}}

\def\Dif{{\sf Diff}(\C^n;0)}
\def\Diff{{\sf Diff}}

\emergencystretch15pt \frenchspacing

\newtheorem{theorem}{Theorem}[section]
\newtheorem{lemma}[theorem]{Lemma}
\newtheorem{proposition}[theorem]{Proposition}
\newtheorem{corollary}[theorem]{Corollary}

\theoremstyle{definition}
\newtheorem{definition}[theorem]{Definition}
\newtheorem{example}[theorem]{Example}

\theoremstyle{remark}
\newtheorem{remark}[theorem]{Remark}
\numberwithin{equation}{section}

\def\bl{\begin{lemma}}
\def\el{\end{lemma}}
\def\br{\begin{remark}}
\def\er{\end{remark}}
\def\bp{\begin{Pro}}
\def\ep{\end{Pro}}
\def\bt{\begin{Thm}}
\def\et{\end{Thm}}
\def\bc{\begin{Cor}}
\def\ec{\end{Cor}}
\def\bd{\begin{Def}}
\def\ed{\end{Def}}
\def\be{\begin{Exa}}
\def\ee{\end{Exa}}
\def\bpf{\begin{proof}}
\def\epf{\end{proof}}
\def\ben{\begin{enumerate}}
\def\een{\end{enumerate}}
\def\beq{\begin{equation}}
\def\eeq{\end{equation}}

\begin{abstract}
In this survey we provide detailed proofs for the results by Hakim regarding the dynamics of germs of biholomorphisms tangent to the identity of order $k+1\ge 2$ and fixing the origin. 
\end{abstract}

\maketitle

\section{Introduction}

One of the main questions in the study of local discrete holomorphic dynamics, i.e., in the study of the iterates of a germ of a holomorphic map of~$\C^p$ at a fixed point, which can be assumed to be the origin, is when it is possible to holomorphically conjugate it to a ``simple'' form, possibly its linear term. It turns out (see \cite{Ab1}, \cite{Abate}, \cite{Bracci} and Chapter 1 of \cite{Raissy} for general surveys on this topic) that the answer to this question strictly depends on the arithmetical properties of the eigenvalues of the linear term of the germ. 

This conjugacy approach is not that useful in the so-called {\it tangent to the identity case} that is when the linear part of the germ coincides with the identity, but the germ does not coincide with the identity. Nevertheless, it is possible to study the dynamics of such germs, which is indeed very interesting and rich. The one-dimensional case, was first studied by Leau \cite{Leau} and Fatou \cite{Fa} who provided a complete description of the dynamics in a pointed neighbourhood of the origin. More precisely, in dimension $1$, a tangent to the identity germ can be written as 
\begin{equation}\label{h}
f(z) := z + a z^{k+1} + O(z^{k+2}),
\end{equation}
where the number $k+1\ge 2$ is usually called the {\it order} of $f$.
We define the {\sl attracting directions} $\{v_1,\ldots, v_k\}$ for $f$ as the $k$-th roots of $-\frac{|a|}{a}$, and these are precisely the directions $v$ such that the term $av^{k+1}$ points in the direction opposite to $v$. An {\sl attracting petal} $P$ for $f$ is a simply-connected domain such that $0\in\partial P$, $f(P)\subseteq P$ and $\lim_{n\to\infty}f^{n}(z)=0$ for all $z\in P$, where $f^n$ denotes the $n$-th iterate of $f$. The attracting directions for $f^{-1}$
are called {\sl repelling directions} for $f$ and the
attracting petals for $f^{-1}$ are {\sl repelling petals} for
$h$.
Then the Leau-Fatou flower theorem is the following result (see, {\sl e.g.},
\cite{Abate}, \cite{Bracci}). We write $a\sim b$ whenever there exist
constants $0<c<C$ such that $ca\le b\le Ca$.

\begin{theorem}[Leau-Fatou, \cite{Leau, Fa}]\label{LF}
Let $f$ be as in \eqref{h}. Then for each attracting
direction $v$ of $h$ there exists an attracting petal $P$ for
$f$ (said {\sl centered at $v$}) such that for each $z\in P$
the following hold:
\begin{enumerate}
  \item $f^n(z)\ne0$ for all $n$ and $\lim_{n\to\infty}\frac{f^n(z)}{|f^{n}(z)|}=v$,
  \item $|f^n(z)|^{k_0}\sim \frac{1}{n}$.
\end{enumerate}
Moreover,  the union of all $k$ attracting petals and $k$ repelling petals for $f$ forms a punctured open neighborhood of $0$.
\end{theorem}

By the property (1), attracting [resp. repelling] petals centered at different attracting [resp. repelling] directions must be disjoint.

\sm
For dimension $p\ge 2$ the situation is more complicated and a general complete description of the dynamics in a full neighborhood of the origin is still unknown (see \cite{AT} for some interesting partial results). Analogously to the one-dimensional case, we can write our germ as sum of homogeneous polynomials
$$
F(z) = z + P_{k+1}(z) + O(\|z\|^{k+2}),
$$
where $k+1\ge 2$ is the {\it order} of $F$. Very roughly, \'Ecalle using his resurgence theory \cite{Ec}, and Hakim with classical tools proved that generically, given a tangent to the identity germ of order $k+1$, it is possible to find one-dimensional ``petals'', called {\it parabolic curves}, that is one-dimensional $F$-invariant analytic discs having the origin on the boundary and where the dynamics is of {\it parabolic type}, i.e., the orbits converge to the origin tangentially to a particular direction. Abate, in \cite{Ab}, then proved that in dimension $2$ such parabolic curves always exist. Hakim also gave sufficient conditions for the existence of basins of attraction modeled on such parabolic curves, proving the following result (see Section 3 for the detailed definitions):

\begin{theorem}[Hakim, \cite{Ha}]\label{Hakim_intro}
Let $F$ be a tangent to the identity germ fixing the origin of order $k+1\ge 2$, and let $[v]$ be a non-degenerate characteristic direction. 
If $[v]$ is attracting, then there exist $k$ parabolic invariant domains, where each point is attracted by the origin along a trajectory tangential to $[v]$.  
\end{theorem}

Hakim's techniques have been largely used in the study of the existence of parabolic curves (see \cite{Ab}, \cite{B-M}, \cite{Molino}, \cite{R1}), basins of attraction and Fatou-Bieberbach domains, i.e., proper open subset of $\C^p$ biholomorphic to $\C^p$, (see \cite{BRZ}, \cite{Ri}, \cite{R2}, \cite{Liz}). 

\me The aim of this survey is to make available important results and very useful techniques, that were included, up to now, only in \cite{hakim2}, a preprint which is not easily retrievable, and where the case $k>1$ was stated with no detailed proofs.

We shall provide, from Section 3 up to Section 7, the reformulations for any order $k+1\ge 2$, with detailed proofs, of the results published by Hakim in \cite{Ha} (Hakim gave detailed proofs of her results for $k=1$ only), and, in the last three sections, again reformulating definitions, lemmas, propositions and theorems for any order $k+1\ge 2$, we shall provide detailed proofs for the unpublished results, including her construction of Fatou-Bieberbach domains, obtained by Hakim in \cite{hakim2}. 

\me\no{\sl Acknowledgments.} We would like to thank the anonymous referee for useful comments and remarks which improved the presentation of the paper.

\section{Notation}

In the following we shall work in $\C^p$, $p\ge 2$ with the usual Euclidean norm
$$
\| z \|=\left( \sum_{i=1}^{p} |{z_i}|^2 \right)^{\frac{1}{2}}.
$$

We shall denote by $\D_{r,k}$ the following subset of $\C$
$$
\D_{r,k}=\left\{ z\in \C \mid  \lvert z^k-r \rvert <r \right\},
$$
which has exactly $k$ connected components, that will be denoted by $\Pi_{r,k}^1,\ldots,\Pi_{r,k}^k$.

Let $F\colon\C^p \to \C^p$ be a holomorphic map. We shall denote with $F^\prime (z_0)$ the Jacobian matrix of $F$ in $z_0$. If, moreover, we write $\C^p=\C^s\times \C^t$, then $\frac{\partial F}{\partial x}$ and $\frac{\partial F}{\partial y}$ will be the Jacobian matrices of $F(\cdot,y)$ and $F(x,\cdot)$.

Given $f,g_1, \ldots, g_s\colon\C^m \to \C^k$, we shall write 
$$ 
f=\oo{g_1, \ldots, g_s},
$$
if there exist $ C_1,\ldots,C_s >0 $ so that
$$ 
\|{f(w)}\|\leq C_1\|{g_1(w)} \|+\cdots+ C_s\|{g_s(w)}\|;
$$
and moreover, with $f=o(g)$ we mean
$$ 
\frac{\|{f(w)}\|}{\|{g(w)}\|}\to 0 \textrm{ as } w\to 0 . 
$$

Similarly, given a sequence $w_n \in \C^p$, we shall write
\begin{equation*}
\begin{split}
&w=\oo{\frac{1}{n}}\Longleftrightarrow \exists C >0\,:\, |{w_n}|\leq \frac{C}{n};\\
&w=o\left( \frac{1}{n}\right) \Longleftrightarrow \frac{w_n}{1/n}\to 0 \textrm{ as } n\to \infty .
\end{split}
\end{equation*}

Given $\{x_n\}$ a sequence in a metric space$(M,d)$, by $x_n\tilde x$ we mean that, for $n$ sufficiently large, $d(x_n, x) \to 0$. 

Finally, we shall denote with $\Diff(\C^p, 0)$ the space of germs of biholomorphisms of $\C^p$ fixing the origin.


\section{Preliminaries}\label{preli} 

One of the main tools in the study of the dynamics for tangent to the identity germs is 
the blow-up of the origin. In our case, it will suffice one blow-up to simplify our germ. 

\begin{definition}
Let $F\in\Diff(\C^p, O)$ be tangent to the identity. The {\it order} $\nu_0(F)$ of $F$ is the minimum $\nu\ge 2$ so that $P_{\nu}\not\equiv 0$, where we consider the expansion of as sum of homogeneous polynomials
$$
F(z) = \sum_{k\ge 1} P_k(z),
$$  
where $P_k$ is homogeneous of degree $k$ ($P_1(z) = z$). We say that $F$ is {\it non-degenerate} if $P_{\nu_0(F)}(z) = 0$ if and only if $z=0$.
\end{definition}

Let $\2\C^p\subset \C^p\times \C\P^{p-1}$ be defined by
$$
\2\C^p\{ (v, [l]) \in\C^p\times \C\P^{p-1} : v\in [l]\}.
$$
Using coordinates $(z_1,\dots, z_p)\in \C^p$ and $[S_1: \dots : S_p]\in\C\P^{p-1}$, we obtain that $\2\C^p$ is determined by the relations
$$
z_h S_k = z_k S_h
$$
for $h, k\in\{1,\dots, p\}$. It is well-known that $\2\C^p$ is a complex manifold of the same dimension as $\C^p$. Given $\sigma\colon \2\C^p\to \C^p$ the projection, the {\it exceptional divisor} $E:=\sigma^{-1}(0)$ is a complex submanifold of $\2\C^p$ and $\sigma|_{\2\C^p\setminus E}\colon \2\C^p\setminus E\to \C^p\setminus\{0\}$ is a biholomorphism. The datum $(\2\C^p, \sigma)$ is usually called {\it blow-up of $\C^p$ at the origin}. 

Note that an atlas of $\2\C^p$ is given by $\{(V_j, \phe_j)\}_{1\le j\le p}$, where
$$
V_i= \{ (z,[S] )\in \2\C^p \mid S_j\neq 0 \},
$$
and $\phe_j\colon V_j\to \C^p$ is given by
$$
\phe_j (z_1,\dots, z_p, [S_1:\cdots:1:\cdots:S_p])= \left( S_1,\dots , z_j,\dots ,  S_p \right),
$$
since the points in $\{S_j = 1\}$ satisfy $z_k = z_j S_k$ for $k\in\{1,\dots, p\}\setminus\{j\}$.
Moreover we have
$$
\phe^{-1}_j (z_1,\dots, z_p) = (z_1z_j, \dots, z_j, \dots, z_p z_j, [z_1:\cdots:1:\cdots:z_p])\in V_j.
$$
The projection $\sigma\colon\2\C^p\to\C^p$ is given by $\sigma(z, [S]) = z$, and in the charts $(V_j, \phe_j)$ it is given by
$$
\sigma \circ\phe^{-1}_j (z_1,\dots, z_p) = (z_1z_j, \dots, z_j, \dots, z_p z_j).
$$

\begin{proposition}\label{blowup}
Let $F\in\Diff(\C^p, 0)$ be tangent to the identity, and let $(\2\C^p, \sigma)$ be the blow-up of $\C^p$ at the origin. Then there exists a unique lift $\tilde F\in\Diff(\2\C^p, E)$ so that
$$
F\circ \sigma = \sigma\circ \tilde F. 
$$
Moreover, $F$ acts as the identity on the points of the exceptional divisor , i.e., $\tilde F(0, [S]) = (0, [S])$.
\end{proposition}

\sm We omit the proof of the previous result, which can be found in \cite{A1}. It is also possible to prove that there exists a unique lift for any endomorphism $G$ of $(\C^p, 0)$ so that $G(z)=\sum_{k\ge h} P_k(z)$, where $h$ is the minimum integer such that $P_h\not\equiv 0$ and so that $P_h(z) = 0$ if and only if $z=0$, and in such a case the action on the exceptional divisor is $\tilde G(0, [S]) = (0, [P_h(S)])$.

\section{Characteristic directions}

We shall use the following reformulation of Definition~2.1 and Definition~2.2 of \cite{Ha} for the case $k+1\ge 2$.

\begin{definition}
Let $F\in\Diff(\C^p, 0)$ be a tangent to the identity germ of order $k+1$, and let $P_{k+1}$ be the homogeneous polynomial of degree $k+1$  in the expansion of $F$ as sum of homogeneous polynomials (that is, the first non-linear term of the series). We shall say that $v\in\C^p\setminus\{0\}$ is a {\it characteristic direction} if $P_{k+1}(v) = \lambda v$ for some $\lambda\in\C$. Moreover, if $P_{k+1}(v) \ne 0$, we shall say that the characteristic direction is {\it non-degenerate}, otherwise, we shall call it {\it degenerate}. 
\end{definition}

Since characteristic direction are well-defined only as elements in $\C\P^{p-1}$, we shall use the notation $[v]\in\C\P^{p-1}$.

\begin{definition}
Let $F\in\Diff(\C^p, 0)$ be a tangent to the identity germ. A {\it characteristic trajectory} for $F$ is an orbit $\{X_n\}:= \{F^n(X)\}$ of a point $X$ in the domain of $F$, such that $\{X_n\}$ converges to the origin tangentially to a complex direction $[v]\in\C\P^{p-1}$, that is 
\begin{equation*}
\left\{\begin{array}{l l}
\lim\limits_{n\to \infty}X_n = 0, \\
\lim\limits_{n\to \infty}[X_n]= [v].
\end{array}\right.
\end{equation*}
\end{definition}

\sm The concepts of characteristic direction and characteristic trajectory are indeed linked as next result shows. We shall use coordinates, following Hakim \cite{Ha}, $z=(x, y) \in\C\times\C^{p-1}$ and $(x_n, y_n):= (f_1^n(x,y), f_2^n(x,y)) \in\C\times\C^{p-1}$ for the $n$-tuple iterate of $F$. We have the following generalization of Proposition 2.3 of \cite{Ha} for the case $k+1\ge 2$.

\begin{proposition}\label{prop:2}
Let $F\in\Diff(\C^p, 0)$ be a tangent to the identity germ, and let $\{X_n\}$ be a characteristic trajectory tangent to $[v]$ at the origin. Then $v$ is a characteristic direction. Moreover, if $[v]$ is non-degenerate, choosing coordinates so that $[v]=[1: u_0]$, writing $P_{k+1}(z) = (p_{k+1}(z), q_{k+1}(z))\in\C\times\C^{p-1}$, we have
\begin{equation}\label{stima1}
x_n^k \sim -\frac{1}{n k p_{k+1}(1, u_0)}, \textit{ as  } n\to \infty,
\end{equation}
where $X_n=(x_n, y_n)$.
\end{proposition}

\begin{proof}
If $P_{k+1}([v])=0$, then $[v]$ is a degenerate characteristic direction and there is nothing to prove. Hence we may assume $P_{k+1}([v])\neq 0$, and, up to reordering the coordinates, we may assume that  $[v]=[1:u_0]$ and $F$ is of the form 
\begin{equation}
\left\{\begin{array}{l l}
x_1= x+ p_{k+1}(x,y)+ p_{k+2}(x,y)+ \cdots, \\
y_1= y+ q_{k+1}(x,y)+ q_{k+2}(x,y)+ \cdots,
\end{array}\right. 
\end{equation}
where $x_1, \, x,\, p_j(x,y)\in \C$ and $y_1,\, y,\, q_j(x,y)\in \C^{p-1}$. Since $\{X_n\}$ is a characteristic trajectory tangent to $[v]$, we have
$$
\lim_{n\to \infty}\frac{y_n}{x_n}= u_0.
$$
Now we blow-up the origin and we consider a neighbourhood of $[v]$. If the blow-up is $y=ux $, with $u\in \C^{p-1}$, then the first coordinate of our map becomes
\begin{equation}\label{eq:eq2.2}
\begin{split}
x_1
&=x(1+p_{k+1}(1,u)x^k+p_{k+2}(1,u)x^{k+1}+\cdots), 
\end{split}
\end{equation}
whereas the other coordinates become
\begin{equation}\label{eq:eq2.3}
\begin{split}
u_1&= \frac{y_1}{x_1}
= u+ r(u)x^k + O(x^{k+1}),
\end{split}
\end{equation}
where 
$$
r(u):=q_{k+1}(1,u )- p_{k+1}(1,u) u.
$$
As a consequence, the non-degenerate characteristic directions of $F$ of the form $[1:u]$ coincide with the ones so that $u$ is a zero of the polynomial map $r(u)$:
\begin{equation*}
\left\{
\begin{array}{l l}
p_{k+1}(1,u)= \lambda \\
q_{k+1}(1,u)=\lambda u
\end{array}\right. 
\Longleftrightarrow r(u)= q_{k+1}(1,u) -p_{k+1}(1,u) u = 0.
\end{equation*}
It remains to prove that if $u_n=\frac{y_n}{x_n}$ converges to $u_0$, then $r(u_0)=0$. Since $u_n\to u_0$, the series 
\begin{equation}\label{serie}
\sum_{n=0}^\infty \left( u_{n+1}- u_n \right)
\end{equation}
is convergent. Thanks to \eqref{eq:eq2.3}, assuming $r(u_n)\neq 0$, we obtain 
$$
u_{n+1}-u_n = r(u_n)x_n^k+ \oo{x_n^{k+1}} \sim r(u_0)x_n^k.
$$
We can now prove \eqref{stima1}. In fact from
\begin{equation*}
\frac{1}{x_1}=\frac{1}{x}\left( 1-p_{k+1}(1,u)x^k+ O(x^{k+1})\right),
\end{equation*} 
we deduce
\begin{equation*}
\begin{split}
\frac{1}{x_1^k}
&= \frac{1}{x^k} -k p_{k+1}(1,u) +\oo{x},
\end{split}
\end{equation*} 
and hence 
\begin{equation*}
\begin{split}
\frac{1}{n x_n^k}
&= \frac{1}{n x^k} - \frac{k}{n}\sum_{j=0}^{n-1}\left(  p_{k+1}(1,u_{j}) +\oo{x_{j}}\right).
\end{split}
\end{equation*}
Setting $a_j:= p_{k+1}(1, u_j)+ O(x_j)$, since $a_j \to p_{k+1}(1,u_0)$, the average $\frac{1}{n}\sum_{j=0}^{n-1} a_j$ converges to the same limit. It follows that, as $n\to\infty$, $\frac{1}{n x_n^k}$ converges to $- k p_{k+1}(1,u_0)$ and 
$$
x_n^k\sim -\frac{1}{n  k p_{k+1}(1,u_0)}.
$$ 
If $r(u_0)\ne0$, then we could find $C\ne 0$ such that
$$ 
u_{n+1}-u_n\sim \frac{C}{n} \,{r(u_0)}, 
$$
and the series $\sum_{n=0}^\infty \left( u_{n+1}-u_n \right)$ would not converge, contradicting \eqref{serie}; hence $r(u_0) = 0$, and this concludes the proof. 
\end{proof}

\sm 

Unless specified, thanks to the previous results, without loss of generality, we shall assume that any given $F\in\Diff(\C^p, 0)$ tangent to the identity germ of order $k+1\ge 2$, with a non-degenerate characteristic direction $[v]$ is of the form 
\begin{equation}\label{eq:4.2}
\left\{\begin{array}{l l}
x_1=x(1+p_{k+1}(1,u) x^k+O(x^{k+1})), \\
u_1= u + (q_{k+1}(1,u )- p_{k+1}(1,u) u)x^k + O(x^{k+1}) ,
\end{array}\right. 
\end{equation}

\begin{lemma}
Let $F\in\Diff(\C^p, 0)$ be a tangent to the identity germ of order $k+1\ge 2$, of the form \eqref{eq:4.2}, with a non-degenerate characteristic direction $[v] = [1:u_0]$. Then there exists a polynomial change of coordinates holomorphically conjugating $F$ to a germ with first component of the form
$$
x_1=x - \frac{1}{k}x^{k+1} +\oo{x^{k+1}\|u\|, x^{2k+1}}.
$$
\end{lemma}

\begin{proof}
We shall first prove that it is possible to polynomially conjugate $F$ to a germ whose first coordinate has no terms in $x^h$ for $h= k+2, \dots, 2k$. 
Thanks to \eqref{eq:eq2.2}, expanding $p_{k+1}(1,u)$ in $u_0$, we obtain
$$
x_1 = f(x,u) = x + p_{k+1}(1,u_0)x^{k+1}+\oo{\|u\| x^{k+1}, x^{k+2}}.
$$
Now we use the same argument one can find in \cite[Theorem 6.5.7 p.122]{Beardon}, conjugating $f$ to polynomials $f_h$, for $1\leq h < k $, of the form 
$$
f_h(x,u)= x+ p_{k+1}(1,u_0)x^{k+1} + b_h x^{k+h+1}+\oo{\|u\| x^{k+1}, x^{k+h+2}},
$$
that is, changing polynomially the first coordinate $x$ and leaving the others invariant, up to get
$$
f_k(x,u)= x+ p_{k+1}(1,u_0)x^{k+1} +\oo{\| u\| x^{k+1}, x^{2k+1}}.
$$
Let us consider $g(x)=x + \beta x^{h+1}$, with $\beta := \frac{b_h}{(k-h)p_{k+1}(1,u_0)}$, and set $\Phi = (g,\id_{p-1} )\colon (x,u)\mapsto (g(x), u)$. Then, conjugating $F_h=(f_h, \Psi_h)$ via $\Phi$, we have $F_{h+1}\circ \Phi = \Phi \circ F_h $, which is equivalent to 
\begin{equation}\label{eq:phi-h}
\left\{\begin{array}{l l}
f_{h+1}(g(x),u) = g( f_h(x,u) ),\\
\Psi_{h+1}(g(x), u)=\Psi_h(x,u).
\end{array}\right.
\end{equation} 
Since $\Phi(0)=0$ and the Taylor expansion of $\Phi$ up to order $k+1$ only depends on $d\Phi_0 $, we must have
\begin{equation*}
\left\{\begin{array}{l l}
f_{h+1}(x,u)= x+ \sum_{m=k+1}^\infty A_m x^m +\oo{\| u\| x^{k+1}},\\
\Psi_{h+1}(x,u)= u + r(u)x^k + \oo{x^{k+1}},
\end{array}\right.
\end{equation*}
and in particular these changes of coordinates do not interfere on $\Psi$ in the order that we are considering.

Let us consider the terms up to order $k+h+2$ in the first equation of \eqref{eq:phi-h}. We obtain
\begin{equation*}
\begin{split}
g( f_h(x,u))
&= x+ p_{k+1}(1,u_0) x^{k+1} + b_h x^{k+h+1}+ \beta (x^{h+1}+ (h+1)x^{k+h+1}  ) + \oo{\| u |x^{k+1}, x^{k+h+2} },
\end{split}
\end{equation*}
and
\begin{equation*}
\begin{split}
f_{h+1} (g(x),u)
&= x +\beta x^{k+1} + A_{k+1}x^{k+1}+ \cdots + A_{k+h+1} x^{k+h+1} + A_{k+1}\beta (k+1) x^{k+h+1} +\oo{x^{k+h+2}, \| u\| x^{k+1}}.
\end{split}
\end{equation*}
Hence the coefficients $A_m$ satisfy 
\begin{equation*}
\left\{\begin{array}{l l}
A_{k+1}=p_{k+1}(1,u_0), \, A_{k+2}=0,\, \ldots ,\, A_{k+h}=0,\\
b_h + (h+1)p_{k+1}(1,u_0) \beta = \beta (k+1) A_{k+1} + A_{k+h+1},
\end{array}\right. 
\end{equation*}
yielding $A_{k+h+1}=0$. In particular there exists $b_{h+1}$ such that
\begin{equation*}
f_{h+1}(x,u)= x+ p_{k+1}(1,u_0)x^{k+1} + b_{h+1} x^{k+h+2}+\oo{\| u\| x^{k+1}, x^{k+h+3}}.
\end{equation*}
Repeating inductively this procedure up to $h=k-1$ we conjugate with a polynomial (and hence holomorphic) change of coordinates our original $F$ to a germ with no terms in $x^h$ for $h = k+2, \dots, 2k$, i.e.,
\begin{equation}\label{eq:2k+1}
x_1=f(x,u)=x+p_{k+1}(1,u_0)x^{k+1}+\oo{\| u \|x^{k+1}, x^{2k+1}}.
\end{equation}

Finally, using the change of coordinates acting as $x\mapsto X=\sqrt[k]{-p_{k+1}(1,u_0)k} \,x $ on the first  coordinate, and as the identity on the other coordinates, the germ \eqref{eq:2k+1} is transformed 
into
\begin{equation*}
X_1 = X - \frac{1}{k}X^{k+1} +\oo{\| u\| X^{k+1}, X^{2k+1}},
\end{equation*}
in the first component, whereas the other components, 
become
\begin{equation*}
U_1= U- r(U)\, \frac{X^k}{k\, p_{k+1}(1,u_0)}+ O(X^{k+1}).
\end{equation*}
\end{proof}

\sm 

Up to now, we simply acted on the first component of $F$, mainly focusing on the characteristic direction $[v]$. We shall now introduce a class of $(p-1) \times (p-1)$ complex matrices which takes care of the remaining $p-1$ components of $F$. 
We consider the Taylor expansion of $r$ in $u_0$, and we have
\begin{equation*}
\begin{split}
u_1
&= u- \frac{x^k}{k\, p_{k+1}(1,u_0)}r^\prime (u_0) (u-u_0) + \oo{\|{u-u_0}\|^2 x^k, x^{k+1}} ,
\end{split}
\end{equation*}
where $r^\prime(u_0)= {\rm Jac}(r)(u_0)$. It is then possible to associate to the characteristic direction $[v]=[1:u_0]$ the matrix
$$
A(v)=\frac{1}{ k\, p_{k+1}(1,u_0)}r^{\prime}(u_0),
$$
and hence, assuming without loss of generality $u_0=0$, after the previous reductions, the germ $F$ has the form
\begin{equation}\label{eq:equ}
\left\{\begin{array}{l l}
x_1=x - \frac{1}{k}x^{k+1} +\oo{\| u \|x^{k+1}, x^{2k+1}},\\
u_1= (I- x^k A)  u + \oo{\|{u}\|^2 x^k, x^{k+1}}.
\end{array}\right.
\end{equation}
The next result gives us a more geometric interpretation of this matrix.

\begin{lemma}\label{Matrice}
Let $F\in\Diff(\C^p, 0)$ be a tangent to the identity germ of order $k+1\ge 2$ and let $[v]\in \C\P^{p-1}$ be a non-degenerate characteristic direction for $F$ with associate matrix $A(v)$. Then the projection $\2P_{k+1}$ in $\C\mathbb P^{p-1}$ of the homogeneous polynomial $P_{k+1}$ of degree $k+1$ in the expansion of $F$ as sum of homogeneous polynomials induces $\2P_{k+1}\colon \C\mathbb P^{p-1}\to\C\mathbb P^{p-1}$, defined by
$$
\2P_{k+1}\colon [x]\mapsto [P_{k+1}(x)], 
$$
which is well-defined in a neighbourhood of $v$. Moreover, $[v]$ is a fixed point of $\2P_{k+1}$ and $A(v)$ is the matrix associated to the linear operator
$$
\frac{1}{k}\left( d(\tilde P_{k+1})_{[v]}- \id\right).
$$
\end{lemma}

\begin{proof}
The germ $F$ can be written as 
$$
F(z) = z + P_{k+1}(z) + P_{k+1}(z) + \cdots,
$$  
where $P_h$ is homogeneous of degree $h$. Let $[v]$ be a non-degenerate characteristic direction for $F$. 
The $p$-uple $P_{k+1}$ of homogeneous polynomials of degree $k+1$ induces a meromorphic map $\2P_{k+1}\colon \C\mathbb P^{p-1}\to\C\mathbb P^{p-1}$ given by 
$$
\2P_{k+1}\colon [x]\mapsto [P_{k+1}(x)], 
$$
and it is clear that the non-degenerate characteristic directions correspond to the fixed points of such a map, and the degenerate characteristic directions correspond to the indeterminacy points.

We may assume without loss of generality, $v=(1, u_0)$. Then
\begin{equation*}
U= \left\{ [ x_1 : \cdots :x_p ]\in \C\mathbb P ^{p-1} \mid x_1 \neq 0  \right\}
\end{equation*}
is an open neighourhood of $[v]$ and the map $\phe_1\colon U\to\C^{p-1}$ defined as
$$ 
[x_1: \cdots x_p] \mapsto \left( \frac{x_2}{x_1}, \cdots , \frac{x_p}{x_1}\right) =(u_1, \ldots, u_{p-1}),
$$
is a chart of $\C\P^{p-1}$ around $[v]$.

The differential $d(\2P_{k+1})_{[v]}: T_{[v]}\C\P^{p-1}\to T_{[v]}\C\P^{p-1}$ is a linear map, and it is represented, in $u_0=\phe_1([v])$, by the Jacobian matrix of the map
$$ 
g:= \phe_1 \circ \2P_{k+1}\circ \phe_1^{-1}: \phe_1(U)\to \phe_1 (\2P_{k+1}(U)) 
$$
given by
$$ 
u=(u_1,\ldots, u_{p-1})\mapsto \left( \frac{q_{k+1, 1}(1,u_1,\ldots, u_{p-1})}{p_{k+1}(1,u_1,\ldots, u_{p-1})}, \ldots , \frac{q_{k+1,p-1}(1,u_1,\ldots, u_{p-1})}{p_{k+1}(1,u_1,\ldots, u_{p-1})} \right). 
$$ 
We can associate to $[v]$ the linear endomorphism
\begin{equation*}
\6A_F([v])= \frac{1}{k}\left( d(\2P_{k+1})_{[v]}-\id\right) : T_{[v]}\C\mathbb P ^{p-1}\to T_{[v]}\C\mathbb P ^{p-1},
\end{equation*}
and we can then prove that the matrix of $\6A_F([v])$ coincides with $A(v)$. In fact, let $g_1,\ldots, g_{p-1}$ be the components of $g$. Since $g(u_0)=u_0$, we have
\begin{equation*}
\begin{split}
\frac{\partial g_i}{\partial u_j}(u_0)
&=\frac{1}{p_{k+1}(1,u_0)}\left[ \frac{\partial q_{k+1,i}}{\partial u_j}(1,u_0) -\frac{\partial p_{k+1}}{\partial u_j}(1,u_0)u_{0,i} \right],
\end{split}
\end{equation*}
for $i,j=1,\ldots ,p-1$. 
Therefore, it follows from $r_i(u)= q_{k+1,i}(1,u)-p_{k+1}(1,u)u_i$ that
$$ 
\frac{\partial r_i}{\partial u_j}(u_0)= \frac{\partial q_{k+1,i}}{\partial u_j}(1,u_0) - \frac{\partial p_{k+1}}{\partial u_j}u_{0,i} u_{0,i} - p_{k+1}(1,u_0)\delta_{i,j}, 
$$
and hence
$$
A(v)= \frac{1}{k}(g^\prime (u_0)-\id),
$$
concluding the proof.
\end{proof}

\sm
\begin{lemma}\label{lemma:prec}
Let $F\in \Diff(\C^p, 0)$ be a tangent to the identity germ and let $\phe \in \C\[X\]^p$ be an invertible formal transformation of $\C^p$. If $F= I + \sum_{h\ge k+1} P_h$ and $\phe= Q_1+\sum_{j\ge 2 } Q_j$ are the expansion of $F$ an $\phe$ as sums of homogeneous polynomials, then the expansion of $F^*= \phe^{-1}\circ F\circ \phe$ is of the form $I+ \sum_{h\ge k+1} P_h^*$, and:
\begin{equation}
P_{k+1}^* = Q_1^{-1}\circ  P_{k+1}\circ Q_1.
\end{equation}
\end{lemma}

\begin{proof} 
It is obvious that the linear term of $F^*$ is the identity.
It then suffices to consider the equivalent condition $F\circ\phe = \phe\circ F^*$, and to compare homogeneous terms up to order $k+1$, writing $F^*=\sum_{h\ge 1} P_h^*$.
\end{proof}

We are now able to prove , as in Proposition~2.4 of \cite{Ha}, that we can associate to $[v]$ the class of similarity of $A(v)$ .

\begin{proposition}
Let $F\in\Diff(\C^p, 0)$ be a tangent to the identity germ of order $k+1\ge 2$ and let $[v]=[1:u_0]\in \C\P^{p-1}$ be a non-degenerate characteristic direction for $F$. Then the class of similarity of $A(v)$ is invariant under formal changes of the coordinates.
\end{proposition} 

\begin{proof}
We may assume without loss of generality $[v]=[1:0]$, and hence $r(0)=0$. Up to a linear change of the coordinate we have
\begin{equation*}
u_1 = u + x^k r^\prime (0) u + \oo{\|u\|^2 x^k, x^{k+1}}.	
\end{equation*}
It suffices to consider linear changes of the coordinates. Indeed, writing $F$ in its expansion as sum of homogeneous polynomials $F= I + P_{k+1} + \sum_{j\ge k+2}P_j$, if $F$ is conjugated by $\phe\in\Diff(\C^p, 0)$ of the form $\phe = L + \sum_{j\ge 2}Q_j$, by Lemma \ref{lemma:prec} we have
$$
F^*= \phe^{-1} \circ F \circ \phe= I + L^{-1}\circ P_{k+1} \circ L +\cdots,
$$
and hence the expansion of $F^*$ up to order $k+1$ only depends on $d\phe_0$.

The projection of $P_{k+1}^*$ on $\C\P^{p-1}$ is, with the notation of Lemma \ref{Matrice}, $\2P_{k+1}^*= \tilde L^{-1}\circ \2P_{k+1}\circ \tilde L$, where $\tilde L $ is just the linear transformation of $\C\mathbb P^{p-1}$ induced by $L$ and $\2P_{k+1}$ is the projection of $P_{k+1}$. 
Note that $[v^*]$ is a characteristic direction for $F^*$ if and only if $[L v^*]$ is a characteristic direction for $F$.
Since we have
$$ 
d(\2P_{k+1}^*)_{[v^*]} = \tilde L ^{-1}\circ d(\2P_{k+1})_{[v]}\circ \tilde L,
$$
we obtain 
$$ 
\frac{1}{k}\left[ d(\tilde P_{k+1}^*)_{[v^*]} - I\right] = \tilde L ^{-1} \circ\frac{1}{k}\left( d(\tilde P_{k+1})_{[v]}- I \right) \circ \tilde L,$$
yielding, by Lemma \ref{Matrice},
$$
A^*(v^*)= L^{-1} A(v) L,
$$
which is the statement.
\end{proof}

As a corollary, we obtain that the eigenvalues of $A(v)$ are holomorphic (and formal) invariants associated to $[v]$, and so the following definition is well-posed.

\begin{definition}
Let $F\in\Diff(\C^p, 0)$ be a tangent to the identity germ of order $k+1\ge 2$ and let $[v]\in \C\P^{p-1}$ be a non-degenerate characteristic direction for $F$. The class of similarity of the matrix $A(v)$ is called (with a slight abuse of notation) the {\it matrix associated to $[v]$} and it is denoted by $A(v)$. The eigenvalues of the matrix $A(v)$ associated to $[v]$ are called {\it directors} of $v$. 
The direction $[v]$ is called {\it attracting} if all the real parts of its directors are strictly positive.
\end{definition}

\section{Changes of coordinates}

We proved in the previous section that in studying germs $F\in\Diff(\C^p, 0)$ tangent to the identity in a neighbourhood of a non-degenerate characteristic direction $[v]$, we can reduce ourselves to the case $v=(1,0)$ and $F$ of the form:

\begin{equation}\label{eq:eq1}
\left\{ \begin{array}{l l}
x_1 = f(x,u) = x -\frac{1}{k}x^{k+1} + O(\| u\| x^{k+1}, x^{2k+1}), \\
u_1 = \Psi(x,u) = (I-x^k A)u + O(\| u\| ^2 x^k, \| u\| x^{k+1})+ x^{k+1}\psi_1(x),
\end{array}
\right.
\end{equation}
where $A=A(v)$ is the $(p-1)\times (p-1)$ matrix associated to $v$, and $\psi_1$ is a holomorphic function. Moreover, we may assume $A$ to be in Jordan normal form. 

\sm In this section we shall perform changes of coordinates to find $F$-invariant holomorphic curves, tangent to the direction $u=0$, that is, we want to find a function $u$ holomorphic in an open set $U$ having the origin on its boundary, and such that
\begin{equation*}
\left\{ \begin{array}{l l l}
u: U\to \C^{p-1}, \\
u(0)=0,\, u^\prime (0)=0, \\
u(f(x,u(x)))=\Psi(x, u(x)).
\end{array}
\right.
\end{equation*}
If we have such a function, the $F$ invariant curve will just be $\phi(x)=(x, u(x))$. 

\sm 
We now give precise definitions, that generalize Definition~1.2 of \cite{Ha} and Definition~1.5 of \cite{hakim2} for the case $k+1\ge 2$.

\begin{definition}
Let $F\in\Diff(\C^p,0)$ be a tangent to the identity germ. A subset $M\subset\C^p$ is a {\it parabolic manifold of dimension $d$ at the origin tangent to a direction $V$} if:
\begin{enumerate}
\item there exist a domain $S$ in $\C^d$, with $0\in\d S$, and an injective map $\psi\colon S\to \C^p$ such that $\psi(S)=M$ and $\lim_{z\to 0}\psi(z)=0$;
\item for any sequence $\{X_h\}\subset S$ so that $X_h\to 0$, we have $[\psi(X_h)]\to [V]$;
\item $M$ is $F$-invariant and for each $p\in M$ the orbit of $p$ under $F$ converges to $0$.
\end{enumerate}
A parabolic manifold of dimension $1$ will we called {\it parabolic curve}.
\end{definition}

\sm 
We shall search for a function $\psi=(\id_{\C}, u)$, defined on the $k$ connected components of  $\mathbb D_r=\{x\in \C \mid |{x^k-r}|<r\}$, and taking values in $\C^p$, verifying 
$$
u(f(x,u(x)))=\Psi(x, u(x)),
$$
and, taking $r$ sufficiently small, we shall obtain parabolic curves.

\sm The idea is to first search for a formal transformation, and then to show its convergence in a sectorial neighbourhood of the origin. The general obstruction to this kind of procedure is given by the impossibility of proving directly the convergence of the formal series. 

\sm As we said, in this section we shall change coordinates to further simplify $F$, by means of changes defined in domains of $\C^p$, with $0$ on the boundary, and involving square roots and logarithms in the first variable $x$.

\sm
Following Hakim \cite{Ha}, we shall first deal with the 2-dimensional case ($p=2$), generalizing Propositions~3.1 and~3.5 of \cite{Ha} for the case $k+1\ge 2$, to better understand the changes of coordinates that we are going to use.
The equations \eqref{eq:eq1} for $p=2$ are the following:
\begin{equation}\label{eq:equaz2}
\left\{ \begin{array}{l l}
x_1 = f(x,u) = x -\frac{1}{k}x^{k+1} + O( u x^{k+1}, x^{2k+1}) \\
u_1 = \Psi(x,u) = (1-x^k\alpha)u + x^{k+1}\psi_1(x)+ O( u ^2 x^k, u x^{k+1}),
\end{array}
\right.
\end{equation}
where $\alpha\in\C$ is the director, and we shall need to consider separately the case $k \alpha \in \N$ and the case $k \alpha \notin \N$.

\subsection{Case $p=2$ and $k \alpha \notin \N^*$}

\begin{proposition}\label{prop:alfa1}
Let $F=(f,\Psi)\in\Diff(\C^2,0)$ be of the form \eqref{eq:equaz2}. If $k\alpha \notin \N$, then there exists a unique sequence $\{P_h\}_{h\in \N}\subset \C[x]$ of polynomials with $\deg(P_h) = h$ for each $h\in\N$, such that 
\begin{equation}\label{eq:eq_k}
\left\{ \begin{array}{l l}
P_h(0)=0, \\
\Psi \left(x, P_h(x) \right)= P_h\left( f(x,P_h (x)) \right)+ x^{h+k+1}\psi_{h+1}(x).
\end{array}
\right.
\end{equation}
Moreover $P_{h+1}(x)=P_h(x)+ c_{h+1}x^{h+1}$, where $c_{h+1}=\frac{k \psi_{h+1}(0) }{k \alpha-(h+1)}$.
\end{proposition}

\begin{proof}
We shall argue by induction on $h$.

If $h=1$, we have to search for $P_1=c_1 x$ satifying \eqref{eq:eq_k}. We have 
\begin{equation*}
\begin{split}
\Psi(x,P_1 (x) )
&=c_1 x\left( 1-\alpha x^k + O(x^{k+1})\right) +x^{k+1} \psi_1(x)
\end{split}
\end{equation*}
and
\begin{equation*}
P_1\left( f(x,P_1(x)) \right)= c_1\left( x-\frac{1}{k}x^{k+1}+O(x^{k+2}) \right).
\end{equation*}
Hence 
\begin{equation*}
\begin{split}
\Psi (x,P_1 (x) )- P_1 \left( f(x,P_1(x)) \right) 
&= c_1 x^{k+1}\left( \frac{1}{k}-\alpha+\frac{\psi_1(0)}{c_1}\right) +O(x^{k+2}).
\end{split}
\end{equation*}
To delete the terms of order less than $k+2$, we must set $c_1=\frac{k \psi_1(0)}{k\alpha -1}$, which is possible since $k\alpha\notin\N^*$. 

Let us now assume that we have a unique polynomial $P_h$ of degree $h$ satisying \eqref{eq:eq_k}. We search for a polynomial $P_{h+1}$ of degree $h+1$ and such that
$$
\Psi \left(x, P_{h+1}(x) \right)= P_{h+1}\left( f(x,P_{h+1} (x)) \right)+ x^{h+k+2}\psi_{h+2}(x).
$$
We can write $P_{h+1}$ as $P_{h+1}(x)=p_h(x)+c_{h+1}x^{h+1}$, where $p_h$ is a polynomial of degree $\leq h$ and $p_h(0)=0$ . In particular, 
\begin{equation*}
\begin{split}
P_{h+1}\left( f(x, P_{h+1}(x)) \right) = p_h \left( f(x, P_{h+1}(x)) \right) + c_{h+1}{ \left( f(x, P_{h+1}(x)) \right)^{h+1}}.
\end{split}
\end{equation*}
Let $x_1=f(x,u)= x- \frac{1}{k} x^{k+1}+ x^{k+1}\phe(x,u)$, with $\phe(x,u)\in\oo{x,u}$. We have
\begin{equation*}
\begin{split}
p_h \left( f(x, P_{h+1}(x)) \right)
&= p_h \left( x-\frac{1}{k}x^{k+1} +x^{k+1}\phe(x,p_h(x) ) \right)+ O(x^{k+h+2})\\
&= p_h \left( f(x,p_h(x)) \right)+ O(x^{k+h+2}), 
\end{split}
\end{equation*}
and
\begin{equation*}
\begin{split}
\left( f(x, P_{h+1}(x)) \right)^{h+1}
&=x^{h+1}\left[ 1-\frac{h+1}{k}x^k +O(x^{k+1}) \right] \\
&=x^{h+1}\left[ 1-\frac{h+1}{k}x^k\right] + O(x^{h+k+2}).
\end{split}
\end{equation*}
It thus follows
\begin{equation*}
P_{h+1}\left( f(x, P_{h+1}(x)) \right) =p_h \left( f(x,p_h(x))\right)+ c_{h+1}x^{h+1} - c_{h+1}\frac{h+1}{k}x^{h+k+1}+O(x^{h+k+2}).
\end{equation*}
By the second equation of \eqref{eq:equaz2}, $u_1=u[1-\alpha x^k+ x^k\phi(x,u)]+x^{k+1}\psi_1(x)$, with $\phi(x,u)\in O(x,u)$, and hence
\begin{equation}\label{eq:equaz3}
\begin{split}
\Psi \left( x,P_{h+1}(x) \right) 
&= \left[ p_h(x)+c_{h+1}x^{h+1} \right]\cdot\left[ 1-\alpha x^k +x^k \phi(x, P_{h+1}(x))\right] +x^{k+1} \psi_1(x)\\
&= \Psi\left( f(x,p_h(x)) \right) + c_{h+1} x^{h+1}- \alpha c_{h+1}x^{h+k+1} +O(x^{h+k+2}).
\end{split}
\end{equation}
Therefore
\begin{equation}\label{eq:3.12}
\begin{split}
&\Psi \left( x, P_{h+1}(x)  \right) - P_{h+1}\left( f(x,P_{h+1} (x)) \right)\\
&\quad\quad\quad\quad\quad = \Psi \left( f(x,p_h(x)) \right) -p_h \left( f(x,p_h(x))\right) + c_{h+1}\left(\frac{h+1}{k}-\alpha\right)x^{h+k+1} +O(x^{h+k+2}).
\end{split}
\end{equation}
To have $P_{h+1}$ satisying \eqref{eq:eq_k}, we need 
\begin{equation*}
\Psi \left( f(x,p_h(x)) \right) -p_h \left( f(x,p_h(x))\right) + c_{h+1}x^{h+k+1}\left(\frac{h+1}{k}-\alpha\right) = O(x^{h+k+2}),
\end{equation*}
that is, $p_h$ has to solve \eqref{eq:eq_k}; and this implies, by our induction hypothesis, $p_h=P_h$. Substituting $P_h$ to $p_h$ in \eqref{eq:3.12} and expanding $\psi_{h+1}$ in a neighbourhood of $0$ we get
\begin{equation*}
\begin{split}
\Psi \left( x, P_{h+1}(x)  \right) - P_{h+1}\left( f(x,P_{h+1} (x)) \right) =  x^{h+k+1}\left[ \psi_{h+1}(0) + c_{h+1}\left(\frac{h+1}{k}-\alpha\right)\right] +O(x^{h+k+2}),
\end{split}
\end{equation*}
and so we have to set the leading coefficient of $P_{h+1}$ to be
$$
c_{h+1}=\frac{k \psi_{h+1}(0)}{k\alpha -(h+1)},
$$
which is possible since $k\alpha\not\in\N^*$, and then we are done.
\end{proof}

The following reformulation of Corollary~3.2 of \cite{Ha} for the case $k+1\ge 2$, shows that we can rewrite the equations of $F$ in a more useful way, with a suitable change of coordinates.

\begin{corollary}
Let $F=(f,\Psi)\in\Diff(\C^2,0)$ be of the form \eqref{eq:equaz2}, with $k\alpha \notin \N$. Then, for any $h\in \N$, there exists a holomorphic change of coordinates conjugating $F$ to 
\begin{equation}\label{eq:equaz2_k}
\left\{ \begin{array}{l l}
x_1 = \tilde{f}(x,u) = x -\frac{1}{k}x^{k+1} + O( u x^{k+1}, x^{2k+1}), \\
u_1 = \2{\Psi}(x,u) = (1-\alpha x^k)u +x^{h+k} \psi_{h}(x)+ O( u ^2 x^k, u x^{k+1}).
\end{array}
\right.
\end{equation}
\end{corollary}

\begin{proof}
It is clear that the change of coordinates will involve only $u$. Let $h\in\N$, and let $P_{h-1}$ be the polynomial of degree $h-1$ of Proposition \ref{prop:alfa1} and consider the change of coordinates
\begin{equation*}
\left\{ \begin{array}{l l}
X= x,\\
U= u- P_{h-1}(x).
\end{array}
\right.
\end{equation*}
The first equation of \eqref{eq:equaz2} does not change, whereas the second one becomes
\begin{equation*}
\begin{split}
U_1 &=u_1 - P_{h-1}(x_1)\\
&={\Psi\left( X, U+P_{h-1}(X)\right)} - {P_{h-1}\left( f(X,U+P_{h-1}(X)) \right)}, 
\end{split}
\end{equation*}
where we have
\begin{equation*}
\begin{split}
\Psi\left( X, U+P_{h-1}(X)\right)
&=U[1-\alpha X^k]+ \Psi(X, P_{h-1}(X)) + O(U^2X^k,U X^{k+1}),
\end{split}
\end{equation*}
and, analogously to the previous proof, we can expand $P_{h-1}\left( f(X,U+P_{h-1}(X)) \right)$ at the first order in $U$ obtaining
\begin{equation*}
\begin{split}
P_{h-1}\left( f(X,U+P_{h-1}(X)) \right)
&=P_{h-1} \left( X-\frac{1}{k}X^{k+1}+X^{k+1}\phe_1(X,P_{h-1}(X))+ O(U X^{k+1}) \right) \\
&= P_{h-1}\left( f(X,P_{h-1}(X)) \right) + O(U X^{k+1}).
\end{split}
\end{equation*}
Therefore we have 
\begin{equation*}
\begin{split}
U_1= X^{h+k}\psi_{h}(X)+ U\left( 1-\alpha X^k + O(U X^k, X^{k+1}) \right),
\end{split}
\end{equation*}
and this concludes the proof.
\end{proof}

\subsection{Case $p=2$ and $k \alpha \in \N^*$}

We now consider the case $k\alpha\in \N^* $, $k\alpha \geq 1$. Proposition~3.3 of \cite{Ha} becomes the following.

\begin{proposition}\label{prop:alfa2}
Let $F=(f,\Psi)\in\Diff(\C^2,0)$ be of the form \eqref{eq:equaz2}, with $k\alpha \in \N$. 
Then there exists a sequence $\{P_h(x,t) \}_{h\in \N}$  of polynomials in two variables $(x,t)$ such that
$$ 
\tilde u_h(x):= P_h\left( x, x^{k\alpha} \log x \right), 
$$
has degree $\le h$ in $x$ (where consider as constant the terms in $\log x$). Moreover,
\begin{equation}\label{eq:relfunz}
\Psi\left( x,\tilde u_h(x) \right) - \tilde u_h\left( f(x,\tilde u_h(x)) \right) = x^{h+k+1}\psi_{h+1}(x),
\end{equation}
where $\psi_{h+1}$ satisfies
\begin{enumerate}
\item $x^{h+k}\psi_{h+1}$ is holomorphic in $x$ and $x^{k\alpha}\log x$;
\item $\psi_{h+1}(x)= R_{h+1}(\log x)+ \oo{x}$, with $R_{h+1}$ a polynomial of degree $p_{h+1}\in \N$, $p_{h+1}\leq h+1$. 
\end{enumerate}
\end{proposition}

\begin{proof}
The proof is done by induction on $h$.

If $h<k\alpha$, then the same argument of Proposition \ref{prop:alfa1} holds, since the polynomials $P_h$ are still well-defined. As a consequence, also the change of variables $u\mapsto u-P_{k\alpha-1}(x)$ is well-defined an hence we can assume that the second component of $F$ is of the form 
$$
u_1=u \left( 1-\alpha x^k+O(u x^k,x^{k+1}) \right) + x^{k\alpha+k}\psi_{k\alpha}(x).
$$
It is clear that, for $h< k\alpha$, the functions $\psi_h$ are holomorphic in $x$ and thus they satisfy the conditions  $(1)$ and $(2)$ of the statement. 

We can then assume that $F$ is of the form
\begin{equation}\label{eq:phi1-phi2}
\left\{ \begin{array}{l l}
x_1 = f(x,u) = x -\frac{1}{k}x^{k+1} + x^{k+1}\phe_1(x, u), \\
u_1 = \Psi(x,u)= u \left( 1 - \alpha x^k + x^k\phe_2( x,u)\right) + x^{k\alpha+k}\psi_{k\alpha}(x),
\end{array}
\right.
\end{equation}
where $\phe_1 $ and $\phe_2 $ are holomorphic functions or order at least $1$ in $x$ and $u$.

If $h=k\alpha$, it suffices to consider $P_{k\alpha}(x,t)=ct$, where $c=-k\psi_{k\alpha}(0)$. In fact $\tilde u_{k\alpha}(x)=c x^{k\alpha} \log x$ verifies \eqref{eq:relfunz} if
\begin{equation*}
\begin{split}
\Psi\left( x,\tilde u_{k\alpha}(x) \right) - \tilde u_{k\alpha}\left( f(x,\tilde u_{k\alpha}(x)) \right) 
&={\tilde u_{k\alpha}(x)\left[ 1-{\alpha}x^k +x^k\phe_2(x, \tilde u_{k\alpha}(x)) \right] + x^{k\alpha +k} \psi_{k\alpha} (x)}\\
&\quad \quad \quad \quad \quad\,\,- {\tilde u_{k\alpha}\left( x-\frac{1}{k}x^{k+1} + x^{k+1}\phe_1(x, \tilde u_{k\alpha}(x)) \right)}\\
&=\oo{x^{k\alpha + k +1}(\log x)^{p_h}},
\end{split}
\end{equation*}
for some $p_h\in \N$. Recall that
\begin{equation}\label{eq:derivate}
\begin{split}
\left\{\begin{array}{l l}
\frac{\partial f}{\partial u}=x^{k+1}\frac{\partial \phe_1}{\partial u}=\oo{x^{k+1}},\\
\frac{\partial \Psi}{\partial u}= 1- \alpha x^k + x^k\left(\phe_2(x,u)+u\frac{\partial \phe_2(x,u)}{\partial u}\right) =1- \alpha x^k  +\oo{x^{k+1}, ux^k}.
\end{array}\right. 
\end{split}
\end{equation}
We have
\begin{equation*}
\begin{split}
\tilde u_{k\alpha}(x)&\left[ 1-{\alpha}x^k +x^k\phe_2(x, \tilde u_{k\alpha}(x)) \right] + x^{k\alpha +k} \psi_{k\alpha} (x)\\
&= cx^{k\alpha} \log x - \alpha c x^{k\alpha +k}\log x + x^{k\alpha +k}\log x \cdot\phe_2\left( x, x^{k\alpha}\log x \right) + x^{k\alpha +k}\psi_{k\alpha}(x),
\end{split}
\end{equation*}
and
\begin{equation*}
\begin{split}
\tilde u_{k\alpha}&\left( x-\frac{1}{k}x^{k+1} + x^{k+1}\phe_1(x, \tilde u_{k\alpha}(x)) \right)\\
&= c\left( x-\frac{1}{k}x^{k+1} +O(x^{k\alpha + k+1}\log x, x^{2k+1}) \right) ^{k\alpha} \log\left({x-\frac{1}{k}x^{k+1} +O(x^{k\alpha + k+1}\log x, x^{2k+1})}\right) \\
&=c x^{k\alpha} \log x - c \alpha x^{k\alpha +k}\log x - \frac{c}{k} x^{k\alpha +k} + O(x^{2k\alpha +k}(\log x)^2 , x^{k\alpha +2k}\log x).
\end{split}
\end{equation*}
Therefore 
\begin{equation*}
\begin{split}
\Psi\left( x,\tilde u_{k\alpha}(x) \right) - \tilde u_{k\alpha}\left( f(x,\tilde u_{k\alpha}(x)) \right) 
&= x^{k\alpha +k}\psi_{k\alpha}(0)+\frac{c}{k} x^{k\alpha + k } + x^{k\alpha + k +1} \oo{x^{k\alpha -1}(\log x)^2, \log x}.
\end{split}
\end{equation*}
If $c=-k\psi_{k\alpha}(0)$, then 
$$ 
\Psi\left( x,\tilde u_{k\alpha}(x) \right) - \tilde u_{k\alpha}\left( f(x,\tilde u_{k\alpha}(x)) \right) =x^{k\alpha +k+1}\psi_{k\alpha +1}(x)= { O(x^{k\alpha + k +1}(\log x)^2)}.
$$
In particular, note that
\begin{equation*}
\psi_{k\alpha + 1}(x)= R_{k\alpha+1}(\log x)+ O(x),
\end{equation*}
where $R_{k\alpha+1}$ is a polynomial of degree $1$ or $2$, depending on whether $k\alpha=1$ or $k\alpha>1$. Also in this case $\psi_{k\alpha+1}$ satisfies the conditions $(1)$ and $(2)$ of the statement. Indeed, since $k\alpha + k\geq 2$, we have that $x^{k\alpha+k}R_{k\alpha+1}(\log x)$ is holomorphic in $x^{k\alpha} \log x$ and $x$.


We are left with the case $h> k\alpha$. The inductive hypothesis ensures that \eqref{eq:relfunz} holds for $h-1$ and there exists a polynomial $R_h(t)$ of degree $\le h$ so that $\psi_h(x)=R_h(\log x) + O(x)$. We search for $\tilde u_h$ of the form
\begin{equation}\label{p16}
\tilde u_h(x)=\tilde u_{h-1}(x)+x^h Q_h(\log x),
\end{equation}
where $Q_h$ is a polynomial, and we shall prove that $\tilde u_h$, of the form \eqref{p16}, satisfies \eqref{eq:relfunz} if and only if $Q_h$ is the unique polynomial solution of the following differential equation
$$ 
(h-k\alpha) Q_h(t) - Q_h^{\prime}(t) = kR_h(t).
$$
In fact we have
\begin{equation*}
\begin{split}
\Psi\left( x,\tilde u_h(x) \right)  - \tilde u_h\left( f(x,\tilde u_h(x)) \right) &={ \Psi\left( x,\tilde u_{h-1}(x)+x^h Q_h(\log x) \right) }
 - {\tilde u_{h-1}(f(x,\tilde u_h(x))) } \\
 & \quad\quad- {\left( f(x, \tilde u_h(x)) \right) ^h Q_h(\log (f(x, \tilde u_h(x))))} .
\end{split}
\end{equation*}
Thanks to the inductive hypothesis, in $\tilde u_h$ for $h\geq k\alpha$, the term of lower degree is $c x^{k\alpha}\log x$. We have 
\begin{equation*}
\begin{split}
\Psi( x,&\tilde u_{h-1}(x)+x^h Q_h(\log x) )\\
&= \Psi\left( x, \tilde u_{h-1}(x)+ x^h Q_h(\log x) \right)\\
&= \Psi(x, \tilde u_{h-1}(x)) +\frac{\partial \Psi}{\partial u}(x, \tilde u_{h-1}(x))x^hQ_h(\log x)+ \sum_{n\geq 2} \frac{1}{n!}\dfrac{\partial^n \Psi }{\partial u^n}(x, \tilde u_{h-1}(x))\left( x^hQ_h(\log x) \right)^n\\
&= \Psi(x, \tilde u_{h-1}(x)) + x^hQ_h(\log x)-\alpha x^{k+h}Q_h(\log x) + \oo{x^{h+k+k\alpha}(\log x)^{\deg Q_h +1}, x^{h+k+1}(\log x)^{\deg Q_h}}.
\end{split}
\end{equation*}
Analogously to the previous proof, using the first equation in \eqref{eq:derivate}, we have
\begin{equation*}
\begin{split}
f\left( x, \tilde u_{h-1}(x) + x^hQ_h(\log x)\right)
& = f(x, \tilde u_{h-1}(x))+ \sum_{n\geq 1}\frac{1}{n!}\frac{\partial^n f}{\partial x^n}(f(x,\tilde u_{h-1}(x)))\left( x^hQ_h(\log x) \right)^n\\
&= f(x, \tilde u_{h-1}(x)) +\oo{x^{h+k+1}(\log x)^{\deg Q_h}}.
\end{split}
\end{equation*}
Therefore
\begin{equation*}
\begin{split}
\tilde u_{h-1}(f(x,\tilde u_h(x)))
&= \tilde u_{h-1}\left( f(x, \tilde u_{h-1}(x) )+\oo{x^{h+k+1}(\log x)^{\deg Q_h}} \right)\\
&= \tilde u_{h-1}\left( f(x, \tilde u_{h-1}(x) ) \right) + \oo{x^{h+k+k\alpha}(\log x)^{\deg Q_h+1}}.
\end{split}
\end{equation*}
Finally, expanding $Q_h$ in a neighbourhood of $\log x$, and considering the terms of degree $h+k$ we obtain
\begin{equation*}
\begin{split}
( f(x, &\tilde u_h(x)) )^h Q_h(\log (f(x, \tilde u_h(x))))\\
&= \left[ x-\frac{1}{k}x^{k+1} +O(x^{k+1} \tilde u_h(x), x^{k+2}) \right] ^h  Q_h\left( \log{\left( x-\frac{1}{k}x^{k+1} +O(x^{k+1} \tilde u_h(x), x^{k+2}) \right)} \right) \\
&=\left[ x^h- \frac{h}{k}x^{h+k} +\oo{x^{k+h} \tilde u_h(x), x^{k+h+1} } \right] \\
&\quad \times \left[ Q_h(\log x) - \frac{1}{k}x^k Q_h^{\prime}(\log x) +\oo{x^k \tilde u_h(x)(\log x)^{\deg Q_h -1}, x^{k+1}(\log x)^{\deg Q_h -1}  } \right] \\
&=x^h Q_h(\log x) - \frac{1}{k}x^{h+k}Q_h^{\prime}(\log x) - \frac{h}{k} x^{h+k}Q_h(\log x)+\oo{x^{h+k+k\alpha}(\log x)^{\deg Q_h +1}, x^{h+k+1}(\log x)^{\deg Q_h}}.
\end{split}
\end{equation*}
The inductive hypothesis implies 
$$
\Psi(x,\tilde u_{h-1}(x))-\tilde u_{h-1}(f(x,\tilde u_{h-1}(x)))= x^{k+h}\psi_{h}(x),
$$
with $\psi_h(x)=R_h(\log x)+o(x)$. Reordering the terms, we then obtain
\begin{equation}\label{eq:analitic}
\begin{split}
\Psi\left( x,\tilde u_h(x) \right)  - \tilde u_h\left( f(x,\tilde u_h(x)) \right) 
= x^{h+k}&\left[ R_h(\log x) + \left( \frac{h}{k}- \alpha\right) Q_h(\log x)+ \frac{1}{k}Q_h^{\prime}(\log x) \right]\\
& + \oo{x^{h+k+k\alpha }(\log x)^{\deg Q_h+1}, x^{h+k+1}(\log x)^{\deg Q_h}},
\end{split}
\end{equation}
where $R_h(t)$ is the polynomial of degree $p_h\le h$. Hence $\tilde u_h$ satisfies \eqref{eq:relfunz} if and only if $Q_h$ is the unique solution of 
\begin{equation}\label{eq:5.12}
(k\alpha -h)Q_h(t) - Q_h^{\prime}(t)= kR_h(t).
\end{equation}
Therefore $R_{h+1}$ is a polynomial so that $\deg R_{h+1}\leq h+1$, and we can have $\deg R_{h+1}= h+1$ only if $k\alpha =1$.
Moreover, if $k\alpha =1$, $\deg R_{h+1}$ can be more that $h+1$.

We finally have to verify that $\psi_{h+1}$ is holomorphic, and that $\tilde u_h$ is a polynomial in $x$ and $x^{k\alpha} \log x$ of degree $\le h$ in $x$. Since $Q_h$ solves the differential equation \eqref{eq:5.12}, it has to be a polynomial of the same degree as $R_h$. Moreover, since $x^h R_h(\log x)$ is a polynomial in $x$ and $x^{k\alpha} \log x$, we have $p_h\leq \frac{h}{k\alpha}$. We thus conclude that $\tilde u_h$ is a polynomial in $x$ and $x^{k\alpha} \log x$ of degree $\le h$. 
Thanks to \eqref{eq:analitic}, $x^{h+k}\psi_{h+1}(x)$ is holomorphic in $x$ and $x^{k\alpha} \log x$.

\sm Summarizing, the sequence of polynomials is the following
\begin{equation*}
P_h(x,t)=\left\{
\begin{matrix}
\sum_{i=1}^{h}c_i x^i , c_i=\frac{k\psi_{i}(0)}{k\alpha -(i+1)}& \textrm{ if }h<k\alpha , \\
\psi_{k\alpha}(0)t&\textrm{ if }h =k\alpha ,  \\
P_{h-1}(x,t)+ x^hQ_h(\log x)& \textrm{ if } h>k\alpha . 
\end{matrix}\right.
\end{equation*}
\end{proof}

Similarly to the case $k\alpha\notin\N^*$, we deduce the following reformulation of Corollary~3.4 of \cite{Ha} for the case $k+1\ge 2$. 

\begin{corollary}
Let $F=(f,\Psi)\in\Diff(\C^2,0)$ be of the form \eqref{eq:equaz2}, with $k\alpha \in \N$. Then for any $h\in \N$ so that $h\geq \max \{k,k\alpha \}$ it is possible to choose local coordinates in which $F$ has the form
\begin{equation*}
\left\{ \begin{array}{l l}
x_1 = \tilde{f}(x,u) = x -\frac{1}{k}x^{k+1} + O( u x^{k+1}, x^{2k+1} \log x), \\
u_1 = \2{\Psi}(x,u) = u \left( 1-\alpha x^k + O( u x^k,  x^{k+1}\log x)\right) +x^{h+k} \psi_{h}(x),
\end{array}
\right. 
\end{equation*}
where $\tilde f ,\, \2{\Psi}$ and $x^{h+k-1} \psi_h(x)$ are holomorphic in $x,\, x^{k\alpha} \log x$ and $u$.
\end{corollary}

\begin{proof}
Consider $h\geq \max \{k,k\alpha \}$, and let $\tilde u_{h-1}$ be the polynomial map in $x$ and $x^{k\alpha}\log x$ given by the previous result. With the change of coordinates
\begin{equation*}
\left\{\begin{array}{l l}
X=x,\\
U= u- \tilde u_{h-1},
\end{array}\right.
\end{equation*}
the first equation becomes
\begin{equation*}
\begin{split}
X_1
&= X- \frac{1}{k}X^{k+1} +\oo{U X^{k+1}, X^{2k+1}\log X}.
\end{split}
\end{equation*}
In particular the term $x^{2k+1}\log x $ appears only if $k\alpha =1$. The second equation becomes
\begin{equation*}
\begin{split}
U_1 
&= u_1 - \tilde u_{h-1} (x_1) \\
&= U \left( 1-\alpha X^k\right) +\oo{U^2 X^k, U X^{k+1}\log x}+ X^{k+h}\psi_h(X).
\end{split}
\end{equation*}
Again, the term $U X^{k+1}\log x $ appears only if $k\alpha =1$, otherwise we have $UX^{k+1}$.
\end{proof}

\begin{remark} Note that if $k\alpha \in \N^*$, due to the presence of the logarithms, all the changes of coordinates used are not defined in a full neighbourhood of the origin, but in an open set having the origin on its boundary.
\end{remark}

\subsection{General case: $p>2$}

Now we deal with the general case of dimension $p>2$. Also in this case, the allowed changes of coordinates will depend on the arithmetic properties of the directors associated to the characteristic direction. 

\begin{proposition}\label{propgen}
Let $F=(f,\Psi)\in\Diff(\C^p,0)$ be of the form \eqref{eq:eq1}, let $[v]=[1:0]$ be a non-degenerate characteristic direction, and let $\{a_1,\ldots, a_s\}$ be the directors of $[v]$ so that $k a_j\in\N$. Then, for all $h\in\N$, there exists $\tilde u_{h}:\C\to \C^{p-1}$ so that its components are polynomials in $x,\,x^{k a_1}\log x,\ldots,\,x^{k a_s}\log x $ of degree $\le h$ in $x$, and the change of coordinates $u\mapsto u-\tilde u_h(x)$ conjugates $F$ to
\begin{equation}\label{eq:eqgen}
\left\{ \begin{array}{l l}
x_1 = \tilde{f}(x,U) = x -\frac{1}{k}x^{k+1} + O(\| U\| x^{k+1}, x^{2k+1} \log x), \\
U_1 = \2{\Psi}(x,U) = ( I- A x^k ) U+ \oo{\| U\|^2 x^k, \| U\| x^{k+1} \log x}  +x^{k+h} \psi_{h}(x),
\end{array}
\right. 
\end{equation}
with $\psi_h(x)=R_h(\log x)+\oo{x}$, where $R_h(t)=(R_h^1(t),\ldots,R_h^{p-1}(t))$ is a polynomial map with $\deg R_h^i = p_h^i\leq h$, for each $i=1,\ldots,p-1$.
\end{proposition}

\begin{proof}
We may assume without loss of generality that $A$ is in Jordan normal form.
For each fixed $h$, the $j$-th component of $\tilde u_h$ is determined by the components from $p-1$ to $j+1$, and each of them is determined with the results proved in dimension $2$.

It suffices to prove the statement when $A$ is a unique Jordan block of dimension $p-1$ with eigenvalue $\alpha$ and with elements out of the diagonal equal to $\alpha$. The equations of $F$ are
\begin{equation*}
\begin{split}
\left\{
\begin{array}{l l l}
x_1\quad\;\;=f(x,u)=x-\frac{1}{k}x^{k+1}+ \oo{x^{2k+1},\|{(u,v)}\| x^{k+1}},\\
u_{1,j}\quad=\Psi_j(x,u)=(1-x^k\alpha)u_j-x^k\alpha u_{j+1} +\oo{\|{u}\|^2 x^k, \|{u}\|x^{k+1}}+x^{k+1}\psi_{j}(x),\quad \hbox{for}~j=1,\dots, p-2\\
u_{1,p-1}=\Psi_{p-1}(x,u)=(1-x^k\alpha)u_{p-1} +\oo{\|{u}\|^2 x^k, \|{u}\|x^{k+1}}+x^{k+1}\psi_{p-1}(x),\\
\end{array}\right.
\end{split}
\end{equation*}
where $\psi_1,\dots,\psi_{p-1}$ are holomorphic bounded functions. 

We proceed by induction on $h$.
If $h=0$, it suffices to consider $\tilde{u}_0\equiv 0$. In fact,
\begin{equation*}
\Psi_j(x,\tilde{u}_0)- \tilde{u}_{0,j}\left( f(x,\tilde{u}_0)\right) = x^{k+1}\psi_j(x),\quad \hbox{for}~j=1,\dots, p-1.
\end{equation*} 
Let us then assume by inductive hypothesis, that there exist $\tilde{u}_{h-1}$ such that
\begin{equation}\label{eq:questa}
\Psi_j(x,\tilde{u}_{h-1})- \tilde{u}_{h-1,j}\left( f(x,\tilde{u}_{h-1})\right) = x^{k+h}\psi_{h,j}(x),\quad \hbox{for}~j=1,\dots, p-1.
\end{equation} 
Analogously to the 2-dimensional case, we want to find polynomials $Q_{h,1},\dots,Q_{h,p-1}$ so that 
$$
\tilde{u}_{h,j}(x)=\tilde{u}_{h-1,j}(x)+Q_{h,j}(\log x)x^h, \quad \hbox{for}~j=1,\dots, p-1,
$$ 
verify \eqref{eq:questa} for $h$. Proposition \ref{prop:alfa2} gives us that $\tilde{u}_{h,p-1}$ is a solution if and only if $Q_{h,p-1}$ verifies
$$
(k\alpha -h)Q_{h,p-1}(t)-(Q_{h,p-1}(t))^\prime(t)= kR_{h,p-1}(t).
$$
Moreover, we have $\deg R_{h,p-1} = p_{h,p-1} \leq h$.
We proceed in the same way for the remaining $\tilde{u}_{h,j}$'s, except for the fact that the equations are a bit different from the ones used before. In particular
\begin{equation*}
\left\{
\begin{array}{l l}
\displaystyle\frac{\partial \Psi_j}{\partial u_j}(x,u)= 1-\alpha x^k+\oo{x^{k+1}, \|{u}\|x^k},\\
\\
\displaystyle\frac{\partial \Psi_j}{\partial u_{j-1}}(x,u)= -\alpha x^k+\oo{x^{k+1}, \|{u}\|x^k}.\\
\end{array}\right.
\end{equation*}
Hence
\begin{equation*}
\begin{split}
{\Psi_j(x,\tilde{u}_h)}&= \Psi_j(x,\tilde{u}_{h-1}+x^hQ_h(\log x))\\
 &=\Psi_j(x,\tilde{u}_{h-1})+ (1-\alpha x^k)x^h Q_{h,j}(\log x)- \alpha x^{k+h} Q_{h,j+1}(\log x)\\
 &\quad\quad+ \oo{x^{k+h+1}(\log x)^{p_h}, \|{\tilde{u}_{h-1}}\|x^{k+h}(\log x)^{p_h}},
\end{split}
\end{equation*}
where $p_h=\max\deg Q_{h,j}$, and 
\begin{equation*}
\begin{split}
\tilde{u}_{h}\left( f(x,\tilde{u}_{h}) \right) &= {\tilde{u}_{h-1}\left( f(x,\tilde{u}_{h}) \right) } +{\left[ f(x,\tilde{u}_h )\right]^h Q_h\left(\log(f(x,\tilde{u}_h)\right)}.
\end{split}
\end{equation*}
We have
\begin{equation*}
{\tilde{u}_{h-1}\left( f(x,\tilde{u}_{h}) \right) } = \tilde u_{h-1}\left( f(x,\tilde u_{h-1}) \right) + \oo{x^{h+k+1}\log x, \norm{\tilde u_{h}} x^{h+k \log x} },
\end{equation*}
and
\begin{equation*}
\begin{split}
{\left[ f(x,\tilde{u}_h )\right]^h Q_h\left(\log(f(x,\tilde{u}_h)\right)}
&= \left[ x\left( 1-\frac{1}{k}x^k+\oo{x^{2k},\norm{\tilde{u}_{h}}x^{k} } \right) \right] ^h \!\!\!Q_h\left( \log x+\log \!\left( \!1-\frac{1}{k}x^k+\oo{x^{2k},\norm{\tilde{u}_{h}}x^{k} }\!\!\right)\!\!\right) \\
&=x^h Q_h(\log x)- x^{k+h}\left( \frac{1}{k}Q_h^\prime(\log x) + \frac{h}{k} Q_h(\log x) \right)\\ &\quad \times\oo{ x^{2k+h}(\log x)^{l_1},\norm{\tilde{u}_h}x^{k+h}(\log x)^{l_2} },
\end{split}
\end{equation*}
for some integer $l_1$ and $l_2$. It follows
\begin{equation*}
\begin{split}
\Psi_j(x,\tilde u_h) - \tilde u_{h,j}(f(x,\tilde u_h))
&=x^{k+h}\biggl[ \psi_{h,j}(x) + \frac{1}{k}Q_{h,j}^\prime(\log x) + \frac{h}{k} Q_{h,j}(\log x)-\alpha Q_{h,j}(\log x)-\alpha Q_{h,j+1}(\log x)\biggr]\\
&\quad +\oo{ x^{2k+h}(\log x)^{l_1},\norm{\tilde{u}_h}x^{k+h}(\log x)^{l_2} }.
\end{split}
\end{equation*}
Hence $\tilde{u}_h$ solves the equations if and only if $Q_{h,j}$ solves
\begin{equation*}
\left[ {h}-k\alpha \right] Q_{h,j}(t) +Q_{h,j}^\prime(t)= k\alpha Q_{h,j+1}(t) - kR_{h,j}(t),\quad\hbox{for}~j=1,\dots,p-2
\end{equation*}
and moreover $\deg R_{h,j}\leq h$. 
\end{proof}
\begin{remark}\label{oss4}
It is clear that in the previous proposition that we have no restrictions on $h$, and hence we can choose $h=k\bar h$, obtaining $F$ of the form
\begin{equation*}
\left\{ \begin{array}{l l}
x_1 = f(x,u) = x -\frac{1}{k}x^{k+1} + O(\norm u x^{k+1}, x^{2k+1} \log x) \\
u_1 = \Psi(x,u) = ( I- A x^k ) u+ \oo{\norm u ^2 x^k, \norm u x^{k+1} \log x}  +x^{k(\bar h+1)} \psi_{h}(x),
\end{array}
\right. 
\end{equation*}
where $\psi_h (x)= R_{k\bar h }(x)+ \oo{x}$. Then, up to changing the degree of the polynomials in  $\log x$, for any $h \in \N$ we can write
$$ 
u_1 = \Psi(x,u) = ( I- A x^k ) u+ \oo{\norm u ^2 x^k, \norm u x^{k+1} \log x}  +x^{k(h+1)} \psi_{h}(x).
$$
\end{remark}

\section{Existence of parabolic curves}

From now on, without loss of generality, we shall assume that non-degenerate characteristic direction is $[1:0]\in\C\P^{p-1}$. Moreover, thanks to Proposition \ref{propgen} and to Remark \ref{oss4}, after blowing-up the origin, it is possible to change coordinates, in a domain having the origin on its boundary, such that $F$, in the coordinates $(x,u)\in \C\times\C^{p-1}$, has the form
\begin{equation}\label{eq:main}
\left\{ \begin{array}{l l}
x_1 = f(x,u) = x -\frac{1}{k}x^{k+1} + O(\norm u x^{k+1}, x^{2k+1} \log x) \\
u_1 = \Psi(x,u) = ( I- A x^k ) u+ \oo{\norm u ^2 x^k, \norm u x^{k+1} (\log x)^{p_h}}  +x^{k(h+1)} \psi_{h}(x),
\end{array}
\right. 
\end{equation}
for an arbitrarily chosen $h \in \N$, and with $p_h\in\N\setminus\{0\}$ depending on $h$.  

\begin{remark}
The existence of parabolic curves $S_1,\ldots,S_k$ tangent to a given direction $[v]$ at $0$ is equivalent to finding $u$ defined and holomorphic on the $k$ connected components $\Pi_r^1,\ldots,\Pi_r^k$ of $\D_r :=\{ x\in \C \mid  |{x^k-r}|< r \}$ and such that
\begin{equation}\label{eq:funz}
\left\{ \begin{array}{l l}
u(f(x,u(x)))=\Psi(x,u(x))\\
\displaystyle\lim _{x\to 0}{u(x)} = \lim _{x\to 0}{u^\prime(x)} = 0.
\end{array}
\right.
\end{equation}
\end{remark}

We are going to prove the existence of such curves finding a fixed point of a suitable operator between Banach spaces. We shall then need to further simplify our equations via a change of coordinates holomorphic that will be holomorphic on $\Re (x^k) >0$.
Let us consider the new coordinates $(x,w)\in \C\times \C^{p-1}$, where $w\in \C^{p-1}$ is defined, on $\Re (x^k) >0$, by
\begin{equation*}
u=x^{kA} w := \exp{(kA \log x)}w.
\end{equation*}
Hence $u_1=x_1 ^{kA} w_1 $. Starting from \eqref{eq:main} we obtain
\begin{equation*}
x_1 - x = -\frac{1}{k}x^{k+1} +\oo{\norm u x^{k+1}, x^{2k+1}\log x}
\end{equation*}
and
\begin{equation}\label{eq:uno}
u_1 - (I- x^k A) u = \oo{\norm u ^2 x^k , \norm u x^{k+1} \log x, x^{k(h+1)}(\log x)^{p_h}} .
\end{equation}
Moreover, we have
\begin{equation}\label{eq:due}
\begin{split}
x_1 ^{kA} 
&= \exp \left( kA \left(\log x + \log \left( 1- \frac{1}{k}x^k + \oo{\norm u x^k, x^{2k} \log x }\right) \right) \right) \\
&= x^{kA} \left[ \left( I - x^k A \right) + \oo{\norm u x^k, x^{2k} \log x }  \right] .
\end{split}
\end{equation}
Using $x^{kA}w= u$, we have $x^{kA}w_1=x^{kA}x_1^{-kA}x_1^{kA}w_1=x^{kA}x_1^{-kA}u_1$. Set
\begin{equation}\label{eq:tre}
H(x,u):= x^{kA} (w-w_1)= u - x^{kA}x_1^{-kA}u_1.
\end{equation}
Thanks to \eqref{eq:due}, we have
\begin{equation}\label{eq:H}
\begin{split}
H(x,u)
&= u - \left[ (I-x^k A) + \oo{\norm u x^k, x^{2k\log x}} \right] ^{-1} u_1\\
&= -\left[ (I-x^k A) + \oo{\norm u x^k, x^{2k\log x}} \right] ^{-1}  \left[ u_1 - \left[ (I-x^k A) + \oo{\norm u x^k, x^{2k\log x}} \right] u \right]\\
&= \oo{\norm u ^2 x^k , \norm u x^{k+1} \log x, x^{k(h+1)}(\log x)^{p_h}}.
\end{split}
\end{equation}
Therefore we can write
\begin{equation*}
w_1 = w- x^{-kA} H(x,u).
\end{equation*}

Now we have all the ingredients to search for parabolic curves tangent to the direction $[v]$. 
For the moment, we only impose that $u$ is at least of order $k+1$. We have the following generalization of Lemma 4.2 of \cite{Ha} .

\begin{lemma}\label{lemma4.2}
Let $f$ be a holomorphic function defined as in the first equation of \eqref{eq:main}. For any $u$ so that $u(x)=x^{k+1} \ell(x)$, for some bounded holomorphic map $\ell\colon\phd\to\C^{p-1}$, let $\{x_n\}$ be the sequence of the iterates of $x$ via
$$ 
x_1 = f_u(x):= f\left( x, u(x) \right).
$$
Then, for $r$ small enough, for any $\ell$ so that $\norm \ell _{\infty}\leq 1$, and any $n\in \N$, if $x\in \phd$ then $x_n\in \phd$, and moreover
\begin{equation*}
\modu{x_n}\leq 2^{1/k} \frac{\modu x}{(\modu{1+nx^k})^{\frac{1}{k}}}.
\end{equation*}
\end{lemma}

\begin{proof}
Thanks to the hypothesis on $u$ we can rewrite the first equation of \eqref{eq:main}, obtaining 
\begin{equation*}
x_1= x- \frac{1}{k}x^{k+1} + ax^{2k+1} + bx^{2k+1} \log x +\oo{x^{2k+2}(\log x)^l, x^{2k+2}}.
\end{equation*}
By Proposition \ref{propgen}, we have the term $bx^{2k+1}\log x$ only if $1$ is an eigenvalue of $kA$. Moreover, we have
\begin{equation*}
x_1^k=x^k\left[ 1- x^k + kax^{2k} + kbx^{2k}\log x +\frac{1}{k^2}\binom{k}{2}x^{2k}+ \oo{x^{2k+1}(\log x)^l} \right].
\end{equation*}
Hence
\begin{equation*}
\begin{split}
\frac{1}{x_1^k}
&=\frac{1}{x^k} +1 + (1-a_1)x^{k} - b_1 x^{k} \log x + O(x^{k+1}(\log x)^l),
\end{split}
\end{equation*}
where $O(x^{k+1} (\log x)^l)$ represents a function bounded by $K\modu x ^{k+1} \modu{\log x }^l$, with $K$ not depending on $u$, because $\norm \ell_\infty\leq 1$.
It is thus possible to write
\begin{equation}\label{eq:4.16}
\begin{split}
\frac{1}{x_1^k}
&=\frac{1}{x^k}\left( 1 +x^k+ x^{2k} \left(\frac{a}{k} + \frac{2b}{k} \log x \right) + O(x^{2k+1}(\log x)^l) \right)
\end{split}
\end{equation}
where the same considerations hold as before. We can now define the following change of variable 
on $\Re x^k >0$
\begin{equation*}
\frac{1}{z}=\frac{1}{x^k} + a\log x + b (\log x)^2 .
\end{equation*}
Therefore \eqref{eq:4.16} becomes
\begin{equation*}
\begin{split}
\frac{1}{z_1}
&=\frac{1}{z} + 1 + O(x^{k+1}(\log x)^l),
\end{split}
\end{equation*}
where we used
\begin{equation*}
\begin{split}
\log{x_1}
=\log x - \frac{1}{k}x^k + O(x^{2k}(\log x)^2)
\end{split}
\end{equation*}
and
\begin{equation*}
\left( \log{x_1} \right) ^2 = \left( \log x \right) ^2 - \frac{2}{k}x^k\log x + O(x^{2k}(\log x)^2).
\end{equation*}
We then deduce
\begin{equation*}
\frac{1}{z_n}=\frac{1}{z_{n-1}}+1 + O(x_{n-1}^{k}(\log{x_{n-1}})^l)= \cdots =\frac{1}{z}+ n+ O(1).
\end{equation*}
On the other hand
\begin{equation*}
\begin{split}
\frac{1}{z_n}= \frac{1}{x_n^k} + a \log x_n +b \log ^2 x_n =\frac{1}{x_n^k}\left[ 1+ a x_n^k \log x_n +b x_n^k \log ^2 x_n \right],
\end{split}
\end{equation*}
and hence
\begin{equation*}
\begin{split}
\frac{1}{z} + n + O(1)
&=\frac{1}{x^k}\left( 1+ nx^k \right) \left[ 1+ \frac{ax^k\log x +bx^k\log^2 x +O(x^k)}{1+ nx^k} \right].
\end{split}
\end{equation*}
If $r$ is small enough, $f_u$ is an attracting map from $\phd$ in itself, and hence for any $\epsilon >0$ there exists $\bar n$ so that, for each $n> \bar n$ 
\begin{equation*}
\modu{a x_n^k \log x_n +b x_n^k \log ^2 x_n}< \epsilon , 
\end{equation*}
and
\begin{equation*}
\modu{\frac{ax^k\log x +bx^k\log^2 x +O(x^k)}{1+ nx^k}}< \epsilon . 
\end{equation*} 
Therefore, for $n> \bar n$ and $r$ small enough
\begin{equation*}
\begin{split}
\modu{x_n}^k
&=\modu{ \frac{x^k}{1+ nx^k}}\modu{ 1+ a x_n^k \log x_n +b x_n^k \log ^2 x_n } \modu{ \frac{1}{ 1+ \frac{ax^k\log x +bx^k\log^2 x +O(x^k)}{1+ nx^k} } } 
\leq 2\frac{\modu x ^k}{\modu{1+ nx^k}},
\end{split}
\end{equation*} 
and hence we obtain the statement.
\end{proof}

Analogously to Corollary~4.3 in \cite{Ha}, for the case $k+1\ge 2$ we have the following very useful inequality.

\begin{corollary}\label{cor:4.3}
Let $f$ be a holomorphic function defined as in the first equation of \eqref{eq:main}. For any $u$ so that $u(x)=x^{k+1} \ell(x)$, for some bounded holomorphic map $\ell\colon\phd\to\C^{p-1}$ with $\|\ell\|_\infty\le 1$, let $\{x_n\}$ be the sequence of the iterates of $x$ via
$$ 
x_1 = f_u(x):= f\left( x, u(x) \right),
$$
and let $r$ be sufficiently small. Then for any $\mu > k$ ($\mu \in \R$) and for any $q\in \N$ there exists a constant $C_{\mu,q}$ such that, for any $x\in \phd$, we have
\begin{equation*}
\sum_{n=0}^{\infty}\modu{x_n}^{\mu}\modu{\log x_n}^q\leq C_{\mu,q}\modu x ^{\mu -k}\modu{\log{\modu x}}^q.
\end{equation*}
\end{corollary}

\begin{proof}
If $x\in \phd$, then $\Re x^k >0$, and hence
$$
\modu{1+nx^k}^2=1+nx^k+n \bar{ x}^k +n^2 \modu{x^k} ^2\ge 1+ \modu{nx^k}^2.
$$ 
Then the inequality of the previous lemma becomes 
\begin{equation*}
\modu{x_n}\leq 2^{1/k}\frac{\modu x}{\modu{1+ nx}^{\frac{1}{k}}}\leq 2^{1/k} \frac{\modu x}{\sqrt[2k]{1+ \modu{nx^k}^2}}.
\end{equation*}
Recalling that, for $x$ sufficiently small, $\modu{\log x}\leq K_1 \modu{\log{\modu x}}$, for each $\mu > k$ and each $q\in \N$ we have
\begin{equation*}
\begin{split}
\modu{x_n}^{\mu}\modu{\log x_n}^q &\leq K_1 \modu{x_n}^{\mu}\modu{\log \modu{x_n}}^q\leq  K_2  \frac{\modu x ^{\mu}}{\sqrt[2k]{(1+ \modu{nx^k}^2)^{\mu}}}\modu{\log \frac{2^{1/k} \modu x}{\sqrt[2k]{1+ \modu{nx^k}^2}}}^q,
\end{split}
\end{equation*}
where $K_2=K_1 2^{\mu/k}$. We then have that there exists $K$ so that
\begin{equation}\label{eq:6.8}
\begin{split}
\sum_{n=0}^{\infty}\modu{x_n}^{\mu}\modu{\log x_n}^q 
&\leq K \int_0^{\infty}{\frac{\modu x ^{\mu}}{(1+ \modu{tx^k}^2)^{\mu /2k}}\modu{\log \frac{2^{1/k} \modu x}{\sqrt[2k]{1+ \modu{tx^k}^2}}}^q dt}\\
&= K\modu x^{\mu -k} \int_0^{\infty}{\frac{1}{(1+ s^2)^{\mu /2k}}\modu{\log \frac{2^{1/k} \modu x}{\sqrt[2k]{1+ s^2}}}^q ds}.
\end{split}
\end{equation} 
To conclude, it suffices the following estimate
\begin{equation*}
\begin{split}
\modu{\log\frac{2^{1/k}\modu x}{\sqrt[2k]{1+s^2}}}^q 
&\leq \modu{\log\modu x}^q \sum_{j=0}^q \binom{q}{j} \modu{\log\frac{\sqrt[2k]{1+s^2}}{2^{1/k}}}^j.
\end{split}
\end{equation*}
In fact, we have 
\begin{equation*}
\begin{split}
\int_0^{\infty}&{\frac{1}{(1+ s^2)^{\mu /2k}}\modu{\log \frac{2^{1/k} \modu x}{\sqrt[2k]{1+ s^2}}}^q ds} \leq \modu{\log\modu x}^q  \sum_{j=0}^q \binom{q}{j}\int_0^{\infty}{ \modu{\log\frac{\sqrt[2k]{1+s^2}}{2^{1/k}}}^j \frac{1}{(1+ s^2)^{\mu /2k}} ds}.
\end{split}
\end{equation*} 
that, together with \eqref{eq:6.8} yields
\begin{equation*}
\begin{split}
\sum_{n=0}^{\infty}\modu{x_n}^{\mu}\modu{\log x_n}^q &\leq C_{\mu,q}\modu x ^{\mu -k}\modu{\log{\modu x}}^q,
\end{split}
\end{equation*} 
where
\begin{equation*}
C_{\mu, q}:=K \sum_{j=0}^q \binom{q}{j}\int_0^{\infty}{ \modu{\log\frac{\sqrt[2k]{1+s^2}}{2^{1/k}}}^j \frac{1}{(1+ s^2)^{\mu /2k}} ds},
\end{equation*}
concluding the proof.
\end{proof}

We have the following analogous of Lemma~4.4 of \cite{Ha}.

\begin{lemma}\label{lemma4.4}
Let $f$ be a holomorphic function defined as in the first equation of \eqref{eq:main}. For any $u$ so that $u(x)=x^{k+1} \ell(x)$, for some bounded holomorphic map $\ell\colon\phd\to\C^{p-1}$, let $\{x_n\}$ be the sequence of the iterates of $x$ via
$$ 
x_1 = f_u(x):= f\left( x, u(x) \right).
$$
Then, if $r$ is sufficiently small, for any $\ell$ so that $\norm \ell_{\infty}\leq 1$ and $\norm{x\ell^{\prime}}_{\infty}\leq 1$, for each $n\in \N$ and each $x\in \overline{\phd}$, we have
\begin{equation*}
\modu{\frac{d x_n}{d x}}\leq 2 \frac{\modu{x_n}^{k+1}}{\modu x ^{k+1}}.
\end{equation*}
\end{lemma}

\begin{proof}
Arguing as in the proof of Lemma \ref{lemma4.2}, 
we have
\begin{equation}\label{eq:4.19}
\frac{1}{x_1^k}+ a\log x_1 + b (\log x_1)^2 =\frac{1}{x^k} + 1 + a\log x + b(\log x)^2 +\phe(x,u),
\end{equation}
where $\phe$ is holomorphic in $x, u, x^{\ell_j}\log x$ and 
\begin{equation*}
\phe(x,u)= O\left( x^{2k}(\log x)^l, \norm u \right) = O\left( x^{2k} (\log x)^l, x^{k+1} \norm \ell \right).
\end{equation*}
By \eqref{eq:4.19} we therefore have
\begin{equation*}
\frac{1}{x_n^k}+ a\log x_n + b (\log x_n)^2 
=\frac{1}{x^k} + n + a\log x + b(\log x)^2 + \sum_{p=0}^{n-1}\phe(x_p,u(x_p)).
\end{equation*}
Differentiating, we obtain
\begin{equation}\label{eq:dxn_dx}
\begin{split}
-\left[ \frac{k-ax_n^{k} -2b x_n^{k} \log x_n}{x_n^{k+1}} \right] \frac{d x_n}{d x}
=-\left[ \frac{k-ax^{k} -2b x^{k} \log x}{x^{k+1}} \right] + \sum_{p=0}^{n-1} \frac{d}{d x_p}\left[ \phe(x_p, u(x_p)) \right] \frac{d x_p}{d x}.
\end{split}
\end{equation}

We shall now proceed by induction on $n$. We first have to estimate the sum of the remainders $\phe(x_p,u(x_p))$. From the hypotheses for $\ell$ and the form of $\phe$ we deduce the existence of a constant $K$ so that
\begin{equation*}
\modu{\frac{d}{d x}\phe(x,u(x))} \leq K \left( \modu{\log\modu x}+ \norm \ell+ \norm{x \ell^{\prime}} \right) \modu x.
\end{equation*} 
For $n=1$ we have
\begin{equation*}
\begin{split}
\modu{\frac{d x_1}{d x}}&=\modu{\frac{k- ax^{k} - 2bx^{k}\log x + x^{k+1} \frac{d}{d x}\phe(x, u(x))   }{k- ax_1^{k} - 2bx_1^{k}\log x_1}\cdot \frac{x_1^{k+1}}{x^{k+1}}  } 
\leq D\frac{\modu{x_1}^{k+1}}{\modu x^{k+1}},
\end{split}
\end{equation*}
for a constant $D\in\R$, that can be chosen to be $D=2$, if $r$ is small enough.

Let us assume, by inductive hypothesis, $\modu{\frac{d x_p}{d x}}\leq 2 \frac{\modu{x_p}^{k+1}}{\modu x ^{k+1}}$ for any $p< n$. Then, by the previous corollary, we have
\begin{equation*}
\begin{split}
\modu{\sum_{p=0}^{n-1} \frac{d}{d x_p}\left[ \phe(x_p, u(x_p)) \right] \frac{d x_p}{d x} } 
&\leq 2 K(1+ \norm \ell + \norm{x\ell^{\prime}}) \sum_{p=0}^{\infty}\frac{\modu{x_p}^{k+2}}{\modu x ^{k+1}} + \sum_{p=0}^{\infty}\frac{\modu{x_p}^{k+2}\modu{\log\modu{x_p}}}{\modu x ^{k+1}} \\
&\leq 2 K(1+ \norm \ell + \norm{x\ell^{\prime}}) \frac{C_{k+2,0}}{\modu x^{k-1}}+ \frac{C_{k+2,1}}{\modu x^{k-1}}= \frac{K_1}{\modu x^{k-1}}.
\end{split}
\end{equation*}
Therefore, we obtain
\begin{equation*}
\begin{split}
\modu{\frac{d x_n}{d x}}
&\leq \frac{\modu{k- ax^{k} - 2bx^{k}\log x }+ K_1 \modu x^{2}}{\modu{k- ax_n^{k} - 2bx_n^{k}\log x_n}}\cdot \frac{\modu{x_n}^{k+1}}{\modu x^{k+1}}
\leq 2 \frac{\modu{x_n}^{k+1}}{\modu x ^{k+1}},
\end{split}
\end{equation*}
for $r$ small enough, and we are done.
\end{proof}

\subsection{The operator $\T$}

To find our desired holomorphic curve, we shall use, as announced, a certain operator acting on the space of maps $u$ of order $k+1\ge 2$. We saw that, given a map $u(\cdot)=x^{k+1}\ell(\cdot)$, with $\ell:\Pi_r^i\to \C^{p-1}$, the iterates $\{x_n\}$ of $x_0\in\Pi_r^i$ defined via
$$ 
x_{j+1}=f_u(x_j):= f(x_j,u(x_j)) 
$$
are well-defined for $r$ sufficiently small. With this choice for $u$, the operator
$$
\T u(x) =x^{kA} \sum_{n=0}^{\infty} x_n ^{-kA}H(x_n, u(x_n))
$$ 
where $A$ is the matrix associated to the non-degenerate characteristic direction we are studying, 
$$
H(x,u):= x^{kA} (w-w_1)= u - x^{kA}x_1^{-kA}u_1,
$$
and $\{x_n\}$ is the sequence of the iterates of $x$ under $f_u$, is well-defined, since the series converges normally.
We shall now restrict the space of definition of $\T$, to obtain a contracting operator. In particular, we are going to search for positive constants $r,\, C_0$ and $C_1$, so that $\T$ is well-defined on a closed subset of the Banach space of the maps of order $k+1\ge 2$. 

We have the following analogous of Definition~4.7 of \cite{Ha}. 

\begin{definition}
Let $k\in\N\setminus\{0\}$. Let $h,q\in \N$ be such that $hk\geq 3$ and $h\geq 1$, and let $r>0$. For any $i=1,\ldots,k$, let $B_{h,q,r}^i$ be the space of maps $u:\phd\to\C^{p-1}$, of the form $u(\cdot)=x^{kh-1}(\log x)^q t(\cdot)$ with $t$ holomorphic and bounded. The space $B_{h,q,r}^i$ endowed with the norm $\norm u = \norm t_{\infty}$ is a Banach space.
\end{definition}

\begin{definition}\label{defT}
Let $k\in\N\setminus\{0\}$, and let $h,q\in \N$ be such that $hk\geq 3$ and $h\geq 1$. Let $r, C_0$ and $C_1$ be positive real constants and let $\Et\subset B_{h,q,r}^i$ be the closed subset of  $B_{h,q,r}^i$ given by the maps so that
\begin{enumerate}
\item $\norm{u(x)}\leq C_0 \modu x^{kh-1} \modu{\log\modu x}^{q}$, for any $x\in \phd$;
\item $\norm{u^{\prime}(x)}\leq C_1 \modu x^{kh-2} \modu{\log\modu x}^{q}$, for any $x\in \phd$.
\end{enumerate}
Let $\T$ be the operator defined as
\begin{equation}\label{eq:4.14}
\T u(x) =x^{kA} \sum_{n=0}^{\infty} x_n ^{-kA}H(x_n, u(x_n)),
\end{equation}  
where $A$ is the matrix associated to the non-degenerate characteristic direction we are studying,  as in \eqref{eq:tre} we have
$$
H(x,u)= x^{kA} (w-w_1)= u - x^{kA}x_1^{-kA}u_1,
$$
and $\{x_n\}$ is the sequence of the iterates of $x$ under $f_u$.
\end{definition}

We shall devote the rest of the section to proving that the restriction of $\T$ to $\Et$ is a continuous operator and a contraction. It will thus admit a unique fixed point $u$, and we shall prove that the unique fixed point is a solution of the functional equation \eqref{eq:funz}.

We shall need the following reformulation of Lemma~4.1 of \cite{Ha} for the case $k+1\ge 2$.

\begin{lemma}\label{lemma4.1}
Let $\{\alpha_1, \alpha_2,\dots,\alpha_{p-1}\}$ be the directors of $A$, and let $\lambda=\max_j\{\Re \alpha_j\}$. If $\epsilon>0$, then for any $ x\in \Pi_r^i$, with $r$ small enough, we have
\begin{equation*}
\norm{x^{-kA}} \leq \modu{x}^{-k(\lambda + \epsilon)}.
\end{equation*}
\end{lemma}

\begin{proof}
We may assume without loss of generality that $A$ is in Jordan normal form, that is
$A=D+N$ where
\begin{equation*}
D=Diag (\alpha_1, \alpha_2,\dots,\alpha_{p-1} ), \, DN=ND,\, N^{p-1}=0.
\end{equation*}
Since $D$ and $N$ commute, we have $x^{-kA}= x^{-k(D+N)}= x^{-kD}\exp{(-kN\log x)}$, and so we have the following estimate
\begin{equation*}
\norm{x^{-kA}}\leq \norm{x^{-kD}} \norm{ \exp{(-kN\log x)} } \leq K \modu{x} ^{-k\lambda} \modu{\log x} ^{p-2} \leq \modu x ^{-k(\lambda +\epsilon)},
\end{equation*}
for $r$ small enough, and we are done.
\end{proof}

\begin{remark}
It follows from \eqref{eq:dxn_dx} that if $u\in B_{h,q,r}^i$, then the operator $H$ verifies
\begin{equation*}
H(x,u(x))=\oo{x^{k(h+1)}(\log x)^{q+1},x^{k(h+1)}(\log x)^{p_h}},
\end{equation*}
mapping $B_{h,q,r}^i$ into intself. We shall see that 
$$
(\T u)(x)=\oo{x^{kh-1}(\log x)^q},
$$
for $q\geq p_h$.
\end{remark}

We have the following generalization of Lemma~4.5 of \cite{Ha} for the case $k+1\ge 2$.

\begin{lemma}\label{lemma4.5}
Let $\T$ be the operator defined in Definition \ref{defT}. Let $\lambda=\max_j \{\Re \alpha_j \}$, where $\alpha_1,\dots,\alpha_{p-1}$ are the directors of the non-degenerate characteristic direction $[v]$, and let $h$ be an integer so that $h> \lambda +\epsilon$. Let $p_h$ be as in \eqref{eq:H}. Then, for $r$ sufficiently small, there exists a constant $C_0$ so that, for any $u$ satisfying 
\begin{equation}\label{eqstar}
\norm{u(x)} \leq C_0 \modu x^{hk-1} \modu{\log\modu x }^{p_h},
\end{equation}
for each $x\in \Pi_r^i$, we have that $\T u$ satisfies the same inequality in $\Pi_r^i$.
\end{lemma}

\begin{proof}
By the definition, we have
\begin{equation*}
\T u(x)= \sum_{n=0}^{\infty}\left( \frac{x_n}{x} \right)^{-kA} H(x_n, u(x_n)).
\end{equation*}
Thanks to equation \eqref{eq:H} we know that
$$
H(x,u)= \oo{\norm u^2 x^k, \norm u x^{k+1} \log x, x^{k(h+1)}(\log x)^{p_h}}.
$$
Therefore there exist $K_1,\,K_2,\,K_3$ such that
\begin{equation*}
\norm{H(x,u)}\leq K_1\norm u^2 \modu x^k + K_2 \norm u \modu x^{k+1} \modu{\log x} + K_3 \modu x^{k(h+1)}\modu{\log x}^{p_h}, 
\end{equation*}
in a neighbourhood of $0$. From the hypothesis $\norm{u(x)}\leq C_0 \modu x ^{kh-1} \modu{\log\modu x}^{p_h}$, it follows that for all $x\in\Pi_r^i$
\begin{equation*}
\begin{split}
\norm{H(x, u(x))} 
&\leq  K \modu x^{k(h+1)}\modu{\log\modu x}^{p_h},
\end{split}
\end{equation*}
with $K$ not depending on $C_0$ provided that $r$ is sufficiently small. 
Then we have 
\begin{equation*}
\norm{H(x_n, u(x_n))}\leq K \modu{x_n}^{k(h+1)}\modu{\log\modu{x_n}}^{p_h},
\end{equation*}
for $x\in \Pi_r^i$, and $r$ small. By Lemma \ref{lemma4.1} we have
\begin{equation*}
\normu{\left( \frac{x_n}{x} \right)^{-kA}}\leq \modu{\frac{x_n}{x}}^{-k(\lambda+\epsilon)}.
\end{equation*}
Applying all these inequalities to $\T u(x)$, and using Corollary \ref{cor:4.3} (note that $h>\lambda + \epsilon$), we obtain
\begin{equation*}
\begin{split}
\norm{\T u(x)} 
&\leq K\sum_{n=0}^{\infty}\modu{\frac{x_n}{x}}^{-k(\lambda+\epsilon)}\modu{x_n}^{k(h+1)}\modu{\log\modu{x_n}}^{p_h} \leq K^\prime\modu x^{kh} \modu{\log\modu x}^{p_h}\\
&\leq K^{\prime \prime}\modu x^{kh-1} \modu{\log\modu x}^{p_h},
\end{split}
\end{equation*}
and we are done.
\end{proof}

For our estimates we shall need the following technical result, generalizing Lemma~4.6 of \cite{Ha} for the case $k+1\ge 2$.

\begin{lemma}\label{lemma:dimo}
Let $\T$ be the operator defined as in Definition \ref{defT}. Let $h$, $p_h$ and $C_0$ be as in Lemma \ref{lemma4.5}. Then, for $r$ sufficiently small, there exists a constant $C_1$ such that for any $u$ satisfying \eqref{eqstar} and 
\begin{equation}\label{eq:6.12}
\norm{u^{\prime}(x)} \leq C_1 \modu x^{hk-2} \modu{\log\modu x }^{p_h},
\end{equation}
for each $x\in \Pi_r^i$, then $(\T u)^\prime$ satisfies the same inequality in $\Pi_r^i$.
\end{lemma}

\begin{proof}
By the definition of $\T$ we have
\begin{equation*}
\T u(x)= x^{kA}\sum_{n=0}^{\infty}\left( x_n \right)^{-kA} H(x_n, u(x_n)).
\end{equation*}
Then, differentiating, we obtain
\begin{equation*}
\begin{split}
\frac{d }{d x} \T u(x)
&=\underbrace{ \frac{d}{d x} x^{kA} \left( \sum_{n=0}^{\infty}x_n^{-kA}H(x_n, u(x_n)) \right) }_{S_1} + \underbrace{x^{kA} \sum_{n=0}^{\infty} \frac{\partial}{\partial u}\left( x_n^{-kA}H(x_n, u(x_n)) \right) \frac{d u}{d x_n}\frac{d x_n}{d x} }_{S_2}\\
&\quad +\underbrace{x^{kA} \sum_{n=0}^{\infty} \frac{\partial }{\partial x_n}\left( x_n^{-kA}H(x_n, u(x_n)) \right) \frac{d x_n}{d x} }_{S_3}.
\end{split}
\end{equation*}
We then have to estimate $S_1$, $S_2$, and $S_3$. Since 
$$
\frac{d x^{kA}}{dx}=kAx^{-1}x^{kA},
$$ 
we have
\begin{equation*}
S_1 = kAx^{-1} x^{kA} \left( \sum_{n=0}^{\infty} x_n^{-kA} H(x_n, u(x_n)) \right),
\end{equation*}
and thus, using the same inequalities as in the previous proof, we obtain
\begin{equation*}
\begin{split}
\norm{S_1}
&\leq \frac{k \norm A}{\modu x} C_0 \modu x^{kh-1} \modu{\log\modu x}^{p_h}
= D_1 \modu x^{kh-2}\modu{\log\modu x}^{p_h},
\end{split}
\end{equation*}
where $D_1=k \norm{A} C_0$. For the second term, we have
\begin{equation*}
\begin{split}
S_2 &= x^{kA} \sum_{n=0}^{\infty} x_n^{-kA}\frac{\partial H}{\partial u}(x_n, u(x_n)) \frac{d u}{d x_n}\frac{d x_n}{d x}.
\end{split}
\end{equation*}
Since $kh\geq 3$ the hypotheses of Lemma \ref{lemma4.4} are satisfied, and hence 
\begin{equation*}
\modu{\frac{dx_n}{dx}}\leq 2\modu{\frac{x_n}{x}}^{k+1}.
\end{equation*}
Moreover, $H(x,u)= \oo{\norm u^2 x^k, \norm u x^{k+1}\log x,  x^{k(h+1)}(\log x)^{p_h}}$ implies that there exist constants $K_1$ and $K_2$ so that
\begin{equation*}
\normu{\frac{\partial }{\partial u}\left( H(x,u) \right)} \leq K_1 \norm u \modu x^k + K_2\modu x^{k+1}\modu{\log\modu x},
\end{equation*}
and, our hypothesis gives that there is $C_0$ so that $\norm{u(x)}\leq C_0 \modu x^{kh-1} \modu{\log\modu x}^{p_h}$. Therefore
\begin{equation*}
\begin{split}
\normu{\frac{\partial H}{\partial u}(x,u(x)) }&\leq K_1 C_0 \modu x^{kh+k-1} \modu{\log\modu x}^{p_h} + K_2\modu x^{k+1}\modu{\log\modu x}
\leq C \modu x^{k+1} \modu{\log\modu x },
\end{split}
\end{equation*}
for some constant $C$, not depending on $C_0$. 
If $C_1$ is so that $\norm{u^{\prime}(x)}\leq C_1 \modu x^{kh-2} \modu{\log\modu x }^{p_h}$, then 
\begin{equation*}
\begin{split}
\normu{\frac{\partial H}{\partial u}(x_n, u(x_n)) \frac{d u(x_n)}{d x_n}\frac{d x_n}{d x} }
&= \normu{\frac{\partial H}{\partial u}(x_n, u(x_n))}\normu{\frac{d u(x_n)}{d x_n}}\modu{\frac{d x_n}{d x}}\\
&\leq 2C C_1 \modu x^{-(k+1)}\modu{x_n}^{2k+ kh}\modu{\log\modu{x_n}}^{p_h}.
\end{split}
\end{equation*}
Analogously to the proof of the previous result, $\norm{(\frac{x_n}{x})^{-kA}}\leq\modu{\frac{x_n}{x}}^{-k(\lambda+\epsilon)}$ and, by Corollary \ref{cor:4.3}, we have
\begin{equation*}
\begin{split}
\norm{S_2}
&\leq \sum_{n=0}^{\infty} 2C C_1 \modu{\frac{x_n}{x}}^{-k(\lambda+\epsilon)} \modu{\frac{x_n}{x}}^{k+1} \modu{x_n}^{kh+ k-1}\modu{\log\modu{x_n}}^{p_h+1} \\
&\leq D_2 \modu x^{kh-2}\modu{\log\modu x }^{p_h},
\end{split}
\end{equation*}
with $D_2$ not depending on $C_0$ and $C_1$. 

We are left with the third term
\begin{equation*}
S_3= x^{kA} \sum_{n=0}^{\infty} \frac{\partial G}{\partial x}(x_n, u(x_n)) \frac{d x_n}{d x},
\end{equation*}
where $G(x,u)= x^{-kA}H(x,u)$, and hence
\begin{equation*}
\frac{\partial G}{\partial x}=-\frac{kA}{x}x^{-kA} H(x,u)+x^{-kA} \frac{\partial H}{\partial x}(x,u).
\end{equation*}
With the same computations as before, using 
$$
H(x,u)=\oo{\normu u^2 x^k, \normu u x^{k+1}\log x,  x^{k(h+1)}(\log x)^{p_h}}
$$ 
and $\norm{u(x)}\leq C_0\modu x^{kh-1}\modu{\log\modu x}^{p_h}$, we have that there exist constants $K_1$, $K_2$ and $K_3$ so that
\begin{equation*}
\begin{split}
\normu{\frac{\partial H}{\partial x}}\leq K_1 \norm u^2\modu x^{k-1} +K_2 \norm u \modu x^k\modu{\log x} + K_3 \modu  x^{k(h+1)-1}\modu{\log\modu x}^{p_h}
\end{split}
\end{equation*}
and thus there exists $C$, depending of $C_0$, so that
\begin{equation*}
\normu{ x^{kA}\frac{\partial G}{\partial x}(x,u(x)) }\leq C \modu x^{k(h+1)-1} \modu{\log\modu x}^{p_h+1}.
\end{equation*}
Using again Corollary \ref{cor:4.3}, we obtain
\begin{equation*}
\begin{split}
\norm{S_3}
&\leq K_4\sum_{n=0}^\infty \modu{\frac{x_n}{x}}^{-k(\lambda + \epsilon)+k+1} \modu{x_n}^{k(h+1)-1}\modu{\log \modu{x_n}}^{p_h+1} \\
&\leq D_3 \modu x^{kh-2}\mlog x^{p_h},
\end{split}
\end{equation*}
with $D_3$ independent of $C_0$. 
Summing up, we obtain
\begin{equation*}
\normu{\frac{d}{dx}\T u(x)}\leq \norm{S_1}+\norm{S_2}+\norm{S_3}\leq (D_1 + D_2 + D_2)\modu x^{kh-2}\modu{\log\modu x}^{p_h},
\end{equation*}
and setting $C_1=D_1 + D_2 + D_3$ we conclude the proof.
\end{proof}

The previous two lemmas prove that $\T$ is an endomorphism of $\Et$. Now we have to prove that $\T$ is a contraction. We shall need the following reformulation of Lemma~4.9 of \cite{Ha} for the case $k+1\ge 2$.

\begin{lemma}\label{lemma:dimo2}
Let $u(\cdot)=x^{kh-1}(\log x)^{p_h}\ell_1(\cdot)$ and $v(\cdot)=x^{kh-1}(\log x)^{p_h}\ell_2(\cdot)$ be in $\Et$ and let $\{x_n\}$ and $\{x_n^{\prime}\}$ be the iterates of $x$ via $f_u$ and $f_v$. Then there exists a constant $K$ so that
\begin{equation*}
\modu{x_n^{\prime}-x_n} \leq K \modu x^{kh} \modu{\log\modu x}^{p_h}\norm{\ell_2- \ell_1}_{\infty}.
\end{equation*}
for any $n$, and $r$ small enough. 
\end{lemma}

\begin{proof}
Let $x$ and $x^{\prime}$ be in $\Pi_r^i$. We estimate
\begin{equation*}
f_v(\xp)-f_u(x)= f(\xp, v(\xp))- f(x, u(x)).
\end{equation*}
Thanks to \eqref{eq:main}, we can find constants $a,b,c$ and $m(x,u)$ so that
\begin{equation*}
\left\{\begin{array}{l l}
f_v(\xp)=\xp -\frac{1}{k}(\xp)^{k+1} + (\xp)^{2k+1} (a + b \log \xp) + c(\xp)^{k+1} v(\xp) + m(\xp, v),\\
f_u(x)= x -\frac{1}{k}x^{k+1} + x^{2k+1} (a + b \log x) + c x^{k+1} u(x) + m(x, u) .
\end{array}\right.
\end{equation*} 
Therefore we have
\begin{equation}\label{eq:4.25}
\begin{split}
f_v(\xp)-f_u(x)&= (\xp - x)\left[ 1 + \frac{1}{k} \sum_{i=0}^{k}(\xp)^i x^{k-i} + \oo{\modu{(x^{\prime\prime})}^{2k} \modu{\log{\modu x^{\prime\prime}}}} \right] + (v(\xp)- u(x))\oo{\modu{x^{\prime\prime}}^{k+1}},\\
\end{split}
\end{equation}
where $\xs = \max\{ \modu{\xp},\, \modu x \}$. Lemma \ref{lemma4.2} implies $x_n^k \sim (\x{n})^k \sim \frac{1}{n}$ as $n\to\infty$, then we can replace $\modu{\xs}^k$ with $\modu x^k$. Moreover, since
\begin{equation}\label{eq:dimo2}
\begin{split}
v(x^{\prime})
= v(x)+ \oo{x^{kh-2}(\log x)^{p_h}}(x^{\prime}-x),
\end{split}
\end{equation}
we obtain
\begin{equation*}
\begin{split}
v(\xp)-u(x)
&= v(\xp)- v(x)+ v(x)- u(x)\\
&= (\xp- x)\oo{\modu x^{kh-2}\modu{\log\modu x}^{p_h}}+ \oo{\modu x^{kh-1} \modu{\log\modu x}^{p_h}}\norm{\ell_2 - \ell_1}_{\infty}.
\end{split}
\end{equation*}
Then, substituting in \eqref{eq:4.25}, we have
\begin{equation*}
\begin{split}
f_v(\xp)-f_u(x)
&= (\xp - x)\left[ 1 - \frac{1}{k} \sum_{i=0}^{k}(\xp)^i x^{k-i} +\oo{\modu x^{kh+k-1}\modu{\log\modu x}, \modu x^{2k}\modu{\log\modu x}}\right] \\
&\quad + \oo{\modu x^{kh+k} \modu{\log\modu x}^{p_h}}\norm{\ell_2 - \ell_1}_{\infty}.
\end{split}
\end{equation*}
We are left with estimating $f_v(\xp)-f_u(x)$. For $x$ and $\xp$ in $\Pi_r^i$ and $r$ small enough we have
\begin{equation*}
\modu{ 1 - \frac{1}{k} \sum_{i=0}^{k}(\xp)^i x^{k-i} +\oo{\modu x^{2k}\modu{\log\modu x}} }= 1 + \oo{x^k} \leq 1.
\end{equation*}
Moreover, there exists a constant $K$ such that
\begin{equation*}
\modu{f_v(\xp)-f_u(x)}\leq \modu{\xp-x} + K \modu x^{kh+k}\modu{\log\modu x}^{p_h}\norm{\ell_2-\ell_1}_{\infty}.
\end{equation*}
Iterating, we obtain
\begin{equation*}
\begin{split}
\modu{f_v^n(\xp)-f_u^n(x)}
&\leq \modu{\xp- x}+ K \sum_{i=0}^{n-1} \modu{x_i}^{kh+k}\modu{\log\modu{x_i}}^{p_h}\norm{\ell_2-\ell_1}_{\infty},
\end{split}
\end{equation*}
for any $n$.

In particular, if $x=\xp$, we have
\begin{equation*}
\begin{split}
\modu{\x{n}-x_n} 
&\leq K \sum_{i=0}^{n-1} \modu{x_i}^{kh+k}\modu{\log\modu{x_i}}^{p_h}\norm{\ell_2-\ell_1}_{\infty} \\
&\leq K^{\prime} \modu x^{kh}\modu{\log\modu x}^{p_h}\norm{\ell_2-\ell_1}_{\infty},
\end{split}
\end{equation*}
where we used Corollary \ref{cor:4.3} to deduce the last inequality, and we put $K^{\prime}=K C_{k(h+1), p_h}$.
\end{proof}

We now have all the ingredients to prove, as in \cite[Proposition 4.8]{Ha}, that $\T|_{\Et}$ is a contraction.

\begin{proposition}
Let $\T$ be the operator defined in Definition \ref{defT}. Then for $r$ small enough 
$$
\T|_{\Et}\colon\Et\to\Et
$$
is a contraction.
\end{proposition}

\begin{proof}
We have to prove that given $u(\cdot)=x^{kh-1}(\log x)^{p_h}\ell_1(\cdot)$ and $v(\cdot)=x^{kh-1}(\log x)^{p_h}\ell_2(\cdot)$ in $\Et$, we have
\begin{equation*}
\norm{\T u - \T v }\leq C \norm{u - v}
\end{equation*}
with $C< 1$. 

We have
\begin{equation*}
\T u(x)- \T v(x)= x^{kA} \sum_{n=0}^{\infty} \left[ x_n^{-kA}H(x_n, u(x_n)) -x_n^{\prime -kA}H(\x{n}, v(\x n))\right];
\end{equation*}
hence 
\begin{equation*}
\begin{split}
\T u(x)- \T v(x)=&\underbrace{ x^{kA} \sum_{n=0}^{\infty} x_n^{-kA}\left[ H(x_n, u(x_n))- H(\x n, v(\x n))\right] }_{S_1} +\underbrace{ x^{kA} \sum_{n=0}^{\infty} \left[ x_n^{-kA}- x_n^{\prime -kA}\right] H(\x n, v(\x n))  }_{S_2}.
\end{split}
\end{equation*}
For $S_1$, since $H(x,u)=\oo{\norm u ^2 x^k,\, \norm u x^{k+1} \log x,\,x^{k(h+1)}(\log x)^{p_h}}$, for $u(x)=x^{kh-1}(\log x)^{p_h}\ell_1(x)$, there exist $\alpha(x,u)$ and $\beta(x,u)$ holomorphic in the variables $x,\,u$ and $x^k(\log x)^{p_h}$, so that
\begin{equation*}
H(x,u)= u x^{k+1} (\log x) \alpha(x,u) + x^{k(h+1)}(\log x)^{p_h}\beta(x,u).
\end{equation*}
Therefore, by the inequalities in the proof of Lemma \ref{lemma:dimo}, we obtain 
\begin{equation*}
\begin{split}
\normu{H(x_n, u(x_n)) - H(\x n, v(\x n))}&\leq K \left[ \normu{\frac{\partial H}{\partial x} (x_n,u(x_n)) }\modu{x_n - x_n^\prime } + \normu{ \frac{\partial H}{\partial u} (x_n, u(x_n))  } \normu{u(x_n)- v(x_n^\prime)} \right]\\
&\leq K_1\left[ \norm{u(x_n)- v(\x n)}\modu{x_n}^{k+1}\modu{\log\modu{x_n}} + \modu{x_n -\x n}\modu{x_n}^{k(h+1)-1} \modu{\log\modu{x_n}}^{p_h} \right] .
\end{split}
\end{equation*}
Arguing as in the proof of Lemma \ref{lemma:dimo2}, thanks to \eqref{eq:dimo2}, there exist constants $A^{\prime},\,B^{\prime}$ and $K_2$ such that
\begin{equation*}
\begin{split}
\norm{v(\x n)- u(x_n)}
&\leq A^{\prime} \modu{\x n- x_n}\modu{x_n}^{kh-2}\modu{\log\modu{x_n}}^{p_h}+ B^{\prime}\modu{x_n}^{kh-1} \modu{\log\modu{x_n}}^{p_h}\norm{\ell_2 - \ell_1}_{\infty} \\
&\le K_2 \modu{x_n}^{kh-2}\modu{\log\modu{x_n}}^{p_h} \left[ \modu x^{kh}\modu{\log\modu x}^{p_h}+ \modu{x_n} \right] \norm{\ell_2 - \ell_1}_{\infty},
\end{split}
\end{equation*}
where the last inequality follows form the previous lemma. Then
\begin{equation*}
\begin{split}
\norm{S_1}
&\leq K_1 \sum_{n=0}^{\infty}\modu{\frac{x_n}{x}}^{-k(\lambda +\epsilon)} \left\{ \modu{x_n -\x n}\modu{x_n}^{k(h+1)-1} \modu{\log\modu{x_n}}^{p_h} \right.\\
&\quad +\left.  K_2 \modu{x_n}^{k(h+1)-1}\modu{\log\modu{x_n}}^{p_h+1} \left[ \modu x^{kh}\modu{\log\modu x}^{p_h}+ \modu{x_n} \right] \norm{\ell_2 - \ell_1}_{\infty} \right\}.
\end{split}
\end{equation*}
Moreover, setting
$$
\tilde S:=\modu{x_n -\x n}\modu{x_n}^{k(h+1)-1} \modu{\log\modu{x_n}}^{p_h} +K_2 \modu{x_n}^{k(h+1)-1}\modu{\log\modu{x_n}}^{p_h+1} \left[ \modu x^{kh}\modu{\log\modu x}^{p_h}+ \modu{x_n} \right] \norm{\ell_2 - \ell_1}_{\infty},
$$
we have
\begin{equation*}
\begin{split}
\tilde S
&\leq {K^{\prime} \modu{x_n}^{k(h+1)-1}\modu{\log\modu{x_n}}^{p_h} \modu x^{kh}\modu{\log\modu x}^{p_h} \norm{\ell_2 - \ell_1}_{\infty}}\\
&\quad+ { K_2 \modu{x_n}^{k(h+1)-1}\modu{\log\modu{x_n}}^{p_h+1} \modu x^{kh}\modu{\log\modu x}^{p_h} \norm{\ell_2 - \ell_1}_{\infty}}\\
&\quad +{ K_2 \modu{x_n}^{k(h+1)}\modu{\log\modu{x_n}}^{p_h+1}\norm{\ell_2 - \ell_1}_{\infty}}
\end{split}
\end{equation*}
and applying Corollary \ref{cor:4.3},
\begin{equation*}
\begin{split}
\norm{S_1}
&\leq C_1 \modu x^{2kh-1}\modu{\log\modu x}^{2 p_h} \norm{\ell_2 - \ell_1}_{\infty} + C_2 \modu x^{2kh-1}\modu{\log\modu x}^{2 p_h + 1} \norm{\ell_2 - \ell_1}_{\infty} \\
&\quad +C_3 \modu x^{kh}\modu{\log\modu x}^{p_h+1} \norm{\ell_2 - \ell_1}_{\infty} \\
&\leq K_1 \modu x^{kh}\modu{\log\modu x}^{p_h+1} \norm{\ell_2 - \ell_1}_{\infty}.
\end{split}
\end{equation*}

We now consider $S_2$. We can write
\begin{equation*}
x_n^{-kA}- x_n^{\prime -kA}= x_n^{-kA}\left( I -\exp\left(-A\log\frac{x_n^{\prime}}{x_n} \right) \right).
\end{equation*} 
Therefore
\begin{equation*}
\begin{split}
\bigg\lVert   \bigg( I -\exp\bigg( -kA \log\frac{x_n^\prime}{x_n} \bigg) \bigg) H(x_n^\prime , v(x_n^\prime)) \bigg\rVert 
& \leq C \normu{k A \log \frac{x_n^\prime}{x_n}} \normu{H(x_n^\prime ,v(x_n^\prime))}\\
& \leq C^\prime \frac{\modu{x_n^\prime -x_n}}{\modu{x_n}}\modu{x_n}^{k(h+1)}\mlog{x_n}^{p_h}\\
&\leq C^{\prime\prime}x^{k(h+1)-1}\mlog{x_n}^{p_h}\modu x^{kh}\mlog x^{p_h}\norm{\ell_2 - \ell_1}_{\infty}. 
\end{split}
\end{equation*}
By Corollary \ref{cor:4.3} we have 
$$
\norm{S_2}\leq K_2 \modu x^{2kh -1} \mlog x^{2p_h}\norm{\ell_2 - \ell_1}_{\infty}.
$$
Thus, for $r$ small enough, there exists $K$ such that
\begin{equation*}
\normu{\T u(x)- \T v(x)}\leq K \modu x^{kh}\mlog x^{p_h}\norm{\ell_2 - \ell_1}_{\infty}.	 
\end{equation*}
From the definition of the norm in $E^i_{\T}(r, C_0, C_1)$, we have then that for $r$ small enough there is $c<1$ such that
\begin{equation*}
\norm{\T u - \T v}\leq c \norm{u-v},
\end{equation*}
proving that $\T|_{E^i_{\T}(r, C_0, C_1)}$ is a contraction.
\end{proof}

\begin{corollary}\label{main}
Let  $\T$ be the operator defined in Definition \ref{defT}. Then there exists $u:\Pi_r^i\to \C^{p-1}$ holomorphic and satisfying \eqref{eq:funz}.
\end{corollary}

\begin{proof}
Thanks to the previous Proposition, $\T$ is a contraction, and hence it has a unique fixed point $u\in \Et$. It suffices to prove that this $u$ satisfies \eqref{eq:funz}. 
The definition of $H$ gives us that $f(x,u)^{-kA} \Psi(x,u)= x^{-kA} u - x^{-kA } H(x,u)$, and hence
\begin{equation*}
H(x,u(x))= u(x)- x^{kA} x_1^{-kA} \Psi(x,u(x)).  
\end{equation*}
We therefore obtain
\begin{equation*}
\begin{split}
\T u(x)
&= x^{kA} \sum_{n=0}^{\infty} x_n^{-kA} H(x_n, u(x_n))\\
&=u(x)-x^{kA}  x_1 ^{-kA}\Psi(x,u(x))+ x^{kA} x_1^{-kA}[u(x_1)-x_1^{kA}x_2^{-kA}\Psi(x_1, u(x_1))]+ \cdots . 
\end{split}
\end{equation*}
This implies that $\T u= u $ if and only if
\begin{equation*}
-x^{kA} x_1^{-kA} [\Psi(x,u(x)) -u(x_1)] - x^{kA} x_2^{-kA}[\Psi(x_1,u(x_1)) -u(x_2)]+ \cdots =0,
\end{equation*}
that is 
$$
\Psi(x_n, u(x_n))= u (f(x_n, u(x_n))) \textrm{ for any } n\geq 0,
$$
and this concludes the proof.
\end{proof}

\section{Existence of attracting domains}

In this section, we shall prove that given non-degenerate attracting characteristic direction $[v]$ it is possible to find not only a curve tangent to $[v]$, but also a open connected set, containing the origin on its boundary and so that each of its points is attracted by the origin tangentially to $[v]$, that is, the following generalization of Theorem~5.1 of \cite{Ha} for the case $k+1\ge 2$.  

\begin{theorem}\label{th:7.1}
Let $F\in\Diff(\C^p,0)$ be a tangent to the identity germ of order $k+1\ge 2$, and let $[v]$ be a non-degenerate characteristic direction. 
If $[v]$ is attracting, then there exist $k$ parabolic invariant domains, where each point is attracted by the origin along a trajectory tangential to $[v]$.  
\end{theorem}

\begin{proof}
Since $[v]$ is a non-degenerate characteristic direction, we can find $r,c>0$ so that we can choose coordinates $(x,y) \in \C\times \C^{p-1}$ holomorphic in the sector 
\begin{equation*}
S_{r,c}^i =\left\{ (x,y)\in \C\times\C^{p-1} \mid x\in\Pi_r^i,\,\norm y \leq c\modu x \right\},
\end{equation*}
where $\Pi_r^i$ is one of the connected components of $\D_r=\left\{ \modu{x^k-r}<r\right\}$, so that, after the blow-up $y=ux$, $F$ is of the form
\begin{equation*}
\left\{
\begin{array}{l l}
x_1 =f(x,u)= x- \frac{1}{k}x^{k+1}+ \oo{\norm u x^{k+1}, x^{2k+1} \log x},\\
u_1= \Psi(x,u) = (I-x^kA) u +\oo{\norm u x^{k+1}\log x, \norm u ^2 x^k}.
\end{array}
\right.
\end{equation*}
In particular, after the blow-up, $\norm u \leq c$. 

Without loss of generality, we may assume that $A$ is in Jordan normal form.
Let $\{\alpha_1 ,\dots,\alpha_{p-1} \}$ be the eigenvalues of $A$. Thanks to the hypthesis, we have 
\begin{equation*}
\Re \alpha_j >0,\quad j=1,\dots,p-1 , 
\end{equation*}
and hence there exists a constant $\lambda>0$ so that $\Re \alpha_j > \lambda$ for all $j=1,\dots, p-1$. We can also assume that the elements off the diagonal in the Jordan blocks are all equal $\epsilon$, with $\eps< \lambda$. 

We shall now restrict our sectorial domain to obtain good estimates for $x_1$ and $u_1$. We define, for $j=1,\dots,p-1$,
\begin{equation*}
\Delta_j:=\{ x\in \C \mid \modu{1-\alpha_j x^k }\leq 1\}.
\end{equation*}
Consider the sector
\begin{equation*}
S_{\gamma,\rho}:=\{ x\in \C \mid \modu{\Im x}\leq \gamma \Re x ,\, \modu x \leq \rho \}.
\end{equation*}
Since $\Re \alpha_j > 0$, there exist positive constants $\gamma$ and $\rho$ so that, setting
for each $i=1,\ldots , k$,
\begin{equation*}
S_{\gamma,\rho}^i := \{ x\in \Pi_r^i \mid x^k\in S_{\gamma, \rho} \} 
\end{equation*}
we have
$$
S_{\gamma,\rho}^i\subset \bigcap_{j=1}^{p-1} \Delta_j \cap \overline{\D}_r \subset \Pi_r^i.
$$
We want to check that, for any $i=1,\ldots,k$, the $k$ sets 
\begin{equation*}
A_{\gamma,\rho, c}^i := \{(x,u)\in\C\times\C^{p-1}\mid x\in S_{\gamma ,\rho}^i,   \norm u \leq c \}
\end{equation*} 
are invariant attractive domains.

Recalling that there is $K$ so that
$$
\norm{u_1- (I-x^k A)u}\leq K(\norm u \modu x^{k+1} \mlog{x}+ \norm u^2 \modu x^k),
$$
for $(x,u)\in A_{\gamma,\rho, c}^i$ we have
\begin{equation*}
\norm{u_1} \leq \norm{(I- x^k A)u}+ K \norm u \modu x^k \left( \modu x \mlog x +\norm u\right),
\end{equation*}
and, provided that $\gamma,\,\rho$ and $c$ are small enough, 
\begin{equation}\label{prima}
\norm{u_1} \leq \norm u \norm{I-x^k A} \leq \norm u (1-\lambda \modu x^k)\le \|u\|,
\end{equation}
where we used that
\begin{equation*}
\begin{split}
\norm{I- x^k A}
&\leq \max_j\modu{1 - \alpha_j x^k } +\modu x ^k \epsilon
\le 1 - (\lambda + \eps')|x^k| + \eps|x^k|.
\end{split}
\end{equation*}
Therefore $\norm{u_1}\le c$. 

To estimate $x_1$, since we know that $x_1= x-\frac{1}{k} x^{k+1} +\oo{\norm u x^{k+1}, x^{2k+1} \log x}$, we have
\begin{equation}\label{eq:7.0}
\begin{split}
\frac{1}{x_1^k}
&=\frac{1}{x^k}+ 1+ \oo{\norm u, x^k\log x}.
\end{split}
\end{equation}
Therefore there is  $\tilde C$, not depending on $u$, so that 
\begin{equation}\label{seconda}
\modu{\frac{1}{x_1^k}- \frac{1}{x^k} -1 }\leq \tilde C\norm u + K \modu x ^k \mlog x\le  \tilde Cc + K \modu x ^k \mlog x.
\end{equation}
We shall use this last inequality to prove that $A_{\gamma,\rho, c}^i$ is an invariant domain. In particular, it suffices to check
\begin{equation*}
\left\{\begin{array}{l l}
\norm{u_1}\leq c \\
\modu{\Im x_1^k}\leq \gamma \Re x_1^k \\
\modu{x_1^k}\leq \rho .
\end{array}\right.
\end{equation*}
We already estimated $u_1$ in \eqref{prima}. 
On the other hand, to prove that $S_{\gamma,\rho, c}^i$ is $f$-invariant it suffices to prove that, for $u$ small enough, $(S_{\gamma,\rho, c}^i)^*=\{x\in \C \mid \frac{1}{x}\in S_{\gamma,\rho} \} $ is $\frac{1}{(f)^k}$-invariant, which follows from \eqref{seconda} using the same argument as in the proof of Leau-Fatou flower Theorem.

To finish, it remains to check that, given a point $(x,u)\in S_{\gamma,\rho, c}^i$, its iterates converge to the origin along the direction $[1:0]$. We shall first show that $x_n^k\sim\frac{1}{n}$ and $\norm{u_n}\leq C\frac{1}{n^{\lambda}}$, for any fixed $0<\lambda < \max_j \Re \alpha_j$. It follows from \eqref{eq:7.0} that 
\begin{equation*}
\begin{split}
\frac{1}{x_n^k}
&=\frac{1}{x^k}+ n +\sum_{i=0}^{n-1}\oo{\norm{u_i}, x_i^k\log x_i},
\end{split}
\end{equation*}
and hence
\begin{equation*}
\frac{1}{n x_n^k}=\frac{1}{n x^k}+ 1 +\frac{1}{n}\sum_{i=0}^{n-1}\oo{\norm{u_i}, x_i^k\log x_i}, 
\end{equation*}
where the sum is bounded. Therefore
$$
\frac{1}{n x_n^k}=\oo{1},
$$
yielding
$$
x_n^k\sim \frac{1}{n}.
$$

Finally, take $\mu < \lambda$ (where $\lambda$ is the positive constant so that $\max_j \Re \alpha_j>\lambda$. Then
\begin{equation*}
\begin{split}
x_1^{-k\mu} 
&= x^{-k\mu} \left[ 1- \frac{1}{k}x^k + \oo{ x^{2k}\log x, \norm u x^k }\right]^{-k\mu}
= x^{-k\mu} \left[ 1+ \mu x^k + \oo{ x^{2k}\log x, \norm u x^k } \right],
\end{split}
\end{equation*}
and hence
\begin{equation*}
\begin{split}
\modu{x_1}^{-k\mu} 
&\leq \modu x ^{-k\mu} \modu{ 1+ \mu x^k +\oo{\norm u x^k, x^{2k}\log x} }
\leq \modu x ^{-k\mu} (1 + \lambda \modu x^k).
\end{split}
\end{equation*}
It thus follows that 
\begin{equation*}
\begin{split}
\norm{u_1}\modu{x_1}^{-k\mu}&\leq \norm u (1-\lambda \modu x ^{k})\modu x ^{-k\mu}(1+\lambda \modu x ^k)
= \norm u \modu x ^{-k\mu} (1 -\lambda ^2 \modu x ^{2k})\\
&< \norm u \modu x ^{-k\mu}.
\end{split}
\end{equation*}
Therefore, there exists $C$ so that 
\begin{equation*}
\norm{u_n} \modu{x_n}^{-k\mu}< \norm u \modu x^{-k\mu}\leq C,
\end{equation*}
implying
\begin{equation*}
\norm{u_n}\leq C \modu{x_n}^{k\mu}.
\end{equation*}
Then, $\norm{u_n}=\displaystyle\oo{\frac{1}{n^{k\lambda}}}$. This shows that each $(x,u)\in A_{\gamma,\rho, c}^i$ converges to the origin along the direction $[1:0]$.
\end{proof}

\section{Parabolic manifolds}

Let $\Phi\in\Diff(\C^p,0)$ be a tangent to the identity germ of order $k+1\ge 2$, and let $[V]=[1:0]$ be a non-degenerate characteristic direction. We can divide the set of the directors of $[V]$ into two sets: the {\it attracting directors}, i.e., the set $\{\lambda_1,\dots,\lambda_a\}$ with $\Re \lambda_j >0$ for $j=1,\dots, a$, and the {\it non-attracting directors}, i.e., the set $\{\mu,\dots,\mu_b\}$ with $\Re \mu_h \le 0$ for $h=1,\dots, b$. Let $d_j$ be the multiplicity of $\lambda_j$ for $j=1,\dots, a$ and let $d:=d_1+\cdots+ d_a$.
We know that, after the blow-up, we can assume that $\Phi$ is of the form
\begin{equation}\label{eq:eqprinc}
\begin{split}
\left\{
\begin{array}{l l l}
x_1=f(x,u,v)=x-\frac{1}{k} x^{k+1} +F(x,u,v),\\
u_1=g(x,u,v)=(I_d- x^kA)u + G(x,u,v),\\
v_1=h(x,u,v)=(I_l- x^kB)v + H(x,u,v),
\end{array}\right. 
\end{split}
\end{equation}
where $A$ is the $d\times d$ matrix in Jordan normal form associated to the attracting directors, $B$ is the $l\times l$ matrix in Jordan normal form associated to the non-attracting directors (where $l:= p-d-1$), and with $F,\,G,\,H$ so that
\begin{equation}\label{eq:acca}
\begin{split}
\left\{
\begin{array}{l l l}
F(x,u,v)= \oo{\norm{(u,v)}x^{k+1}, x^{2k+1}\log x}, \\
G(x,u,v)=\oo{\norm{(u,v)}x^{k+1}\log x, \norm{(u,v)}^2 x^k},\\
H(x,u,v)=\oo{\norm{(u,v)}x^{k+1}\log x, \norm{(u,v)}^2 x^k}.
\end{array}\right. 
\end{split}
\end{equation}
Moreover $F,\,G,\,H$ are holomorphic in an open set of the form
$$
\Delta_{r,\rho}=\left\{(x,u,v)\in\C\times\C^d\times\C^{p-d-1} \;\Bigm\vert\; \big\lvert x^k-r\big\rvert <r, \, \norm{(u,v)}< \rho \right\}, 
$$
and therefore also in the set
$$
S_{\gamma,s, \rho}:=\left\{(x,U)\in\C\times\C^{p-1} \;\Bigm\vert\; \bigr\vert \Im x^k\bigr\vert \leq \gamma\, \Re x^k, \, |{x^k}| < s, \, \norm U < \rho\right\}\subset\Delta_{r,\rho}.
$$

\sm In the next result, the analogous of Proposition~2.2 of \cite{hakim2}, we shall see that it is possible to further modify the last $p-d-1$ components of $\Phi$.
 
\begin{proposition}\label{prop:ora}
Let $\Phi\in\Diff(\C^p,0)$ be a tangent to the identity germ of order $k+1\ge 2$ as in \eqref{eq:eqprinc}, with $[V]=[1:0]$ non-degenerate characteristic direction so that the matrix $A(v)=Diag(A,B)$ satisfies
\begin{equation*}
\begin{array}{l l}
\Re\lambda_j >\alpha >0 ,\,\textrm{ for any $\lambda_j$ eigenvalue of $A$} \\
\Re\mu_j \leq 0, \quad\quad\textrm{ for any $\mu_j$ eigenvalue of $B$}.
\end{array}
\end{equation*}
Then, for any choice of $N,m\geq 2$, it is possible to choose coordinates $(x,u,v)$ in $\Delta_{r,\rho}$, with $H$ satisfying
$$
H(x,u,0)=\oo{\modu x^k \norm u^m+ \modu x^N \norm u}.
$$
\end{proposition}

\begin{proof}
Thanks to \eqref{eq:acca}, it is possible to write $H(x,u,v)$ in a more convenient form. Indeed, for any  $N\in \N$ we have 
\begin{equation}\label{eq:acca0}
H(x,u,0)=\sum_{k\leq s\leq N,\, t\in E_s} c_{s,t}(u)x^s (\log x)^t +\oo{\norm{u} \modu x^{N} \modu{\log\modu x}^{h_N}},
\end{equation}
for some $h_N\in \N$ depending on $N$, where for any $s$ we define $E_s$ as the (finite) set of integers $t$ so that the series above contains the term $x^s(\log x)^t$, and where $c_{s,t}(u)$ are holomorphic in $\|u\|\le \rho$ and $c_{s,t}(0)\equiv 0$.  
We shall prove by induction on $s$, $t$ and the order of $c_{s,t}(u)$, that, if $s\leq N$, using changes of coordinates of the form $\tilde v= v- \phe(x,u)$, it is possible to obtain $c_{s,t}$ of order at least $m$. We shall need the following reformulation of Lemma~{2.3} of \cite{hakim2} for the case $k+1\ge 2$.

\begin{lemma}
Let $\Phi\in\Diff(\C^p,0)$ be a tangent to the identity germ of order $k+1\ge 2$ as in \eqref{eq:eqprinc}, with $[V]=[1:0]$ so that $A(v)=Diag(A,B)$ satisfies
\begin{equation*}
\begin{array}{l l}
\Re\lambda_j >\alpha >0 ,\,\textrm{ for any $\lambda_j$ eigenvalue of $A$} \\
\Re\mu_j \leq 0, \quad\quad\textrm{ for any $\mu_j$ eigenvalue of $B$}.
\end{array}
\end{equation*}
Let $H$ be so that \eqref{eq:acca0} holds, let $\bar s$ be the smallest integer in \eqref{eq:acca0}, and let $m\geq 2$; for such an $\bar s$, let $\bar t$ be the greatest integer in $E_{\bar s}$ so that $c_{\bar s,\bar t}$ has order $\bar d$ less than $m$. Then there exists a polynomial map $P(u)$, homogeneous of degree $\bar d$, with values in $\C^l$, such that, after changing $v$ in 
\begin{equation*}
\tilde v = v - x^{\bar s -k}(\log x)^{\bar t} P(u),
\end{equation*}
$c_{\bar s,\bar t}(u)$ has order greater than $\bar d$.
\end{lemma}

\begin{proof}
Since $c_{\bar s,\bar t}(u)$ has order $\bar d$, we can write 
\begin{equation*}
c_{\bar s,\bar t}(u)= Q(u)+ \oo{\norm u ^{\bar d+1}},
\end{equation*}
where $Q(u)$ is a homogeneous polynomial of degree $\bar d$, and takes values in $\C^l$. Moreover, the term $c_{\bar s,\bar t}(u)x^{\bar s}(\log x)^{\bar t}$ in \eqref{eq:acca0} is
\begin{equation}
H(x,u,0)=c_{\bar s, \bar t}(u)x^{\bar s}(\log x)^{\bar t}+\sum_{k\leq s\leq N,\atop t\in E_s,\, (s,t)\neq (\bar s, \bar t)} c_{s,t}(u)x^s (\log x)^t +\oo{\norm{u} \modu x^N \modu{\log\modu x}^{h_N}}.
\end{equation}
Using a change of coordinates of the form
\begin{equation*}
\tilde v = v - x^{\bar s - k}(\log x )^{\bar t}P(u),
\end{equation*}
with $P(u)$ homogeneous polynomial, we have 
\begin{equation*}
\begin{split}
\tilde v _1 &=v_1 - x_1^{\bar s -k}(\log x_1)^{\bar t}P(u_1)\\
&=(I_l-x^k B)(\tilde v + x^{\bar s -k}(\log x )^{\bar t}P(u))+ H(x,u,v)- x_1^{\bar s -k}(\log x_1)^{\bar t}P(u_1)\\
&=(I_l-x^k B)\tilde v +\tilde H (x,u,\tilde v),
\end{split}
\end{equation*}
where $\tilde H (x,u,\tilde v):= (I_l-x^k B)x^{\bar s -k}(\log x )^{\bar t}P(u)+H(x,u,\tilde v +x^{\bar s -k}(\log x)^{\bar t}P(u))- x_1^{\bar s -k}(\log x_1)^{\bar t}P(u_1)$.
Expanding $\tilde H(x,u,0)$ we obtain 
\begin{equation}\label{eq:accatilde}
\begin{split}
\tilde H(x,u,0)&=  x^{\bar s -k}(\log x )^{\bar t}P(u) -B x^{\bar s}(\log x)^{\bar t}P(u)+ {H(x,u, x^{\bar s -k}(\log x)^{\bar t}P(u))}\\
&\quad -{x_1^{\bar s -k}(\log x_1)^{\bar t}P(u_1)}.
\end{split}
\end{equation}
We have
\begin{equation*}
\begin{split}
H(x,u, x^{\bar s -k}(\log x)^{\bar t}P(u))
&=Q(u)x^{\bar s}(\log x)^{\bar t}+ \oo{\norm u ^{\bar d+1}x^{\bar s}(\log x)^{\bar t}, \norm u ^{\bar d} x^{\bar s}(\log x)^{\bar t -1}, \norm{u} \modu x^N \modu{\log\modu x}^{h_N}},
\end{split}
\end{equation*}
and
\begin{equation*}
\begin{split}
x_1^{\bar s -k}(\log x_1)^{\bar t}P(u_1)
&=\left[ x^{\bar s -k}- \frac{\bar s-k}{k}x^{\bar s}+\oo{\norm u x^{\bar s}, x^{\bar s +k}\log x}\right](\log x)^{\bar t}P(u_1)+\oo{x^{\bar s}(\log x)^{\bar t}}P(u_1)\\
&=x^{\bar s -k}(\log x)^{\bar t}P(u)- x^{\bar s -k}(\log x)^{\bar t}\langle \mbox{\rm grad\,}P ; x^k A u \rangle -\frac{\bar s-k}{k}x^{\bar s}(\log x)^{\bar t}P(u)\\
&\quad\quad\quad+\oo{x^s(\log x)^{\bar t},\norm u x^{\bar s}(\log x)^t, x^{\bar s +k}(\log x)P(u_1)},
\end{split}
\end{equation*}
where we used
\begin{equation*}
\begin{split}
P(u_1)
&=P((I_d-x^k A)u)+\oo{x^{\bar s}}\\
&= P(u)+\langle \mbox{\rm grad\,}P , -x^k A u\rangle +\oo{x^{2k},x^{\bar s}}.
\end{split}
\end{equation*}
It is then clear that the terms of order $\bar s-k$ in \eqref{eq:accatilde} cancel each other, whereas we can put in evidence the terms of order $\bar s$ in $x$ and of order $\bar d$ in $u$. In particular, the $l$ homogeneous polynomials of degree $\bar d$ of $\tilde c_{\bar s, \bar t}$ in \eqref{eq:accatilde} vanish identically if and only if $P$ satisfies the following $l$ equations
\begin{equation}
\langle \mbox{\rm grad\,}P_i , A u \rangle - \left( \left( B-\frac{\bar s -k}{k}I_l\right) P(u) \right)_i = -Q_i(u)\quad  i=1,\ldots, l .
\end{equation}
These equations form a square linear system in the coefficients of $P$. Therefore, to prove that such a system has a solution it suffices to prove that
\begin{equation}
\langle \mbox{\rm grad\,}P_i , A u \rangle - \left( \left( B-\frac{\bar s -k}{k}I_l\right) P(u) \right)_i = 0\quad   i=1,\ldots, l~~\Longrightarrow~~P=0 .
\end{equation}
Moreovere, since $B$ is in Jordan normal form, if we denote by $\epsilon_{i,i+1}$ the elements out of the diagonal, we can rewrite the previous equation as 
\begin{equation}
\frac{\partial P_i}{\partial u_1}(Au)_1 + \cdots+ \frac{\partial P_i}{\partial u_q}(Au)_q - \left(\mu_i- \frac{\bar s -k}{k} \right) P_i -\epsilon_{i,i+1}P_{i+1}=0,
\end{equation}
recalling that, for any $1\leq i < l$, we have $\epsilon_{i,i+1}=0$ or $1$, and $\epsilon_{l,l+1}=0$. Therefore, arguing by decreasing induction over $i$ from $l$ to $1$, we reduce ourselves to solve
\begin{equation}
\frac{\partial R}{\partial u_1}(Au)_1 + \cdots+ \frac{\partial R}{\partial u_d}(Au)_d - \left(\mu_i- \frac{\bar s -k}{k} \right) R =0 \Longrightarrow R=0,
\end{equation}
for a homogeneous polynomial $R$ of degree $\bar d$. By Euler formula, we know that
\begin{equation*}
R={\bar d}^{-1}\left[ \frac{\partial R}{\partial u_1}u_1+\cdots+\frac{\partial R}{\partial u_d}u_d \right].
\end{equation*}
We can therefore reduce ourselves to solve
\begin{equation}\label{eq:euler_formula}
\frac{\partial R}{\partial u_1}(C_i u)_1+\cdots+\frac{\partial R}{\partial u_d}(C_i u)_d= 0 \Longrightarrow R=0 , 
\end{equation}
where $C_i= A- {(\mu_i - \bar s + k)}{\bar d}^{-1}I_d$ is invertible, since from our hypotheses $\Re(\alpha - \frac{\mu_i - \bar s + k}{\bar d})>0$. We prove \eqref{eq:euler_formula} with a double induction, on the dimension $d$ and on the degree $\bar d$ of $R$. For any degree $\bar d$, if $d=1$, then there exists a constant $K_i$ so that $R=K_i u_1 ^{\bar d}$; then, since $\alpha - \frac{\mu_i - \bar s+ k}{\bar d}\neq 0$, we have
\begin{equation*}
\frac{\partial R}{\partial u_1}\left(\alpha - \frac{\mu_i - \bar s+ k}{\bar d} \right) u_1 = 0 \Longrightarrow \bar d K_i u_1^{\bar d}= 0 \Longrightarrow K_i=0,
\end{equation*}
implying $R=0$. Similarly, for any dimension $d$, if $\bar d=1$, then there exist constants $a_1, \ldots, a_d$ so that $R=a_1 u_1+ \cdots + a_d u_d$; hence
\begin{equation*}
a_1(C_i u)_1 + \cdots + a_d (C_i u)_d = 0 \Longrightarrow a_1=\cdots =a_d= 0 \Longrightarrow R=0.
\end{equation*}
Assume, by inductive hypothesis, that \eqref{eq:euler_formula} holds for any pair $(d-1,\bar  d) $ and $(d,\bar d-1)$, with $d>1$ and $\bar d>1$, and we shall prove that \eqref{eq:euler_formula} holds also for $(d, \bar d)$.
Assume that $\langle \mbox{\rm grad\,}R , C u \rangle =0 $ for a certain homogeneous polynomial $R$ of degree $\bar d$ in $q$ variables. By inductive hypothesis, setting $\tilde R(u_1,\ldots, u_{d-1}):= R(u_1,\ldots, u_{d-1},0)$, we have 
\begin{equation*}
\langle \mbox{\rm grad\,}\tilde R , C\cdot (u_1, \ldots, u_{d-1}, 0) \rangle =0  \Longrightarrow \tilde R =0,  
\end{equation*}
and so $R(u)= u_d S(u)$, with $S$ homogeneous polynomial of degree $\bar d-1$ in $d$ variables. Therefore
\begin{equation*}
\frac{\langle \mbox{\rm grad\,}R , C u \rangle}{u_d}= 0 \Longrightarrow \langle \mbox{\rm grad\,}S , C u \rangle + \left(\lambda_d -\frac{\mu_i - \bar s+ k}{\bar d}\right)S = 0.
\end{equation*}
Again, by Euler's formula, we can then write
\begin{equation*}
\langle \mbox{\rm grad\,}S , C^{\prime} u \rangle = 0,
\end{equation*}
with $C^{\prime}= C + \frac{\lambda_d -({\mu_i - \bar s+ k})/{\bar d}}{\bar d}I_d$, and applying the inductive hypothesis, we obtain $S=0$, and thus $R=0$. 
\end{proof}

We shall now apply the previous Lemma, for the integers $s$ and $t$, until $c_{s,t}(u)$ has order at least $m$. Then either $E_s=\emptyset$, or the greatest $t^\prime$ in $E_s$ is less than $t$. In this last case, we can apply again the Lemma, with integers $s$ and $t^\prime$, until $E_s=\emptyset$. We can then apply the Lemma with $s+1$ instead of $s$, until we have $s+1=N$.
This proves the proposition. 
\end{proof}

\sm We shall prove, analogously to the way we found a parabolic curve, that we can find parabolic manifolds as fixed points of a certain operator between spaces of functions, proving the following generalization of Theorem~1.6 of \cite{hakim2} for the case $k+1\ge 2$.

\begin{theorem}\label{thm:1.6}
Let $\Phi\in\Diff(\C^p, 0)$ be a tangent to the identity germ of order $k+1\ge 2$. Let $[V]$ be a non-degenerate characteristic direction and let $A=A(V)$ be its associated matrix. If $A$ has exactly $d$ eigenvalues, counted with multiplicity, with strictly positive real parts, then there exists a parabolic manifold of dimension $d+1$, with $0$ on its boundary, and tangent to $\C V \oplus E$ in $0$, where $E$ is the eigenspace associated to the attracting directors, and so that each of its points is attracted to the origin along the direction $[V]$. Moreover, it is possible to find coordinates $(x,u,v)$ in a sector of $\C\times\C^d\times\C^{p-d-1}$ so that the parabolic manifold is locally defined by $\{v=0\}$. 
\end{theorem}

\begin{proof}
We may assume that $\Phi$ is of the form \eqref{eq:eqprinc}, with $[V]=[1:0]$ so that $A(v)=Diag(A,B)$ satisfies
\begin{equation*}
\begin{array}{l l}
\Re\lambda_j >\alpha >0 ,\,\textrm{ for any $\lambda_j$ eigenvalue of $A$} \\
\Re\mu_j \leq 0, \quad\quad\textrm{ for any $\mu_j$ eigenvalue of $B$},
\end{array}
\end{equation*}
and
$$
H(x, u, 0) = O(|x|^k\|u\|^m + |x|^N\|u\|),
$$
with $m, N>0$.

\sm

We shall search for $\phi(x,u)$, holomorphic in a sector
\begin{equation}\label{sec}
S_{\gamma,s,\rho}=\{(x,u)\in \C\times \C^d \mid |{\Im x^k}|\leq \gamma \Re x^k,\, \modu{x} \leq s,\, \norm u \leq \rho\},
\end{equation}
so that, for
\begin{equation*}
\begin{split}
\left\{
\begin{array}{l l}
x_1^{\phi}=f(x,u,\phi(x,u)),\\
u_1^{\phi}=g(x,u,\phi(x,u)),
\end{array}\right. 
\end{split}
\end{equation*}
we have 
\begin{equation}\label{eq:eqrel}
\phi(x_1^{\phi},u_1^{\phi})= h(x,u,\phi(x,u)).
\end{equation}

\me
Repeating the same changes of coordinates performed in the Section 6, we first transform $v_1$, for $\Re x>0$ by setting
$$
w=x^{-kB}v,
$$
and we define $ H_1$ as
\begin{equation*}
w- w_1= x^{-kB} H_1(x,u,v).
\end{equation*}
From the definitions of $x_1$ and $u_1$ in \eqref{eq:eqprinc}, we have
\begin{equation*}
\begin{split}
x_1 ^{-kB} 
= x^{-kB} \left[ \left( I + x^k B \right) + \oo{\norm u x^k, x^{2k} \log x }  \right] ,
\end{split}
\end{equation*}
and
\begin{equation*}
\begin{split}
w_1 &= x^{-kB} \left[ \left( I + x^k B \right) + \oo{\norm u x^k, x^{2k} \log x }  \right] \left[ (I- x^kB)x^{kB}w + H(x,u,v)\right] \\
&= \left( I+ \oo{\norm u x^k, x^{2k} \log x } \right) w + x^{-kB}\left( I + O(x^k) \right) H(x,u,v).
\end{split}
\end{equation*}
Hence $H_1(x,u,v)$ satisfies the same estimates as $H(x,u,v)$:
\begin{equation*}
H_1(x,u,v)= \oo{\norm u ^2 x^k, \norm u x^{k+1}\log x},	
\end{equation*}
and
\begin{equation}\label{8.13}
H_1(x,u,0)= \oo{\modu x^k \norm u ^m +\modu x ^N \norm u }.	
\end{equation}
Therefore \eqref{eq:eqrel} is equivalent to  
\begin{equation}\label{eq:relfunz2}
x^{-kB} \phi(x,u) - x_1^{-kB} \phi(x_1^\phi , u_1^\phi)= x^{-kB} H_1(x,u,\phi (x,u)).
\end{equation}

\me\no{\bf Operator $\T$.}
Let $\{(x_n, u_n)\}$ be the iterates defined by 
\begin{equation*}
\begin{split}
\left\{\begin{array}{l l}
x_1^\phi = f(x,u, \phi(x,u))= x-\frac{1}{k}x^{k+1} + F(x,u,\phi(x,u)),\\
u_1^\phi = g(x,u, \phi(x,u))= \left( I_d - x^k A\right) u + G(x,u,\phi(x,u)),
\end{array}\right.
\end{split}
\end{equation*}
with $f$ and $g$ as in \eqref{eq:eqprinc}, and $\phi$ holomorphic from the sector $S_{\gamma, s, \rho}$, defined in \eqref{sec}, to $\C^{p-d-1}$. Now we consider the operator
\begin{equation*}
\T \phi(x,u):= x^{kB} \sum_{n=0}^{\infty} x_n^{-kB} H_1 (x_n, u_n, \phi(x_n, u_n)).
\end{equation*}
We shall prove that this operator, restricted to a suitable closed subset $\mathcal F$ of the Banach space of bounded holomorphic maps $\phi:S_{\gamma,s,\rho}\to \C^{p-d-1}$, is a contraction. Then there exists a unique fixed point in $\mathcal F$, and, by the definition of $\T$, such a fixed point will be a solution of \eqref{eq:relfunz2}. 

We shall proceed as follows:
\begin{enumerate}
\item we shall prove that there exists a constant $K_0>0$ such that
\begin{equation}\label{eq:ccond1}
\begin{split}
\norm{\phi(x,u)}&\leq  K_0 \left( \norm u ^m +\modu x^{N-k} \norm u \right) \Longrightarrow \norm{\T\phi(x,u)}\leq K_0 \left( \norm u ^m +\modu x^{N-k} \norm u \right);
\end{split}
\end{equation}
\item we shall prove that if $K_0>0$ satisfies \eqref{eq:ccond1}, then there exist positive constants $K_1$ and $K_2$ such that
\begin{equation}\label{eq:ccond2}
\begin{split}
\begin{split}
\left\{\!\!\!
\begin{array}{l l}
\modu{\frac{\partial \phi}{\partial x}}\leq K_1 ( \norm u ^m \modu x^ {-1}+ \norm u \modu x ^{N-k-1} ), \medskip\\
\modu{\frac{\partial \phi}{\partial u}}\leq K_2 ( \norm u ^{m-1} + \modu x ^{N-k} ),
\end{array}\right.
\end{split}
\Rightarrow
\begin{split}
\left\{\!\!\!
\begin{array}{l l}
\modu{\frac{\partial \T\phi}{\partial x}}\leq K_1 ( \norm u ^m \modu x^ {-1}+ \norm u \modu x ^{N-k-1} ), \medskip  \\
\modu{\frac{\partial \T\phi}{\partial u}}\leq K_2 ( \norm u ^{m-1} + \modu x ^{N-k} );
\end{array}\right. 
\end{split}
\end{split}
\end{equation}
\item considering the Banach space $(F_0, \norm{\cdot}_0)$ defined as
\begin{equation*}
F_0 = \{\phi:S_{\gamma,s,\rho}\to \C^{p-d-1} \mid \norm{\phi}_0 <+\infty  \}, 
\end{equation*}
with the norm
$$
\norm{\phi}_0 := \sup_{x,u}\left\{ \frac{\norm{\phi(x,u)}}{\norm u ^m + \modu x ^{N-k}\norm u} \right\},
$$
we shall prove that the subset $\mathcal F$ of $F_0$, given by the maps $\phi$ satisfying  \eqref{eq:ccond1} and \eqref{eq:ccond2} with the constants $K_0,\,K_1$ and $K_2$ we found in 
$(1)$ and $(2)$ is closed;
\item we shall finally show that $\T$ is a contraction.
\end{enumerate}

We first prove the following analogous of Proposition~3.2 of \cite{hakim2}.

\begin{proposition}\label{prop:cond1}
If $m$ and $N$ are integers so that $H_1$ satisfies \eqref{8.13}, then there exists a positive constant $K_0$ such that, if
\begin{equation}\label{eq:cond1}
\norm{\phi(x,u)}\leq K_0 \left( \norm u ^m + \modu x^{N-k} \norm u \right) 
\end{equation} 
then
\begin{enumerate}
\item the series defining the operator $\T$ is uniformly convergent in $S_{\gamma, s, \rho}\cap \{(x,u)\in\C^p\mid \norm u \modu x^{-k\alpha}\le 1 \}$;
\item also $\norm{\T \phi(x,u)}$ satisfies the same inequality
\begin{equation*}
\norm{\T \phi(x,u)}\leq K_0 \left( \norm u ^m + \modu x^{N-k} \norm u \right) .
\end{equation*}
\end{enumerate} 
\end{proposition}
\begin{proof}
Since all the eigenvalues of $A$ have strictly positive real parts, as we saw in Theorem \ref{th:7.1}, for any $(x,u)\in S_{\gamma, s, \rho}$ we have 
\begin{equation*}
\lim_{n\to\infty}\norm{u_n}\modu{x_n}^{-k\alpha}=0, 
\end{equation*}
where $\alpha>0$ is strictly less then the real parts of the eigenvalues of $A$.
Therefore, without loss of generality, we may assume that $\norm{u}\modu{x}^{-k\alpha}$ is bounded by $1$. Let $\beta<k\alpha$ be a positive real number so that each eigenvalue $\mu_j$ of $B$ satisfies $\Re \mu_j<\beta$. By Lemma \ref{lemma4.1}, this implies that there exists a constant $C_1>0$ so that
\begin{equation*}
\norm{x^{kB}x_n^{-kB}}\leq C_1 \modu{\frac{x_n}{x}}^{-k\beta}.
\end{equation*}
Moreover, choosing $\gamma ,\,s,\,\rho$ small enough, if $(x,u)\in S_{\gamma,s,\rho}$, then
\begin{equation*}
\modu{x_n^k}\leq \frac{2}{n},\quad \norm{u_n}\leq \norm u \modu x^{-k\alpha}\modu{x_n}^{k\alpha}.
\end{equation*}
Thanks to the hypotheses on $H_1(x,u,v)$, there exist positive constants $K_1$ and $K_2$ so that\begin{equation}\label{eq:eq_3.16}
\begin{split}
\norm{H_1(x,u,v)}\!\leq\! K_1\!\!\left( \norm u ^m \!\modu x^k\! +\! \modu x ^N \!\!\norm u \right) \!\!+\! {K_2 \!\!\left(\norm v \modu x ^{k+1}\mlog x ^q +\norm v ^2 \modu x^k + \norm u \norm v \modu x^k \right)},
\end{split}
\end{equation}
for a certain $q\in\N$. 

\no Let us assume that $\norm{\phi(x,u)}\leq K \left( \norm u ^m + \modu x ^{N-k}\norm u \right)$ for a constant $K>0$. For $v=\phi(x,u)$, we have
\begin{equation*}
\begin{split}
\norm v \modu x ^{k+1}\mlog x ^q +\norm v ^2 \modu x^k + \norm u \norm v \modu x^k=\oo{\norm u + \modu x \mlog x ^q}\left( \modu x^k \norm u ^m + \modu x ^N \norm u \right).
\end{split}
\end{equation*}
Hence, taking $s$ and $\rho$ small enough, we have 
\begin{equation*}
\begin{split}
\norm{H_1(x,u,\phi(x,u))}\leq (K_1 + 1)\left( \modu x^k \norm u ^m+ \modu x ^N \norm u \right),
\end{split}
\end{equation*}
and therefore
\begin{equation*}
\begin{split}
\norm{\T \phi(x,u)} &\leq \sum_{n=0}^\infty \left\|{ \left( \frac{x_n}{x} \right)^{-kB} H_1(x_n,u_n,\phi(x_n, u_n)) }\right\|\\
&\leq (K_1 + 1)\sum_{n=0}^\infty \modu{\frac{x_n}{x}}^{-k\beta} \left( \modu{x_n}^k\norm{u_n}^m + \modu{x_n}^N \norm{u_n} \right) \\
&\leq (K_1 + 1)\sum_{n=0}^\infty \modu{\frac{x_n}{x}}^{-k\beta} \left( \modu{x_n}^{k+ k\alpha m} \norm u ^m \modu x ^{-k\alpha m }+ \modu{x_n}^{N+k\alpha} \norm u \modu x ^{-k\alpha}\right) .
\end{split}
\end{equation*}
Since $k\alpha > \beta$, the series is normally convergent in the set $\{\norm u \modu x ^{-k\alpha}\leq 1\}$. By Corollary \ref{cor:4.3}, there exists a positive constant $K_0$, depending only on $H_1$, so that
\begin{equation*}
\norm{ \T \phi (x,u)}\leq K_0 \left( \norm u ^m + \modu x^{N-k} \norm u \right).
\end{equation*}
Then, to conclude the proof it suffices to take $K=K_0$.
\end{proof}

Let $\mathcal F_0$ be the set of holomorphic maps from $S_{\gamma,s,\rho}$ to $\C^p$, satisfying \eqref{eq:cond1} with the constant $K_0$ of Proposition \ref{prop:cond1}. We just proved that $\T$ maps $\mathcal F_0$ into itself. Since we want $\T$ to be a contraction, we need to restrict this set. We first do it by restricting the domain of definition of the maps in $\mathcal F_0$.

\me\no{\bf Choice of the domain of definition $\mathcal D$.}
In the following, instead of $S_{\gamma,s,\rho}$, we shall use the following domain of definition for the maps $\phi$
\begin{equation*}
\mathcal D := S_{\gamma, s,\rho}\cap \{ (x,u) \in\C^r\mid \norm u \modu x^{-k\alpha}\leq 1 \}, 
\end{equation*}
and we shall denote with $\mathcal F_0$ the set of maps $\phi:\mathcal D \to \C^r$ satisfying \eqref{eq:cond1}. 
We shall prove a result analogous to Proposition \ref{prop:cond1} for the partial derivatives of $\phi$. To do so, we shall need bounds for the series
\begin{equation*}
\sum_{n=0}^\infty \normu{\frac{\partial}{\partial x} \left\{\left( \frac{x_n}{x} \right)^{-kB}H_1(x_n, u_n, \phi(x_n, u_n)) \right\}},
\end{equation*}
and
\begin{equation*}
\sum_{n=0}^\infty \normu{\frac{\partial}{\partial u} \left\{\left( \frac{x_n}{x} \right)^{-kB}H_1(x_n, u_n, \phi(x_n, u_n)) \right\}}.
\end{equation*}
We thus have to control the partial derivatives $\modu{\frac{\partial x_n}{\partial x}}$, $\normu{\frac{\partial u_n}{\partial x}}$, $\normu{\frac{\partial x_n}{\partial u}}$, and $\normu{\frac{\partial u_n}{\partial u}}$.

Following Lemma~3.5 of \cite{hakim2}, we have the following estimates.

\begin{lemma}\label{lemma:8.7}
Let $\delta = min\{k \alpha, k \}$, and let $\epsilon>0$, with $\epsilon < \delta$. Then, for $\gamma$, $s$ and $\rho$ small enough, we have the following inequalities in $\mathcal D$:
\begin{equation*}
\modu{\frac{\partial x_n}{\partial x}}\leq \modu{\frac{x_n}{x}}^{1+\delta -2\epsilon},  \quad \normu{\frac{\partial u_n}{\partial x}}\leq \frac{\norm u \modu{x_n}^{\delta -\epsilon}}{\modu x ^{1+\delta -2 \epsilon}},
\quad\normu{\frac{\partial x_n}{\partial u}}\leq \frac{\modu{x_n}^{1+\delta -2\epsilon}}{\modu x ^{\delta - \epsilon}},
~~{\rm and}~~\normu{\frac{\partial u_n}{\partial u}}\leq \modu{\frac{x_n}{x}}^{\delta -\epsilon}.
\end{equation*}
\end{lemma}

\begin{proof}
We argue by induction over $n$. If $n=1$, deriving $x_1=x-\frac{1}{k}x^{k+1}+ \oo{x^{2k+1}, \norm u x^{k+1}\log x}$ and $u_1=(I-x^k A)u +\oo{\norm u ^2 x^k , \norm u x^{k+1}\log x}$ with respect to $x$ and $u$, we obtain 
\begin{equation*}
\modu{\frac{\partial x_1}{\partial x}}=\modu{1-\frac{k+1}{k}x^k+ o(x^k)}\leq \modu{\frac{x_1}{x} }^{k+1-2\epsilon},
\end{equation*}
because $\modu{\frac{\modu{x_1}}{\modu x}}^{k+1-2\epsilon}=\modu{1-\frac{k+1-\epsilon}{k}x^k+o(x^k)}$, and
\begin{equation*}
\normu{\frac{\partial u_1}{\partial x}}\leq K \norm u\modu x^{k-1}.
\end{equation*}
Moreover
\begin{equation*}
\normu{\frac{\partial x_1}{\partial u}}\leq K \modu x ^{k+1}\leq \modu{\frac{\modu{x_1}^{1+\delta -2\epsilon}}{\modu x^{\delta -\epsilon}}},~~{\rm and}~~\normu{\frac{\partial u_1}{\partial u}}\leq \modu{1-\alpha x^k}\leq \modu{\frac{x_1}{x}}^{\delta-\epsilon},
\end{equation*}
for $\gamma$, $s$ and $\rho$ small enough. By the definition of $\Phi$ we deduce 
\begin{equation*}
\modu{\frac{\partial x_{n+1}}{\partial x}}\leq \modu{1- \frac{k+1}{k}x_n^k +o(x_n^k)}\modu{\frac{\partial x_n}{\partial x}}+ K \modu{x_n}^{k+1}\normu{\frac{\partial u_n}{\partial x}},
\end{equation*}
and
\begin{equation*}
\normu{\frac{\partial u_{n+1}}{\partial x}}\leq K \norm{u_n}\modu{\frac{\partial x_n}{\partial x}}+ \modu{1-\alpha x_n^k} \normu{\frac{\partial u_n}{\partial x}}.
\end{equation*}
Hence, by inductive hypothesis, 
\begin{equation*}
\begin{split}
\modu{\frac{\partial x_{n+1}}{\partial x}}&\leq \modu{\frac{x_n}{x}}^{1+\delta -2 \epsilon}\left( 1- \frac{k+1}{k}\Re x_n^k +o(x_n^k)+ K \norm u \modu{x_n}^{1+\epsilon}\right)\\
&\leq \modu{\frac{x_{n+1}}{x}}^{1+\delta-2\epsilon}= \modu{1- \frac{1+\delta -2\epsilon}{k}x_n^k+ o(x_n^k) },
\end{split}
\end{equation*}
because $1+\delta -2\epsilon < \frac{k+1}{k}$. On the other side, using the inductive hypothesis and the inequality $\norm{u_n}\leq \norm u \modu x ^{-k\alpha}\modu{x_n}^{k\alpha}$, we obtain 
\begin{equation*}
\normu{\frac{\partial u_{n+1}}{\partial x}}\leq \frac{\norm u }{\modu x ^{1+\delta -2\epsilon}} \modu{x_n}^{\delta - \epsilon} \left( 1-\alpha \Re x_n^k + o(x_n^k) + K \modu x^{-k\alpha}\modu{x_n}^{1+k\alpha -\epsilon}\right) ,
\end{equation*}
which is less than $\frac{\norm u }{\modu x^{1+\delta -2\epsilon }}\modu{x_{n+1}}^{\delta -\epsilon}$, because $\delta -\epsilon < k\alpha$. Arguing analogously by induction, we prove also the inequalities for the partial derivatives with respect to $u$. In fact,
\begin{equation*}
\begin{split}
\normu{\frac{\partial x_{n+1}}{\partial u}}
&\leq \modu{1- \frac{k+1}{k}x_n^k +o(x_n^k)}\modu{\frac{\partial x_n}{\partial u}}+ K \modu{x_n}^{k+1}\norm{\frac{\partial u_n}{\partial u}}\\
&\leq \modu{ \frac{\modu{x_n}^{k+1-2\epsilon}}{\modu x^{\delta -\epsilon}} }\left[ 1-\frac{k+1}{k}x_n^k +o(x_n^k)+ K\modu{x_n}^{\delta +\epsilon} \right]\\
&\leq \frac{\modu{x_n}^{k+1-2\epsilon}}{\modu x^{\delta -\epsilon}},
\end{split}
\end{equation*}
because $\delta + 1- \epsilon < k+1$, and
\begin{equation*}
\begin{split}
\normu{\frac{\partial u_{n+1}}{\partial u}}&\leq K \norm{u_n}\normu{\frac{\partial x_n}{\partial u}}+ \modu{1-\alpha x_n^k}\normu{\frac{\partial u_n}{\partial u}}\leq \modu{\frac{x_n}{x}}^{\delta -\epsilon}\left[ 1-\alpha \Re x_n^k +o(x_n^k) K\norm u \frac{\modu{x_n}^{2k+k\alpha -\delta}}{\modu x^{k\alpha}} \right]\\
&\leq \modu{\frac{x_{n+1}}{x}}^{\delta - \epsilon},
\end{split}
\end{equation*}
because $\delta-\epsilon < k\alpha$. This concludes the proof.
\end{proof}

We can now prove the following reformulation of Proposition~3.4 of \cite{hakim2} for the case $k+1\ge 2$.

\begin{proposition}\label{prop:cond2}
Let $\phi $ be in $\mathcal F_0$. There exist positive constants $K_1$ and $K_2$ so that, if we have
\begin{equation}\label{eq:cond2}
\left\{\begin{array}{l l l}
\modu{\frac{\partial \phi}{\partial x}}\leq K_1 \left( \norm u ^m \modu x^{-1} + \modu x ^{N-k-1}\norm u \right),\\
\vspace{-0.2cm}\\
\modu{\frac{\partial \phi}{\partial u}}\leq K_2 \left( \norm u ^{m-1} + \modu x ^{N-k} \right),
\end{array}\right.
\end{equation}
than the same inequalities hold for $\norm{\frac{\partial \T \phi}{\partial x}}$ and $\norm{\frac{\partial \T \phi}{\partial u}}$.
\end{proposition}
\begin{proof}
We first deal with the partial derivative of $H_1$. There exist positive constants $C_1$ and $C_2$ so that
\begin{equation*}
\begin{split}
\normu{H_1(x,u,v)}&\leq C_1\left( \norm u ^m \modu x ^k + \modu x^N \norm u  \right) +C_2 \left( \norm v \modu x^{k+1}\mlog x^q +\norm v ^2 \modu x^k +\norm u \norm v \modu x ^k \right) . 
\end{split}
\end{equation*}
Then there exist positive constants $C_3$ and $C_4$ such that
\begin{equation*}
\begin{split}
\normu{\frac{\partial H_1}{\partial x} }&\leq C_3 \left( \norm u ^m \modu x^{k-1} + \modu x^{N-1}\norm u \right) + C_4\left( \norm v \modu x^k \mlog x^q  +\norm v ^2 \modu x ^{k-1} +\norm u \norm v \modu x^{k-1} \right) . 
\end{split}
\end{equation*}
On the other side,
\begin{equation*}
\begin{split}
\normu{\frac{\partial H_1}{\partial u} }&\leq C_5 \left( \norm u ^{m-1}\modu x^k + \modu x^N\right) + C_6 \norm v \modu x^k ,
\end{split}
\end{equation*}
for some positive constants $C_5$ and $C_6$. Finally, there exist positive constants $C_7$ and $C_8$ such that
\begin{equation*}
\begin{split}
\normu{\frac{\partial H_1}{\partial v} }&\leq C_7 \left( \modu x^{k+1} \mlog x ^q + \norm v \modu x^k + \norm u \modu x ^k \right) \leq C_8\modu x^k .
\end{split}
\end{equation*}
Let us assume that there exist positive constants $K$ and $K^{\prime}$ such that
\begin{equation*}
\left\{\begin{array}{l l l}
\normu{\frac{\partial \phi}{\partial x}}\leq K \left( \norm u ^m \modu x^{-1} + \modu x ^{N-k-1}\norm u \right),\\
\vspace{-0.2cm}\\
\normu{\frac{\partial \phi}{\partial u}}\leq K^\prime \left( \norm u ^{m-1} + \modu x ^{N-k} \right).
\end{array}\right.
\end{equation*} 
Then we have
\begin{equation*}
\begin{split}
\normu{\frac{\partial \T\phi}{\partial x}}&\leq \sum_{n=0}^\infty \normu{\frac{\partial}{\partial x} \left\{\left( \frac{x_n}{x} \right)^{-kB}H_1(x_n, u_n, \phi(x_n, u_n)) \right\}}\\
&\leq \sum_{n=0}^\infty \Bigg[{\normu{ \frac{\partial }{\partial x}\left( \left( \frac{x_n}{x} \right)^{-kB}\right)} \normu{ H_1(x_n,u_n, \phi(x_n,u_n)) }}\\
&\quad\quad \quad + \normu{ \left(\frac{x_n}{x}\right)^{-kB}  } \Bigg[ {\left\|{\frac{\partial H_1}{\partial x}}\right\|\modu{\frac{\partial x_n}{\partial x}}} + {\left\|{\frac{\partial H_1}{\partial u}}\right\|\modu{\frac{\partial u_n}{\partial x}} }\Bigg]\\
&\quad \quad \quad +\normu{ \left(\frac{x_n}{x}\right)^{-kB}  } {\left\|{\frac{\partial H_1}{\partial v}}\right\| \left[ \modu{\frac{\partial\phi}{\partial x}}\modu{\frac{\partial x_n}{\partial x}} + \modu{\frac{\partial \phi}{\partial u}}\modu{\frac{\partial u_n}{\partial x}} \right]} \Bigg].
\end{split}
\end{equation*}
We now use Lemma \ref{lemma:8.7} to give estimates. We have
\begin{equation*}
\begin{split}
\left\|\frac{\partial}{\partial x}\left( \left( \frac{x_n}{x} \right)^{-kB}\right)\right\|& \normu{ H_1(x_n,u_n, \phi(x_n,u_n)) } \\
&\leq \norm{-kB} \modu{\frac{x_n}{x}}^{-k\beta}\left[ \frac{1}{\modu{x_n}}\modu{\frac{x_n}{x}}^{1+\delta-2\epsilon}+ \frac{1}{\modu x} \right]
 \left( \norm{u_n}^m \modu{x_n}^{-1}+ \modu{x_n}^{N-k-1}\norm{u_n} \right)\\
 &\quad \times \modu{x_n}^{k+1}\left[ C_1 +C_2 K_0 \left( \modu{x_n}\mlog{x_n}^q + \norm{\phi(x_n,u_n)}+ \norm{u_n}\right) \right]. 
\end{split}
\end{equation*}
Similarly
\begin{equation*}
\begin{split}
\left\|{\frac{\partial H_1}{\partial x}}\right\|\modu{\frac{\partial x_n}{\partial x}}
&\leq \left( \norm{u_n} ^m \modu{x_n}^{-1} + \modu{x_n}^{N-k-1}\norm{u_n} \right) \modu{x_n}^k\\
&\quad \times\Bigg[ C_3+ C_4K_0\left( \modu{x_n} \mlog{x_n}^q  +\norm{\phi(x_n,u_n)}+\norm{u_n} \right)\Bigg] \modu{\frac{x_n}{x}}^{1+\delta -2\epsilon},
\end{split}
\end{equation*}
and
\begin{equation*}
\begin{split}
\left\|{\frac{\partial H_1}{\partial u}}\right\|\modu{\frac{\partial u_n}{\partial x}} 
&\leq  \left( \norm{u_n} ^m \modu{x_n}^{-1} + \modu{x_n}^{N-k-1}\norm{u_n} \right) \modu{x_n}^k\Bigg[ C_5 \frac{\modu{x_n}}{\norm{u_n}} + C_6 \modu{x_n}\Bigg]\frac{\norm u \modu{x_n}^{\delta -\epsilon}}{x^{1+ \delta -2\epsilon}}.
\end{split}
\end{equation*}
Finally 
\begin{equation*}
\begin{split}
\left\|{\frac{\partial H_1}{\partial v}}\right\| \left[ \modu{\frac{\partial\phi}{\partial x}}\modu{\frac{\partial x_n}{\partial x}} + \modu{\frac{\partial \phi}{\partial u}}\modu{\frac{\partial u_n}{\partial x}} \right]
&\leq C_8 \left( \norm{u_n}^m\modu{x_n}^{-1}+ \modu{x_n}^{N-k-1}\norm{u_n} \right) \modu{x_n}^k
\\
&\quad \quad \quad\quad \quad \quad \quad  \quad\times \Bigg[ K\modu{\frac{x_n}{x}}^{1+\delta -2\epsilon} +K^\prime \frac{\norm u \modu{x_n}^{1+\delta -\epsilon}}{\norm{u_n}\modu x^{1+\delta-2\epsilon}}  \Bigg].
\end{split}
\end{equation*}
By Corollary \ref{cor:4.3}, we have the following estimate
$$
\sum_{n=0}^{\infty}\modu{x_n}^{\mu}\modu{\log x_n }^q\leq C_{\mu, q} \modu x^{\mu-k}\modu{\log x}^q,
$$
for a constant $C_{\mu,q}>0$, and hence there exists a positive constant $K_1$, depending only on $H_1$, so that
\begin{equation*}
\begin{split}
\normu{\frac{\partial \T\phi}{\partial x}}
&\leq K_1\left( \norm{u}^m\modu{x}^{-1}+ \modu{x}^{N-k-1}\norm{u} \right).
\end{split}
\end{equation*}
Setting $K=K_1$, we proved the first inequality.  
In a similar way, we estimate $\normu{\frac{\partial \T\phi}{\partial u}}$ obtaining
\begin{equation*}
\begin{split}
\normu{\frac{\partial \T\phi}{\partial u}}
&\leq \sum_{n=0}^\infty {\normu{\frac{\partial}{\partial u}\left(\left(\frac{x_n}{x}\right)^{-kB}\right)} \normu{ H_1(x_n,u_n, \phi(x_n,u_n)) }}+ \normu{\left(\frac{x_n}{x}\right)^{-kB} } {\normu{\frac{\partial }{\partial u}\left( H_1(x_n,u_n,\phi(x_n,u_n)) \right) }}.
\end{split}
\end{equation*}
For the first term , we have
\begin{equation*}
\begin{split}
\normu{\frac{\partial}{\partial u}\left(\left(\frac{x_n}{x}\right)^{-kB}\right)} \normu{ H_1(x_n,u_n, \phi(x_n,u_n)) }&\leq \normu{ -kB \frac{1}{x_n}\frac{\partial x_n}{\partial u}\left( \frac{x_n}{x} \right)^{-kB} }\left( \norm{u_n}^{m-1} +\modu{x_n}^{N-k} \right)\modu{x_n}^k \\
&\quad\times\underbrace{\Big[ C_1 \norm{u_n} + C_2\norm{u_n}\left( \modu{x_n}\mlog{x_n}^q + \norm{\phi(x_n,u_n)} +\norm{u_n}\right) \Big]}_{\tilde K (x_n,u_n)}\\
&\leq \tilde C \modu{\frac{1}{x_n}}\modu{\frac{\partial x_n}{\partial u}}\modu{\frac{x_n}{x}}^{-k\beta}\left( \norm{u_n}^{m-1} +\modu{x_n}^{N-k} \right)\modu{x_n}^k \tilde K (x_n,u_n).
\end{split}
\end{equation*}
The second term contains the partial derivatives of $H_1$ with respect to $x,u$ and $v$, and we have
\begin{equation*}
\begin{split}
\normu{\frac{\partial }{\partial u}\left( H_1(x_n,u_n,\phi(x_n,u_n)) \right) }
&\le \left( \norm{u_n}^{m-1} + \modu{x_n}^{N-k} \right)\\
&\quad\times \modu{x_n}^k \Bigg[ \Big[ C_3\norm{u_n}+ C_4 K_0 \norm{u_n} \left( \modu{x_n} \mlog{x_n}^q  +\norm{\phi(x_n,u_n)}+\norm{u_n} \right)\Big]\frac{\modu{x_n}^{\delta -\epsilon}}{\modu x^{\delta -\epsilon}}\\
&\quad + \Big[ C_5+ C_6 K_0\norm{u_n} \Big]\modu{\frac{x_n}{x}}^{\delta -\epsilon} + C_8 \Big[ K\norm{u_n}\modu{x_n}^{-1}\frac{\modu{x_n}^{1+\delta -2\epsilon}}{\modu x^{\delta -\epsilon}}+ K^\prime\modu{\frac{x_n}{x}}^{\delta -\epsilon}  \Big] \Bigg]\\
&\leq \overline K(x_n,u_n,x) \left( \norm{u_n}^{m-1} + \modu{x_n}^{N-k} \right) \modu{x_n}^k .
\end{split}
\end{equation*}
Therefore 
\begin{equation*}
\begin{split}
\normu{\frac{\partial \T\phi}{\partial u}}
&\leq K_2\left( \norm u^{m-1} + \modu x^{N-k} \right)  ,	
\end{split}
\end{equation*}
and the constant $K_2$ only depends on $H_1$. Taking $K^\prime = K_2$ we conclude the proof.
\end{proof}

\me\no{\bf Definition of $\mathcal F$.}
We are left with finding a suitable subset of maps, such that $\T$ is a contraction. Let $m$ and $N$ be integers satisfying \eqref{8.13}. Let $F_0$ be the Banach space of the holomorphic maps $\phi$, defined on $S_{\gamma,s, \rho}$, such that
\begin{equation*}
\norm{\phi}_0:= \sup_{x,u}\left\{ \frac{\norm{\phi(x,u)}}{\norm u ^m + \modu x ^{N-k} \norm u} \right\}
\end{equation*}
is bounded, endowed with the norm $\norm\phi _0$. Define $\mathcal F$ as the closed subset of $F_0$ given by the maps satisfying \eqref{eq:cond1} and \eqref{eq:cond2}, with the constants $K_0,K_1$ and $K_2$ given by Propositions \ref{prop:cond1} and \ref{prop:cond2}. 

\begin{proposition}
If $\mathcal F$ is the subset defined above, then $\T|_{\mathcal F}$ is a contraction.
\end{proposition}

\begin{proof}
Let $\phi$ and $\psi$ be in $\mathcal F$. We need to control 
\begin{equation*}
S:= \normu{ \sum_{n=0}^\infty \left( \frac{x_n}{x} \right)^{-kB} H_1(x_n,u_n,\phi(x_n,u_n)) -\sum_{n=0}^\infty \left( \frac{x_n^\prime}{x} \right)^{-kB} H_1(x_n^\prime, u_n^\prime,\psi(x_n^\prime, u_n^\prime))},
\end{equation*}
where $(x_n, u_n)$ and $(x_n^\prime, u_n^\prime )$ are the iterates of $(x,u)$ via \eqref{eq:eqprinc} respectively with $\phi $ and $\psi$. We can bound $S$ with the sum of $S_1$ and $S_2$, where
\begin{equation*}
S_1:= \normu{ \sum_{n=0}^\infty \left( \frac{x_n}{x} \right)^{-kB} H_1(x_n,u_n,\phi(x_n,u_n)) -\sum_{n=0}^\infty \left( \frac{x_n}{x} \right)^{-kB} H_1(x_n, u_n,\psi(x_n, u_n))},
\end{equation*}
and
\begin{equation*}
S_2:= \normu{ \sum_{n=0}^\infty \left( \frac{x_n}{x} \right)^{-kB} H_1(x_n,u_n,\psi(x_n,u_n)) -\sum_{n=0}^\infty \left( \frac{x_n^\prime}{x} \right)^{-kB} H_1(x_n^\prime, u_n^\prime,\psi(x_n^\prime, u_n^\prime))}.
\end{equation*}
It is easy to control the term $S_1$. From \eqref{eq:eq_3.16}, we have
\begin{equation*}
S_1\leq C \sum_{n=0}^\infty \modu{\frac{x_n}{x}}^{-k\beta} \left( \norm{u_n}^m + \modu{x_n}^{N-k} \norm{u_n}  \right) \left( \modu{x_n}^{k+1}\modu{\log x_n}^q + \modu{x_n}^k\|u_n\|\right) \norm{\phi -\psi}_0,
\end{equation*}
for some integer $q$.
By Corollary \ref{cor:4.3}, since $\norm{u_n}\leq \norm u \modu{x_n}^{k\alpha}\modu x^{-k\alpha}$, we obtain 
\begin{equation*}
S_1 \leq C^\prime \left( \norm u^m + \modu x^{N-k}\norm u \right) \left( \modu x \modu{\log x}^q  +\norm u \right)\norm{\phi - \psi}_0.
\end{equation*}
To estimate $S_2$, we have to estimate the dependence of $\{(x_n,u_n)\}$ on $\phi$ in \eqref{eq:eqprinc}. We have the following reformulation of Lemma~3.7 of \cite{hakim2} .

\begin{lemma}\label{lemma:lemma3.7}
Let $\delta=\min \{k\alpha , k\}$. Let $\epsilon$ be a positive real number, with $\epsilon < \delta$, and $\Re \lambda_j> \alpha + \epsilon$ for each eigenvalue $\lambda_j$ of $A$. Let $\phi$ and $\psi $ be in $\mathcal F$, and let $\{(x_n, u_n)\}$ and $\{(x_n^\prime, u_n)^\prime\}$ be the iterates via \eqref{eq:eqprinc} associated to $\phi$ and $\psi$. Then for $\gamma, s$ and $\rho$ small enough, the following estimates hold in $S_{\gamma, s , \rho}$:
\begin{equation*}
\modu{x_n-x_n^\prime}\leq \modu{x_n}^{1+\delta -\epsilon}\modu x^{-\delta}(\norm u ^m + \modu x^ {N-k-1}\norm u)\norm{\phi - \psi}_0,
\end{equation*} 
and
\begin{equation*}
\norm{u_n- u_n^\prime} \leq \modu{x_n}^{\delta}\modu x ^{-\delta} (\norm u ^m +\modu x^{N-k-1}\norm u)\norm{\phi- \psi}_0.
\end{equation*}
\end{lemma}

\begin{proof}
We use the following notation: $\Delta x_n := \modu{x_n - x_n^\prime}$, $\Delta u_n := \normu{u_n - u_n^\prime}$, and $\Delta \phi(x,u):= (\norm u ^m + \modu x^{N-k-1}\norm u)\norm{\phi - \psi}_0$. We argue by induction over $n$. If $n=1$, thanks to \eqref{eq:eqprinc}, there exists $K>0$ such that
\begin{equation*}
\left\{\begin{array}{l l}
\Delta x_1 \leq K\modu x^{k+1} \Delta\phi(x,u), \\
\Delta u_1 \leq K \left( \modu x^{k+1}+ \modu x^k\norm u \right) \Delta\phi(x,u),
\end{array}\right.
\end{equation*}
for $\gamma$ and $s$ small enough. Let us assume that the inequalities hold for $n$, and we prove that they hold also for $n+1$. Since $x_n^k$ and $(x_n^\prime)^k$ are equivalent to $\frac{1}{n}$, we have $(x_n^\prime)^k = x_n^k + o(x_n^k)$. From the definition of $\Phi$ it follows
\begin{equation*}
\begin{split}
\Delta x_{n+1}&= \modu{x_n^k\left( 1- x_n^k + \oo{x_n^{2k}, \norm {u_n} x_n^k} \right) -(x_n^\prime)^k\left( 1- (x_n^\prime)^k + \oo{(x_n^\prime)^{2k}, \norm{u_n^\prime} (x_n^\prime)^k} \right) }\\
&\leq \Delta x_n \modu{1- x_n^k + o(x_n^k) } + K\modu{x_n}^{k+1} \Delta u_n + K \modu{x_n}^{k+1} \Delta\phi(x_n,u_n),
\end{split}
\end{equation*} 
and
\begin{equation*}
\begin{split}
\Delta u_{n+1}&\leq K \norm{u_n} \Delta x_n + \modu{ 1- (\alpha + \epsilon)x_n + o(x_n) }\Delta u_n + K \left( \modu{x_n}^{k+1}+\norm{u_n}\modu{x_n}^k \right)\Delta \phi(x_n,u_n).
\end{split}
\end{equation*}
Thanks to $\norm{u_n}\leq \norm u \modu{\frac{x_n}{x}}^{k\alpha}$, we have
\begin{equation*}
\Delta \phi(x_n,u_n)\leq \left( \norm u ^m +\modu x^{N-k-1} \norm u\right) \left( \modu{\frac{x_n}{x}}^{m k \alpha} +\modu{\frac{x_n}{x}}^{k\alpha +N -k-1} \right)\norm{\phi- \psi}_0,
\end{equation*}
and, since $\modu{\frac{x_n}{x}}^{\gamma}\leq \modu{\frac{x_n}{x}}^{\delta}$ when $\gamma > \delta$, we have
\begin{equation*}
\modu x^\delta \Delta \phi(x_n,u_n) \leq 2 \modu{x_n}^\delta \Delta \phi(x,u).
\end{equation*}
By inductive hypothesis, we may bound $\modu x^{\delta}\frac{\Delta x_{n+1}}{\Delta \phi}$ and $\modu x^{\delta}\frac{\Delta u_{n+1}}{\Delta \phi}$. We thus obtain
\begin{equation*}
\begin{split}
\modu x^{\delta}\frac{\Delta x_{n+1}}{\Delta \phi}&\leq \modu{ 1-\epsilon x_n^k +o(x_n^k) }\modu{x_{n+1}}^{1+\delta- \epsilon}+ K\modu{x_n}^{2+\delta}+ 2K\modu{x_n}^{2+\delta}\\
&\leq  \modu{ 1-\epsilon x_n^k +o(x_n^k) }\modu{x_{n+1}}^{1+\delta- \epsilon} \leq \modu{x_{n+1}}^{1+\delta- 2\epsilon},
\end{split}
\end{equation*}
and
\begin{equation*}
\begin{split}
\modu x^{\delta}\frac{\Delta u_{n+1}}{\Delta \phi}&\leq K\norm{u_n}\modu{x_n}^{1+\delta-\epsilon}+ \modu{ 1- \epsilon x_n^k +o(x_n^k) }\modu{x_{n+1}}^{\delta}+ K\left( \modu{x_n} +\norm{u_n} \right)\modu{x_n}^{1+\delta}\\
&\leq \modu{ 1- \epsilon x_n^k +o(x_n^k) }\modu{x_{n+1}}^{\delta} \leq \modu{x_{n+1}}^\delta.
\end{split}
\end{equation*}
Since $\norm{u_n}\modu{x_n}^{-k\alpha}=o(1)$, we can now prove the last inequality
\begin{equation*}
\norm{u_n}\modu{x_n}^{1+\delta-\epsilon} = o(\modu{x_n}^{1+k\alpha +\delta - \epsilon})= \modu{x_n}^\delta o (\modu{x_n}),
\end{equation*}
for $\epsilon$ small enough. 
\end{proof}

We can now estimate $S_2$ as follows.
\begin{equation*}
\begin{split}
S_2 
&\leq \sum_{n=0}^\infty K_1 \modu{\frac{x_n}{x}}^{-k\beta-\epsilon} \!\!\left( \norm{u_n}^m +\modu{x_n}^{N-k-1}\norm{u_n} \right)\! \Delta x_n + K_2 \modu{\frac{x_n}{x}}^{-k\beta-\epsilon}\!\!\left( \norm{u_n}^{m-1}\modu{x_n}^k +\modu{x_n}^{N-k-1} \right)\! \Delta u_n .
\end{split}
\end{equation*}
By Lemma \ref{lemma:lemma3.7}, Corollary \ref{cor:4.3} and the fact that $\norm{u_n}\modu{x_n}^{-\epsilon}=o(1)$, we thus obtain
\begin{equation*}
S_2\leq K \left( \norm u ^{m-1} +\modu{x_n}^{N-k-1} \right)\left( \norm u ^m +\modu x^{N-k-1}\norm u \right)\norm{\phi- \psi}_0.
\end{equation*}
Therefore $\T$ is a contraction.
\end{proof} 

Taking $\phi$ the unique fixed point of $T$, we can use the following change of coordinates: $\tilde v= v-\phi (x,u)$. Then
\begin{equation*}
\begin{split}
\tilde v_1&= v_1- \phi(x_1, u_1)=(I-x^k B)v+ H(x,u,v)- \phi(x_1,u_1)\\
&=(I-x^k B)(\tilde v + \phi(x,u))+ H(x,u,\tilde v + \phi(x,u))- \phi(x_1,u_1)\\
&=(I-x^k B)\tilde v + \underbrace{(I-x^k B)\phi(x,u)+ H(x,u,\phi(x,u))}_{=\phi(x_1^{\phi},u_1^{\phi})} + \sum_{n\geq 1} \left[\frac{1}{n!}\frac{\partial^n H}{\partial v^n}(x,u,\phi(x,u)) \tilde v ^n\right] - \phi(x_1,u_1)\\
&=(I-x^k B)\tilde v +\phi( x_1^{\phi},u_1^{\phi} )-\phi(x_1^{\phi},u_1^{\phi})+ \cdots \\
&=(I-x^k B)\tilde v + \tilde H(x,u,\tilde v),
\end{split}
\end{equation*}
with $\tilde H(x,u,\tilde v)= O(\|\tilde v\|)$, and hence $H(x,u,0)=0$. Therefore we can apply Theorem \ref{th:7.1} to $\Phi|_{\{\tilde v=0\}}$ and this concludes the proof of Theorem \ref{thm:1.6}.
\end{proof}

We then deduce the following reformulation of Corollary~3.8 of \cite{hakim2} .

\begin{corollary}
Let $\Phi\in\Diff(\C^p,0)$ be a tangent to the identity germ of order $k+1\ge 2$ and let $[V]$ be a non-degenerate characteristic direction. Let $\{\lambda_1,\dots,\lambda_h \}$ be the directors associated to $[V]$ with strictly positive real parts and assume that
\begin{equation*}
\alpha_1 > \alpha_2 > \cdots > \alpha_h >0, 
\end{equation*}
where $\alpha_j= \Re \lambda_j$. Then there exists an increasing sequence
\begin{equation*}
M_1 \subset M_2 \subset \cdots \subset M_h 
\end{equation*}
of parabolic manifolds, defined in a sector, attracted by the origin along the direction $[V]$. Moreover, for any $1\leq i \leq h$, the dimension of $M_i$ is $1 +\sum_{\Re \lambda_j\geq \alpha_i}m_{alg}(\lambda_j)$ and $M_i$ is tangent at the origin to $\C V \bigoplus_{\Re \lambda_j \geq \alpha_i} E_{\lambda_j}$, where $E_{\lambda_j}$ is the eigenspace associated to the eigenvalue $\lambda_j$.
\end{corollary}

We can also deduce a partial converse of Theorem \ref{th:7.1}, using the following result, which holds for germs of biholomorphisms and hence also for global biholomorphisms.

\begin{lemma}\label{lemma5.1}
Let $\Phi\in\Diff(\C^p, 0)$ be a tangent to the identity germ of order $k+1\ge 2$. If $X=(x,y)\in\C^p\setminus\{(0,0)\}$ is so that $X_n=\Phi^n(x,y)$ converges to the origin and $[X_n]$ converges to $[1:0]$, then there exist constants $\gamma$, $s$ and $\rho$ so that, for any $n>n_0$, with $n_0$ large enough, we have $x_n\neq 0$ and $X_n=(x_n,y_n)\in S_{\gamma,s, \rho}$, where, for $x\neq 0$ and $U=\frac{y}{x}$, we set
$$
S_{\gamma,s, \rho}=\left\{(x,U)\in\C\times\C^{p-1} \;\Bigm\vert\; \bigr\vert \Im x^k\bigr\vert \leq \gamma\, \Re x^k, \, \modu{x^k} < s, \, \norm U < \rho\right\}.
$$
\end{lemma}

\begin{proof}
Since $X_n=(x_n,y_n)$ converges to $0$ and $[X_n]$ converges to $[1:0]$, we have that $x_n$ is definitively different from $0$. Moreover $X_n$ definitively lies in $D_{s,\rho}:=\{(x,U) \mid \modu x\leq s,\, \norm U\leq \rho\}.$
Thanks to Proposition \ref{prop:2}, the first component of $\Phi$ is of the form $x_1=x-\frac{1}{k}x^{k+1}+ \oo{\norm U x^{k+1}, x^{2k+1}}$, and $x_n^k\sim \frac{1}{n}$. Therefore, for any $\gamma$ arbitrarily small and any $n$ large enough, we have $\modu{\Im x_n^k}\leq \gamma\, \Re x_n^k $, and hence $X_n$ definitively lies in $S_{\gamma,s, \rho}$.
\end{proof}

\begin{corollary}
Let $\Phi\in\Diff(\C^p,0)$ be a tangent to the identity germ of order $k+1\ge 2$, and let $[V]$ be a non-degenerate characteristic direction. If there exists an attracting domain $\Omega$ where all the orbits converge to the origin along $[V]$, then all the directors of $[V]$ have non-negative real parts. 
\end{corollary}

\begin{remark}
It is not true that if $[V]$ be a non-degenerate characteristic direction and there exists an attracting domain $\Omega$ where all the orbits converge to the origin along $[V]$, then all the directors of $[V]$ have strictly positive real parts. In fact, as shown by  Vivas in \cite{Liz2}, it is possible to find examples of germs having attracting domains along non-degenerate characteristic direction even when the directors have vanishing real parts. 
\end{remark}

\section{Fatou Coordinates}

We have the following analogous of Theorem~1.9 of \cite{hakim2}.

\begin{theorem}\label{thm:1.9}
Let $\Phi\in\Diff(\C^p, 0)$ be a tangent to the identity germ. Let $[V]$ be an attracting non-degenerate characteristic direction. Then there is an invariant domain $D$, with $0\in\partial D$, so that every point of $D$ is attracted to the origin along the direction $[V]$, and such that $\Phi|_D$ is holomorphically conjugated to the translation
$$
\left\{
\begin{array}{l l}
\displaystyle{1\over x_1} = {1\over x} + 1, \\
U_1= U,
\end{array}
\right.
$$
with $(x,U) \in \C\times\C^{p-1}$.
\end{theorem}

We may assume that $[V]=[1:0]$, and that its associated matrix $A$ is in Jordan normal form, with the non-zero elements out of the diagonal equal to $\eps>0$ small.

Let $\lambda_1,\dots, \lambda_h$ be the distinct eigenvalues of $A$, and up to reordering, we may assume that, setting $\alpha_j = \Re(\lambda_j)$, we have
$$
\alpha_1\ge \alpha_2\ge \cdots\ge \alpha_h>\alpha>0.
$$
Let $J_1,\dots, J_h$ be the Jordan blocks of $A$, where $J_l$ is the block relative to $\lambda_l$ for $1\le l\le h$, and let $u= (u^1,\dots, u^h)\in\C^{p-1}$ be the splitting of the coordinates of $\C^{p-1}$ associated of the splitting of $A$ in Jordan blocks.
Therefore we can write
\begin{equation}\label{eq:eqprinc9}
\Phi(x,u)= \left\{
\begin{array}{l l l l l}
x_1= x- \frac{1}{k}x^{k+1} +F(x,u),\\
u_1^1=(I^1- x^k J_1)u^1 + G^1(x,u),\\
u_1^2=(I^2- x^k J_2)u^2 + G^2(x,u),\\
\quad\,\,\,\vdots\\
u_1^h=(I^h- x^k J_h)u^h + G^h(x,u),
\end{array}\right.
\end{equation}
where $I^l$ is an identity matrix of same dimension of the block $J_l$ for $1\le l\le h$, and
$$
G^j(x,u) = O(\|u\|^2 x^k, \|u\| x^{k+1}|\log x|).
$$

Set
$$
u^{\le j} := (u^1, \dots, u^j) \quad\hbox{and}\quad u^{> j} := (u^{j+1}, \dots, u^h),
$$
and analogous definitions for $u^{< j}$ and $u^{\ge j}$.

Given $N\in\N$ with $N\ge k+1$, thanks to \eqref{eq:acca0}, for every $1\le j\le h$, we can write
\begin{equation}\label{4.5}
G^j(x,u) = P_N^j(x,u) + \oo{\norm{u} \modu x^{N} \modu{\log\modu x}^{h_N}},
\end{equation}
with
$$
P_N^j(x,u)= \sum_{k\leq s\leq N,\, t\in E_s} c^j_{s,t}(u)x^s (\log x)^t, 
$$
for some $h_N\in \N$ depending on $N$, where for any $s$ we defined $E_s$ as the (finite) set of integers $t$ so that the series above contains the term $x^s(\log x)^t$, and where $c_{s,t}(u)$ are holomorphic in $\|u\|\le \rho$ and $c_{s,t}(0)\equiv 0$.  

The following result is the analogous of Proposition~4.1 of \cite{hakim2}.

\begin{proposition}
Let $\Phi\in\Diff(\C^p,0)$ be a tangent to the identity germ of order $k+1\ge 2$ as in \eqref{eq:eqprinc9}, with $[V]=[1:0]$ attracting non-degenerate characteristic direction. 
For any positive integers $N\ge k+1$ and $m$, there exist local holomorphic coordinates (defined in a sector) such that  
\eqref{4.5} holds, and moreover
\begin{equation}\label{4.6}
P_N^j(x,(0,u^{>j})) = O(x^k\|u^{>j}\|^m)
\end{equation}
for $1\le j\le h$.
\end{proposition}

\begin{proof}
We want to change coordinates holomorphically to delete in $P_N^j(x,(0,u^{>j}))$ the terms in $u^{>j}$ with degree less than $m$. We use holomorphic changes of coordinates of the form $\tilde u^j = u^j - q_j(x, u^{>j})$, where the $q_j$'s are polynomials in $x$, $\log x$ and $u^{>j}$ with $q_j(x, 0) = 0$; if we obtain \eqref{4.6} for $j=1,\dots, j_0$, then changing the variables $u^j$ for $j>j_0$ will provide no effect on the first $j_0$ variables. We shall then perform the construction by induction on $j$, considering only changes on $u^{\ge j}$ with $u^{<j} = 0$, which allow us to forget about the first $j-1$ coordinates.

Let $v= u^{>j}$, and let us consider the matrix $B_j$ defined as
$$
B_j= Diag(A_{j+1},\dots, A_h).
$$
We now have to prove a statement similar to the one in Proposition \ref{prop:ora} but with the opposite notation, i.e., we look for changes of coordinates of the form $\tilde u = u - \phe(x,v)$ such that
$$
G^j(x, 0, v) = O(|x|^k\|v\|^m + |x|^N\|v\|),
$$
and hence the r\^oles of $A(=J_j)$ and $B(=B_j)$ are here exchanged. 

Note that, if $1,\lambda_1,\dots,\lambda_h$ are rationally independent, then we can prove the statement exactly as in the proof of Proposition \ref{prop:ora}. 

Otherwise, let $Q(v)x^s(\log x)^t$ be the lower degree term in $P_N^j(x, (0, v))$, with $Q(v)$ homogeneous polynomial of degree $d$ in $v$. The change of coordinates
$$
\tilde u^j = u^j - x^{s-k} (\log x)^t P(v),
$$
deletes the term $Q(v)x^s(\log x)^t$ if $P$ solves
$$
P((I_q-x^kB)v) - (I_r - x^k J_j)P(v) - \frac{s-k}{k} x^k P(v) = x^k Q(v) + O(\|v\|x^{k+1}),
$$ 
where $r$ and $q$ are the dimensions, respectively, of $u^j$ and $v$. Therefore, by decreasing induction on the indices of the components of $u^j$, we may reduce ourselves to solve, for the $r$ components $P_i$ of $P$, equations of the form
\begin{equation}
\frac{\partial P_i}{\partial v_1}(Bv)_1 + \cdots+ \frac{\partial P_i}{\partial v_q}(Bv)_q - \left(\lambda_j -\frac{s-k}{k}\right) P_i = \tilde Q_i(v),  
\end{equation} 
that is, by Euler formula, of the form
\begin{equation}\label{4.9}
\frac{\partial P_i}{\partial v_1}(Cv)_1 + \cdots+ \frac{\partial P_i}{\partial v_q}(Cv)_q  = \tilde Q_i(v),  
\end{equation} 
where $C = B - \frac{\lambda_j -{(s-k)}/{k}}{d} I_q$. We solve these equations component by component, by comparing the coefficients of the monomials $v^T$, where $T\in\N^q$ in both sides. For any $T:=(t_1,\dots, t_q)\in\N^q$, we define the weight of $T$ as
$$
w(T)= t_1+ 2t_2+ \cdots + qt_q.
$$
If $P_i(v) = a v^T$ and $\tilde Q_i(v)= c v^T$, equation \eqref{4.9} is reduced, modulo terms of greater weight, to
$$
a(\nu_1t_1+\cdots+ \nu_q t_q) = c,
$$
where $\nu_1,\dots, \nu_q$ are the eigenvalues of $C$. Now, if $\nu_1t_1+\cdots+ \nu_q t_q\ne 0$, then we can solve the equation; otherwise, we can consider the change of coordinates
$$
\tilde u^j = u^j - a v^T x^{s-k} (\log x)^{t+1},
$$
under which the terms in $v^T x^{s-k}(\log x)^{t+1}$ and in $v^T x^{s}(\log x)^{t+1}(\log x)^{t+1}$ in the left-hand side vanish and the equation is reduced to $a(t+1) = -c$; this change introduces new terms in $x^s(\log x)^{t+1}$, but it can happen only finitely many times, and hence it is not a problem. Iterating this procedure on the weight of $T$, given a degree $d$, we have to solve the case of $T$ of maximal weight, i.e., $v^T + v_q^d$. In this case, the equation is simply $a\nu_q d = c$, and it is solvable if $\nu_q\ne 0$; if $\nu_q = 0$, as before, we can consider the change 
$$
\tilde u^j = u^j - a v_q^d x^{s-k} (\log x)^{t+1},
$$
and we are done.
\end{proof}

We can then deduce the following reformulations of Corollaries~4.2,~4.3, and~4.4 of \cite{hakim2} for the case $k+1\ge 2$.

\begin{corollary}
Let $\Phi\in\Diff(\C^p,0)$ be a tangent to the identity germ of order $k+1\ge 2$ as in \eqref{eq:eqprinc9}, with $[V]=[1:0]$ attracting non-degenerate characteristic direction. 
For any positive integers $N\ge k+1$ and $m$, there exist local holomorphic coordinates such that  
\begin{equation}\label{4.10}
G^{\le j}\left(x, (0, u^{>j})\right) = O\left(|x|^k\|u^{>j}\|^m + |x|^N\|u^{>j}\|\right),
\end{equation}
for $1\le j\le h$.
\end{corollary}

\begin{corollary}\label{co:4.3-4}
Let $\Phi\in\Diff(\C^p,0)$ be a tangent to the identity germ of order $k+1\ge 2$ as in \eqref{eq:eqprinc9}, with $[V]=[1:0]$ attracting non-degenerate characteristic direction. 
Let $0<\eps<\alpha$, and assume that the local coordinates are chosen so that the non-zero coefficients out of the diagonal in $A$ are equal to $\eps_0>0$ small enough, and \eqref{4.10} is satisfied with $m$ and $N$ such that
\begin{equation}\label{4.11}
m\alpha_h -  \alpha_1 \ge 1\quad\hbox{and}\quad N +k( \alpha_h -\alpha_1) \ge k+1.
\end{equation}
Then, for every $j$ and for each $(x,u)\in S_{\gamma, s, \rho}$ with $\gamma, s, \rho$ small enough, there exists a constant $K>0$ such that
\begin{equation}\label{4.12}
\|u_n^{\le j}\|\le |x_n|^{k(\alpha_j - \eps)}(\|u^{\le j}\||x|^{-k(\alpha_j - \eps)} + K|x|^k),
\end{equation}
and moreover, $\|u_n^{\le j}\||x_n|^{-k(\alpha_j - \eps)}$ converges to zero as $n$ tends to infinity.
\end{corollary}

\begin{proof}
From the proof of Theorem \ref{th:7.1}, we know that, taking $\alpha= \alpha_h-\eps$ we have
$$
\|u_n\|\le |x_n|^{k\alpha}\|u\| |x|^{-k\alpha}.
$$
Hence, from \eqref{4.10} and \eqref{4.11} we obtain
$$
|x_{n+1}|^{-k(\alpha_j-\eps)}\left\|G^{\le j}\left(x, (0, u^{>j})\right) \right\|\le K_1|x_n|^{k+1}.
$$
Arguing as in the proof of Theorem \ref{th:7.1}, choosing $\eps$ and $\eps_0$ small enough, since the eigenvalues of $J_j - k(\alpha_j-\eps)I^j$ have positive real parts, for each $x\in S_{\gamma, s, \rho}$ with $\gamma, s, \rho$ small enough, we have
$$
\|I^j - (J_j - k(\alpha_j-\eps)I^j)x^k + o(x^k)\|\le 1.
$$
We have
$$
\begin{aligned}
|x_{n+1}|^{-k(\alpha_j-\eps)}\|u_{n+1}^{\le j}\|
&\le \|I^j - (J_j - k(\alpha_j-\eps)I^j)x^k + o(x^k)\| |x_{n}|^{-k(\alpha_j-\eps)}\|u_{n}^{\le j}\|\\
&\quad\;+ |x_{n+1}|^{-k(\alpha_j-\eps)}\|G^{\le j}(x, (0, u^{>j}))\|.
\end{aligned}
$$
Hence, setting $V_n^{\le j} := |x_n|^{-k(\alpha_j-\eps)}\|u_n^{\le j}\|$, we obtain
$$
V_{n+1}^{\le j}\le V_n^{\le j} + K_1|x_n|^{k+1}.
$$
Since for any $(x,u)\in S_{\gamma, s, \rho}$, for $\gamma, s, \rho$ small enough, there exists $0<c<1$ such that $|x_{n+1}|^k\le |x_n|^k (1-c |x_n|)$, we have
$$
|x_n|^{k+1}\le \frac{|x_n|^k- |x_{n+1}|^k}{c},
$$
implying that there exists a positive constant $K>0$ such that
$$
V_{n+1}^{\le j} + K|x_{n+1}|^k \le V_n^{\le j} + K|x_n|^k\le \cdots\le V^{\le j}+ K|x|^{k},
$$ 
proving \eqref{4.12}. Moreover, this proves that there exists $\eps_0<<\eps_1<\eps$ such that $\|u_n^{\le j}\| |x_n|^{-k(\alpha_j - \eps)} \le \eps_1$, and so $\|u_n^{\le j}\| |x_n|^{-k(\alpha_j - \eps)}$ converges to zero as $n\to+\infty$.
\end{proof}

Thanks to the last corollary, we may assume without loss of generality that our germ $\Phi$ is of the form \eqref{eq:eqprinc9} and in the hypotheses of Corollary \ref{co:4.3-4}. Define the set
\begin{equation}\label{4.13}
\6 D: = \{(x,u)\in S_{\gamma, s, \rho} : \|u_n^{\le j}\| |x_n|^{-k(\alpha_j - \eps)} \le 1\}.
\end{equation}

To prove Theorem \ref{thm:1.9}, we shall need this reformulation of Lemma~4.5 of \cite{hakim2} .

\begin{lemma}
Let $\Phi\in\Diff(\C^p, 0)$ be a tangent to the identity germ. Let $[V]$ be an attracting non-degenerate characteristic direction. Consider local holomorphic coordinates where $\Phi$ satisfies the hypotheses of Corollary \ref{co:4.3-4} with $\eps$ such that $3\eps<\min(\alpha_1, 1)$. Then the sequence $\{x_n^{-kA} u_n\}$ converges normally on the set $\6D$ defined by \eqref{4.13}.
\end{lemma}

\begin{proof}
Given $(x,u)\in \6D$, we shall bound $\|x_{n+1}^{-kJ_j}u_{n+1}^j - x_{n}^{-kJ_j}u_{n}^j\|$, for each $j=1,\dots, h$. We have
\begin{equation*}
\begin{aligned}
\|x_{n+1}^{-kJ_j}u_{n+1}^j - x_{n}^{-kJ_j}u_{n}^j\|
&\le K\|x_{n}^{-kJ_j}\| \|u_{n}^{\le j}\|\left( \|u_{n}\||x_{n}|^{k} + |x_{n}|^{k+1}|\log x_n|^q\right)\\
&\quad + K'\|x_{n}^{-kJ_j}\|\|G^j(x, (0, u^{>j}))\|,
\end{aligned}
\end{equation*}
for some positive integer $q$. Since $(x,u)\in\6D$, we have $\|u_n^{\le j}\| \le  |x_n|^{k(\alpha_j - \eps)}$, and hence, using the inequality $\|x_{n}^{-kJ_j}\|\le |x_n|^{-k\alpha_j -k \eps}$, we obtain
\begin{equation*}
\begin{aligned}
\|x_{n+1}^{-kJ_j}u_{n+1}^j - x_{n}^{-kJ_j}u_{n}^j\|
&\le
K_1|x_{n}|^{-k\alpha_j-k\eps} \left(|x_{n}|^{k(1+ \alpha+\alpha_j -2\eps)} + |x_{n}|^{k+1+k\alpha_j - k\eps}|\log x_n|^q\right) + K_2|x_{n}|^{k+1},
\end{aligned}
\end{equation*}
and hence there exists $K>0$ such that
$$
\|x_{n+1}^{-kJ_j}u_{n+1}^j - x_{n}^{-kJ_j}u_{n}^j\|
\le
K\left(|x_{n}|^{k(1+\alpha_1 -3\eps)} + |x_{n}|^{k+1-2 k\eps}|\log x_n|^q\right),
$$
and hence we are done.
\end{proof}

\sm We now have all the ingredients to prove Theorem \ref{thm:1.9}.

\begin{proof} {\sl of Theorem \ref{thm:1.9}.}
Thanks to the previous lemma, we can define in $\6D$ the following holomorphic bounded map
\begin{equation}\label{4.14}
H(x,u) := \sum_{n=0}^\infty \left(x_{n+1}^{-k A}u_{n+1} - x_{n}^{-k A}u_{n}\right),
\end{equation}
which satisfies
\begin{equation}\label{4.15}
\|H(x,u)\| 
\le
K \left(|x|^{k(\alpha_1 -3\eps)} + |x_{n}|^{k-2 k\eps}|\log x_n|^q\right)\le K \left(|x|^{k(\alpha_1 -3\eps)} + |x_{n}|^{k- 3k\eps}\right).
\end{equation}
Therefore, the holomorphic map
\begin{equation}\label{4.16}
U(x,u) := x^{-k A} u + H(x,u) = \lim_{n\to+\infty}x_n^{-k A} u_n 
\end{equation}
is invariant. The main term near to the origin is $x^{-k A} u$, and the level sets $\{U(x,u) = c\}$ with $c\in\C$ are complex invariant analytic curves. 
Therefore, taking $(x, U)$ as new coordinates, $\Phi$ becomes
\begin{equation}
\left\{\begin{array}{l l}
x_1= x - \frac{1}{k}x^{k+1} + \tilde F(x, U), \\
U_1= U,
\end{array}\right. 
\end{equation}
where $\tilde F$ is a holomorphic function of order at least $k+2$ in $x$, and $U$ behaves as a parameter. We can thus argue as in Fatou \cite{Fa}, and change coordinates, in $\6D$, in the first coordinate $x$, with a change depending on $U$, to obtain $\Phi$ of the form
\begin{equation*}
\left\{\begin{array}{l l}
\displaystyle\frac{1}{z_1}= \frac{1}{z} + 1, \\
U_1= U,
\end{array}\right. 
\end{equation*}
and this concludes the proof.
\end{proof}

We thus deduce the following reformulation of Corollary~4.6 of \cite{hakim2} .

\begin{corollary}
Let $\Phi\in\Diff(\C^p,0)$ be a tangent to the identity germ of order $k+1\ge 2$ and let $[V]$ be a non-degenerate characteristic direction. Assume that $[V]$ has exactly $d$ (counted with multiplicity) directors with positive real parts. Let $M$ be the parabolic manifold of dimension $d+1$ provided by Theorem \ref{thm:1.6}. Then there exist local holomorphic coordinates $(x,u, v)$ such that $M= \{v=0\}$, and $\Phi|_M$ is holomorphically conjugated to:
\begin{equation*}
\left\{\begin{array}{l l}
\displaystyle\frac{1}{z_1}= \frac{1}{z} + 1, \\
U_1= U.
\end{array}\right. 
\end{equation*}
\end{corollary}

\begin{proof}
Thanks to Theorem \ref{thm:1.6} there exist local holomorphic coordinates $(x,u, v)$ defined in a sector $S_{\gamma,s, \rho}$ such that the parabolic manifold $M$ is defined by $M= \{v=0\}$, and $\Phi$ is defined by \eqref{eq:eqprinc} with $F$, $G$, and $H$ satisfying \eqref{eq:acca}, and $H(x,u, 0)= 0$. Then $\Phi|_M$ is given by
$$
\left\{
\begin{array}{l l}
x_1= x-\frac{1}{k} x^{k+1} +F(x,u,0),\\
u_1=(I_d- x^kA)u + G(x,u,0),
\end{array}\right. 
$$
where all the eigenvalues of $A$ have positive real parts. Let $\lambda_1,\dots,\lambda_h$ be the distinct eigenvalues of $A$, and let $\alpha_j=\Re(\lambda_j)$. Up to reordering, we may assume  $\alpha_1>\cdots >\alpha_h>\alpha>0$. 

Let $m$ and $N\ge k+1$ be positive integers such that $m\alpha_h -  \alpha_1 \ge 1$ and $N +k(\alpha_h -\alpha_1) \ge k+1$. We can thus write the Taylor expansion of $G$ as
$$
G(x,u,0) = \sum_{1\le s\le N\atop t\in E_s} c_{s,t}(u)x^s(\log x)^t + O(|x|^k\|u\|^m + |x|^N\|u\|),
$$
where $c_{s,t}(u)$ is a polynomial and $\deg(c_{s,t}(u))\le m$. Therefore we can apply Theorem \ref{thm:1.9} to $\Phi(x,u,0)$ and we are done.
\end{proof}

\section{Fatou-Bieberbach domains}

In this section we shall assume that $\Phi$ is a global biholomorphism of $\C^p$ fixing the origin and tangent to the identity of order $k+1\ge 2$.

\begin{definition}
Let $\Phi$ be a global biholomorphism of $\C^p$ fixing the origin and tangent to the identity of order $k+1\ge 2$. Let $[V]$ be a non-degenerate characteristic direction of $\Phi$ at $0$. The {\it attractive basin to $(0, [V])$} is the set
\begin{equation}\label{5.1}
\Omega_{(0, [V])}:= \{X\in\C^p\setminus\{0\} : \Phi^n(X)\to 0, [\Phi^n(X)]\to [V]\}.
\end{equation}
\end{definition}

We shall study the attractive basin $\Omega_{(0, [V])}$ when some of the directors of $[V]$ have positive real parts. 

We can assume that, writing $X=(x,y)\in\C\times\C^{p-1}$, $[V] = [1:0]$ and $\Phi$ is of the form
$$
\left\{
\begin{array}{l l}
x_1= x+ p_{k+1}(x, y) +p_{k+2} (x, y) + \cdots,\\
y_1= y + q_{k+1}(x, y) +q_{k+2} (x, y) + \cdots,
\end{array}\right. 
$$
with $ p_{k+1}(1, 0) = -1/k$ and $q_{k+1}(1,0) = 0$.

Thanks to Lemma \ref{lemma5.1}, we have
\begin{equation}\label{5.2}
\Omega_{(0, [V])}=\bigcup_{n\ge 0} \Phi^{-n}\left( \Omega_{(0, [V])}\cap S_{\gamma,s, \rho}\right),
\end{equation}
and we can restrict ourselves to study $\Omega_{(0, [V])}\cap S_{\gamma,s, \rho}$.

Since $S_{\gamma,s, \rho}\cap\{x=0\}=\emptyset$, we can use the blow-up $y = xu$ and we can assume that, in the sector, $\Phi$ has the form
\begin{equation}\label{5.3}
\left\{
\begin{array}{l l}
x_1= x-\frac{1}{k} x^{k+1} +O(\|u\|x^{k+1},x^{k+2}),\\
u_1=(I_{p-1}- x^k A)u + O(\|u\|x^{k},\|u\|x^{k+1}),
\end{array}\right. 
\end{equation}
where $A=A([V])$ is the matrix associated to $[V]$, and we can perform all the changes of coordinates used to prove Theorem \ref{thm:1.6} and Theorem \ref{thm:1.9}.

\sm We thus can prove the following generalization of Theorem~5.2 of \cite{hakim2} for the case $k+1\ge 2$.

\begin{theorem}\label{thm:5.2}
Let $\Phi$ be a global biholomorphism of $\C^p$ fixing the origin and tangent to the identity of order $k+1\ge 2$ and let $[V]$ be a non-degenerate characteristic direction of $\Phi$ at $0$. If $[V]$ is attracting, then the attractive basin $\Omega_{(0, [V])}\subset\C^p$ is a domain isomorphic to $\C^p$, i.e., it is a Fatou-Bieberbach domain.
\end{theorem}

\begin{proof}
We can reduce ourselves to consider $\Phi$ as in \eqref{5.3}, with $A$ in Jordan normal form. Let $\lambda_1,\dots,\lambda_h$ be the distinct eigenvalues of $A$, and let $\alpha_j=\Re(\lambda_j)$. Up to reordering, we may assume  $\alpha_1>\cdots >\alpha_h>\alpha>0$. 
Let $\eps>0$ be small and such that
$$
\alpha_1>\alpha_1-\eps>\alpha_2>\alpha_2-\eps>\cdots >\alpha_h>\alpha_h-\eps>0.
$$
Thanks to Theorem \ref{thm:1.9} and Corollary \ref{co:4.3-4}, the coordinates $u= (u^1, \dots, u^h)$ adapted to the structure in blocks of $A$ can be chosen such that, for $n$ large enough, we have
\begin{equation}\label{5.4}
\|u_n^j\|\le |x_n|^{k(\alpha_j -\eps)},  
\end{equation}
and we know that on 
\begin{equation}\label{5.5}
\6D = \{(x, u)\in S_{\gamma, s, \rho}: \|u_n^j\|\le |x_n|^{k(\alpha_j -\eps)},~\hbox{for}~j=1,\dots, h\},  
\end{equation}
we can conjugate holomorphically $\Phi$ to the translation
\begin{equation*}
\left\{\begin{array}{l l}
\displaystyle\frac{1}{z_1}= \frac{1}{z} + 1, \\
U_1= U,
\end{array}\right. 
\end{equation*}
with a change of the form $(z(x,u), U(x,u))$ such that
\begin{equation}\label{5.6}
U(x,u) = x^{-kA} u + O(x^\eta),
\end{equation}
for some positive $\eta$, and $z(x,u)\sim x^k$ as $x\to0$.

Let $\psi\colon \6D\to\C^p$ be defined by
$$
\psi(x, u) := (Z(x,u),U(x,u)) = \left(\frac{1}{z(x,u)}\,,U(x,u)\right),
$$
and let $\tau\colon\C^p\to\C^p$ be the translation $\tau(Z, U) := (Z+1, U)$. We know that $\6D$ is $\Phi$-invariant
\begin{equation}\label{5.7}
\tau\circ\psi = \psi\circ\Phi.
\end{equation}
Let us consider $W:=\psi(\6D)$. For $\gamma$ small enough, and $R>0$ big enough, the projection $Z(W)$ of $W$ on $\C$ contains the set
\begin{equation}\label{5.8}
\Sigma_{\gamma,R}:= \{Z\in\C : |\Im Z|<\gamma \Re Z, |Z|>R\}.
\end{equation}
For any fixed $Z\in\C$ and $r>0$, consider the generalized polydisc
$$
P_{(Z,r)} : = \{(Z,U)\in\C^p: \|U^j\|\le r~\hbox{for}~j=1,\dots, h\}.
$$
The definition \eqref{5.5} of $\6D$, and the form \eqref{5.6} of $U(x,u)$ imply that for $Z\in\Sigma_{\gamma,R}$ and $R$ big enough, $W$ contains the generalized polydisc $P_{(Z, |Z|^{\eps/2})}$. For $|Z|$ tending to infinity, the fiber of $W$ above $Z$ contains generalized polydisc $P_{(Z,r)}$ of radius arbitrarily large. Hence we have
\begin{equation}\label{5.9}
\bigcup_{n\ge 0} \tau^{-n}(W) = \C^p.
\end{equation}

The end of the argument is then as in Fatou \cite{Fa2, Fa3}, as follows.

Since $\6D\subset \Omega_{(0,[V])}$, and, thanks to \eqref{5.4}, for $n$ large enough, for every $X\in\Omega_{(0,[V])}$, $X_n\in\6D$, we also have
\begin{equation}\label{5.9}
\Omega_{(0,[V])} = \bigcup_{n\ge 0} \Phi^{-n}(\6D).
\end{equation}
Therefore, we can extend the isomorphism $\psi\colon\6D \to W$, to 
$$
\tilde\psi\colon\Omega_{(0,[V])}\to\C^p
$$
as follows: given $X\in\Omega_{(0,[V])}$, consider $n_0$ such that $\Phi^{n_0}(X)\in\6D$, and define
$$
\tilde\psi(X): = \tau^{-n_0} \circ \psi\circ \Phi^{n_0}(X). 
$$
Thanks to \eqref{5.7}, the definition does not depend on $n$. It is immediate to check that $\tilde\psi$ is injective, whereas its surjectivity follows from \eqref{5.9}.
\end{proof}

This last result is the generalization of Theorems~1.10 and~1.11 of \cite{hakim2} for the case $k+1\ge 2$.

\begin{theorem}\label{thm:1.10}
Let $\Phi\in\Diff(\C^p, 0)$ be a tangent to the identity germ. Let $[V]$ be a non-degenerate characteristic direction, and assume it has exactly $d$ directors, counted with multiplicities, with strictly positive real parts, greater than $\alpha>0$. Then
\begin{enumerate}
\item if the remaining directors have strictly negative real parts, the attractive basin $\Omega_{(0,[V])}$ is biholomorphic to $\C^{d+1}$;
\item otherwise, considering coordinates such that $[V] = [1:0]$, the set
$$
\widetilde\Omega_{(0,[V])}:= \{X\in\Omega_{(0,[V])}: \lim_{n\to +\infty} X_n^{-k\alpha} u_n = 0\}
$$ 
is biholomorphic to $\C^{d+1}$, and moreover its definition does not depend on $\alpha$.
\end{enumerate}
\end{theorem}

\begin{proof}
Thanks to the previous results we can apply Lemma \ref{lemma5.1} and property \eqref{5.2}. We can thus choose local holomorphic coordinates in a sector, such that, after the blow-up, $\Phi$ has the form
$$
\left\{
\begin{array}{l l l}
x_1=f(x,u,v)=x- \frac{1}{k}x^{k+1} +F(x,u,v),\\
u_1=g(x,u,v)=(I_{d}- x^kA)u + G(x,u,v),\\
v_1=h(x,u,v)=(I_{p-d-1}- x^kB)v + H(x,u,v),
\end{array}\right.
$$
where $A$, and $B$ are in Jordan normal form, $A$ has eigenvalues with strictly positive real parts, $B$ has eigenvalues with non-positive real parts, and $F$, $G$, and $H$ satisfying \eqref{eq:acca}. Moreover, thanks to Theorem \ref{thm:1.6}, we may assume $H(x,u,0)=0$.

If $X\in\Omega_{(0,[V])}$, for $\gamma, s, \rho$ arbitrarily small positive numbers, then $X_n\in S_{\gamma, s,\rho}$, for $n$ big enough. 

Assume that $B$ has only eigenvalues with strictly negative real parts. Therefore, thanks to the previous equations, we have $\|v_{n+1}\|> \|v_n\|$ for $n$ big enough, so $v_n$ cannot converge to $0$ unless we have $v_n=0$. Hence
$$
\Omega_{(0,[V])}\cap S_{\gamma, s,\rho} \subset \{v=0\},
$$
and we can apply the same argument as in Theorem \ref{thm:5.2} to $\Phi|_{S_{\gamma, s,\rho} \cap \{v=0\}}$.

If $B$ has eigenvalues with non-positive real parts, since in $\widetilde\Omega_{(0,[V])}$, for $n$ big enough, we have $\|x_{n+1}^{-k\alpha} v_{n+1}\|>\|x_{n}^{-k\alpha} v_{n}\|$, we cannot have $x_{n}^{-k\alpha} v_{n}$ converging to $0$ unless $v_n=0$. Therefore we argue as before, but considering $\widetilde\Omega_{(0,[V])}$.

\end{proof}

\bibliographystyle{alpha}

\end{document}